\documentclass{amsart}

\usepackage{amssymb}
\usepackage{amsthm}
\usepackage[curve,matrix,arrow]{xy}
\theoremstyle{plain}\newtheorem{Theorem}{Theorem}[section]
\theoremstyle{plain}
\theoremstyle{plain}\newtheorem{Corollary}[Theorem]{Corollary}
\theoremstyle{plain}\newtheorem{Lemma}[Theorem]{Lemma}
\theoremstyle{plain}\newtheorem{Proposition}[Theorem]{Proposition}
\theoremstyle{definition}
\theoremstyle{definition}
\theoremstyle{definition}
\theoremstyle{definition}\newtheorem{Remark}[Theorem]{Remark}

    \def\OG{{\mathcal{O}G}}  \def\OGb{{\mathcal{O}Gb}}
    \def\OH{{\mathcal{O}H}}  \def\OHc{{\mathcal{O}Hc}}
    \def\OP{{\mathcal{O}P}}
    
\def\CE{{\mathcal{E}}}    
    
    \def\OT{{\mathcal{O}T}}

\def\CO{{\mathcal{O}}}

\def\F{{\mathbb F}}

\def\Q{{\mathbb Q}}
\def\Z{{\mathbb Z}}

\def\Gb{{\bf G}} 
\def\Cb{{\bf C}} 
\def\Lb{{\bf L}} 
 
\def\Zb{{\bf Z}}

\def\Xb{{\bf X}}
\def\Yb{{\bf Y}}

\def\Aut{\mathrm{Aut}}           \def\tenk{\otimes_k}     
\def\Br{\mathrm{Br}}             \def\ten{\otimes}

\def\End{\mathrm{End}}

\def\Id{\mathrm{Id}}             \def\tenA{\otimes_A}
             \def\tenB{\otimes_B}
\def\Ind{\mathrm{Ind}}           

\def\Inndiag{\mathrm{Inndiag}}
\def\Irr{\mathrm{Irr}}           
           \def\tenO{\otimes_{\mathcal{O}}}
\def\mod{\mathrm{mod}}
           \def\tenOP{\otimes_{\mathcal{O}P}}
         \def\tenOQ{\otimes_{\mathcal{O}Q}}

\newcommand{\POmega}{\operatorname{P\Omega}}
\newcommand{\PSp}{\operatorname{PSp}}
\newcommand{\PSL}{\operatorname{PSL}}
\newcommand{\SL}{\operatorname{SL}}
\newcommand{\PGL}{\operatorname{PGL}}
\newcommand{\GL}{\operatorname{GL}}
\newcommand{\PSU}{\operatorname{PSU}}
\newcommand{\SU}{\operatorname{SU}}
\newcommand{\PGU}{\operatorname{PGU}}
\newcommand{\GU}{\operatorname{GU}}
\newcommand{\Sp}{\operatorname{Sp}}

\newcommand{\SO}{\operatorname{SO}}

\newcommand{\GO}{\operatorname{O}}

\title{Conjectures of Alperin and Brou\'e for $2$-blocks
with elementary abelian defect groups of order $8$}

\author{Radha Kessar} 
\address{ University of Aberdeen, Institute of Mathematics,
          Fraser Noble Building, Aberdeen, AB24 3UE, United Kingdom
        }
\author{Shigeo Koshitani} 
\address{ Department of Mathematics, Graduate School of Science,
          Chiba University, 1-33 Yayoi-cho, Inage-ku, Chiba,
          263-8522, Japan
        }
\author{Markus Linckelmann} 
\address{ University of Aberdeen, Institute of Mathematics,
          Fraser Noble Building, Aberdeen, AB24 3UE, United Kingdom
        }
\date{}

\begin{document}

\begin{abstract}
Using the classification of finite simple groups, we prove 
Alperin's weight conjecture and the character theoretic version 
of Brou\'e's abelian defect conjecture for $2$-blocks of finite 
groups with an elementary abelian defect group of order $8$.
\end{abstract}

\maketitle

\section{Introduction}

Throughout this paper $p$ is a prime and $\CO$ a complete discrete
valuation ring having an algebraically closed residue field $k$ of
characteristic $p$ and a quotient field $K$ of characteristic $0$, which
is always assumed to be large enough for the finite groups under
consideration.  For $G$ a finite group, a {\it block of} $\OG$ or $kG$ is
a primitive idempotent $b$ in $Z(\OG)$ or $Z(kG)$, respectively. The
canonical map $\OG\to$ $kG$ induces a defect group preserving bijection 
between the sets of blocks of $\OG$ and $kG$. By Brauer's First Main
Theorem there is a canonical bijection between the set of blocks of
$kG$ with a fixed defect group $P$ and the set of blocks of $kN_G(P)$ with
$P$ as a defect group. Alperin's weight conjecture predicts that the
number $\ell(b)$ of isomorphism classes of simple $kGb$-modules 
is an invariant of the local structure of $b$. For $b$ a block of $kG$ 
with an abelian defect group $P$, denoting by $c$ the block of 
$kN_G(P)$ corresponding to $b$,  Alperin's weight conjecture holds 
if and only if the block algebras of $b$ and $c$ have the same number 
of isomorphism classes of simple modules, or equivalently, the same 
number of ordinary irreducible characters. Thus, for blocks with 
abelian defect groups,  Alperin's weight conjecture would be 
implied by any of the versions of Brou\'e's Abelian Defect Conjecture, 
predicting that there should be a perfect isometry, or isotypy, or even 
a (splendid) derived equivalence between the block algebras, over 
$\CO$, of $b$ and $c$. 
Alperin announced the weight conjecture in \cite{Alperin87}.
At that time, the conjecture was known to hold for all
blocks of finite groups with cyclic defect groups (by work
of Brauer and Dade), dihedral, generalised quaternion, semidihedral
defect groups (by work of Brauer and Olsson), and all defect 
groups admitting only the trivial fusion system (by the work
of Brou\'e and Puig on nilpotent blocks). Since then many 
authors have contributed to proving Alperin's weight conjecture 
for various classes of finite groups - such as finite 
$p$-solvable groups (Okuyama), finite groups of Lie type in
defining characteristic (Cabanes), symmetric and general linear 
groups (Alperin, Fong, An) and some sporadic simple groups (An). 

\begin{Theorem} \label{c2c2c2}
Suppose $p=2$.
Let $G$ be a finite group and let $b$ be a block of $kG$ with
an elementary abelian defect group $P$ of order $8$. Denote
by $c$ the block of $kN_G(P)$ corresponding to $b$. Then
$b$ and $c$ have eight ordinary  irreducible characters and
there is an isotypy between $b$ and $c$; in particular, Alperin's
weight conjecture holds for all blocks of finite groups with
an elementary abelian defect group $P$ of order $8$. 
\end{Theorem} 

For principal blocks, Theorem \ref{c2c2c2} follows from work of 
Landrock \cite{Landrock81} and Fong and Harris \cite{FongHarris}. 
Theorem \ref{c2c2c2} implies Alperin's weight conjecture 
for all blocks with a defect group of order at most $8$. Indeed, 
the groups $C_2$, $C_4$, $C_8$ and $C_2\times C_4$ 
admit no automorphisms of odd order, hence arise as defect 
groups only of nilpotent blocks, and by work of Brauer 
\cite{Brauer71}, \cite{Brauer74}
and Olsson \cite{Olsson75}, Alperin's weight conjecture is known 
in the case of $C_2\times C_2$, $D_8$ and $Q_8$.
Using a stable equivalence due to Rouquier we show in
Theorem \ref{AlperinimpliesBroue} that for blocks with
an elementary abelian defect group of order $8$ Alperin's weight 
conjecture implies Brou\'e's isotypy conjecture. 

The proof of Theorem \ref{c2c2c2} uses the classification of finite simple
groups. By the work of Landrock already 
mentioned, Alperin's weight conjecture holds for blocks with an elementary
abelian defect group of order $8$ if and only if all irreducible characters 
in the block have height zero (this is the `if' part of 
Brauer's height zero conjecture, which predicts that all characters in 
a block have height zero if and only if the defect groups are abelian).
This part of Brauer's height zero conjecture has been reduced to quasi-simple
finite groups by Berger and Kn\"orr \cite{BeKn}; we verify in
Theorem \ref{c2c2c2quasisimple} that this reduction works within the
realm of blocks with an elementary abelian defect group of order at 
most $8$ and certain fusion patterns. 
We finally prove Alperin's weight conjecture for blocks with an
elementary abelian defect group of order $8$ for quasi-simple groups
in the remaining sections. For certain classes, such as central
extensions of alternating groups, sporadic groups or finite groups
of Lie type defined over a field of characteristic $2$, this is a simple
inspection (based on calculations and well-known results by many
authors) and yields results for higher rank defect groups as well:

\begin{Theorem} \label{quasisimplealtsporLietwo}
Suppose $p=2$. Let $G$ be a quasi-simple finite group
such that $Z(G)$ has odd order. 

\smallskip\noindent (i)
If $G/Z(G)$ is isomorphic to an alternating group $A_n$, $n\geq 5$, 
then $kG$ has no block with an elementary abelian defect group of 
order $2^r$, where $r\geq 3$.

\smallskip\noindent (ii)
If $G/Z(G)$ is a sporadic simple group or
a finite group of Lie type defined over a field of
characteristic $2$ and if $kG$ has a block $b$ with an
elementary abelian defect group of order $2^r$, where $r\geq 3$,
then either $b$ is the principal block of $\PSL_2(2^r)$
or $r=3$ and $b$ is the principal block of $J_1$, or $b$ is a 
non-principal block of $Co_3$,
and Alperin's weight conjecture holds in these cases.
\end{Theorem}

This follows from combining \ref{Sporadic}, \ref{alt},
\ref{altagain} and \ref{Lietwo} below.
Before embarking on the verification for finite groups
of Lie type defined over a field of odd characteristic, we need
further background material on these groups and their local
structure, collected in the sections \S \ref{redback}, \S \ref{2local}.
Assembling these parts yields the proof of \ref{c2c2c2}
in \S \ref{proofsection}. It is noticeable how few blocks     
of quasi-simple groups have an elementary abelian defect group -
and when they do, many of them are nilpotent:

\begin{Theorem} \label{quasisimpleLieodd}
Suppose $p=2$.
Let $q$ be an odd prime power and $G$ a finite quasi-simple
group. Suppose that $Z(G)$ has odd order.

\smallskip\noindent (i)
If $G/Z(G)$ is a simple group of Lie type $A_n(q)$ or
${^2{A}}(q)$ and if $b$ is a block of $kG$ with an elementary
abelian defect group of order $2^r$ for some $r\geq 3$ 
then $r$ is even and $b$ satisfies Alperin's weight
conjecture.

\smallskip\noindent (ii)
If $G/Z(G)$ is a simple group of Lie type $B_n(q)$, $C_n(q)$, 
$D_n(q)$ then every block of $kG$ with an elementary abelian
defect group of order $2^r$ for some integer $r\geq 3$ is nilpotent.

\smallskip\noindent (iii)
If $G/Z(G)$ is simple of type $G_2(q)$ or ${^3D_4(q)}$
then $kG$ has no block with an elementary
abelian defect group of order $2^r$, where $r\geq 3$.

\smallskip\noindent (iv)
If $G/Z(G)$ is simple of type ${^2G}_2(q)$ and if $b$ is
a block of $kG$ with an elementary
abelian defect group of order $2^r$, where $r\geq 3$,
then $r=3$ and $b$ is the principal block of ${^2G_2(q)}$,
and Alperin's weight conjecture holds in this case.
\end{Theorem}

This follows from combining
Corollary \ref{jordandecomabel}, the Theorems \ref{typeA}, 
\ref{type2A}, \ref{classical} and the Propositions \ref{typeG2},
\ref{type2G2}, \ref{type3D4} below. 
Note the absence of the exceptional types $F$ and $E$ in the
above result - we do not know whether these groups actually
have blocks with elementary abelian defect groups of order
$8$. What we show is that if they do then these blocks satisfy 
Alperin's weight conjecture. 
At present there seems to be no general reduction for blocks with 
elementary abelian defect groups of order $2^r$ for $r\geq 4$, 
but the above results for quasi-simple groups
would allow us to settle Alperin's weight conjecture
in infinitely many cases if we did have a satisfactory
reduction. More precisely, using the above results 
and the fact (see \cite[Corollary]{CoSe})
that the $2$-rank of exceptional finite simple groups of Lie 
type in odd characteristic is at most $9$ we obtain the following.

\begin{Corollary} \label{quasisimplecheckcor}
Suppose $p=2$.
Let be $G$ a finite quasi-simple group such that $Z(G)$ has
odd order. Every block of $kG$ with
an elementary abelian defect group of order $2^r$ for some 
integer $r\geq 10$ satisfies Alperin's weight conjecture.
\end{Corollary}


\section{Reduction techniques}

For general background on block theory we refer to \cite{Thev}.
Given a finite group $G$ and a block $b$ of $\OG$ or of $kG$, 
we denote by $\Irr_K(G,b)$ the set of ordinary irreducible $K$-valued
characters of $G$ associated with $b$ and by $\Irr_k(G,b)$ the
set of irreducible Brauer characters of $G$ associated with $b$.
We set $\ell(b)=$ $|\Irr_k(G,b)|$.
We denote by $\Z\Irr_K(G,b)$ the group of class functions on
$G$ generated by $\Irr_K(G,b)$, and by $\Z\Irr_k(G,b)$ the
corresponding group of class functions on the set of $p$-regular
elements in $G$. By a classical result of Brauer, 
the {\it decomposition map} $\Z\Irr_K(G,b)\to$ $\Z\Irr_k(G,b)$ 
induced by restriction of class functions to $p$-regular elements 
is surjective. The kernel of the decomposition map, denoted by 
$L^0(G,b)$, consists of all class functions associated with $b$ 
which vanish on the set of $p$-regular elements of $G$, or 
equivalently, all generalised characters in $\Z\Irr_K(G,b)$ 
which are perpendicular to the characters of the projective 
indecomposable $\OG$-modules associated with $b$.
We write $L^0(G)$ instead of $L^0(G,1)$ if $\OG$ is indecomposable
as an $\CO$-algebra. 
We denote by $C_n$ a cyclic group of order $n$.
For $G$ a finite group and $\alpha\in$ $H^2(G;k^\times)$ we denote
by $k_\alpha G$ the {\it twisted group algebra} which is equal to
$kG$ as a $k$-vector space, endowed with the bilinear multiplication
$x\cdot y=$ $\alpha(x,y) (xy)$, for $x$, $y\in$ $G$,  where $\alpha$ denotes 
abusively a $2$-cocycle representing the class $\alpha$ (and one verifies that
this construction is, up to isomorphism, independent of the choice of this
$2$-cocycle). If $b$ is a block of $kG$ with a defect group $P$ and if
$c$ is the corresponding block of $kN_G(P)$ then $c$ is a sum of
$N_G(P)$-conjugate block idempotents of $kC_G(P)$; for any choice
$e$ of such a block idempotent of $kC_G(P)$ the group 
$E=$ $N_G(P,e)/PC_G(P)$ is called the {\it inertial quotient of}
$b$. This is a $p'$-subgroup of the outer automorphism group of $P$,
unique up to conjugacy by an element in $N_G(P)$, 
hence lifts uniquely, up to conjugacy, to a $p'$-subgroup of 
$\Aut(P)$, still abusively denoted $E$ and called inertial quotient.
By there is $\alpha\in$ $H^2(E;k^\times)$ such that 
$kN_G(P)c$ is Morita equivalent to the twisted group algebra
$k_\alpha(P\rtimes E)$, where $\alpha$ is extended trivially from
$E$ to $P\rtimes E$. 
We review some of the standard reduction techniques due to 
Dade, Fong, K\"ulshammer, Puig, Reynolds.
The reduction techniques work irrespective of the characteristic,
so for now, $p$ is an arbitrary prime.  

\begin{Proposition}[Fong-Reynolds reduction \cite{Fongred}, 
\cite{Reynolds63}] 
\label{Fongreduction}
Let $G$ be a finite group, $N$ a normal subgroup of $G$, $c$ a $G$-stable
block of $kN$ and $b$ a block of $kG$ such that $bc = b$. Let $P$
be a defect group of $b$. If $P\cap N = 1$ then $kNc$ is a block
of defect zero and we have a canonical  isomorphism
$$kGc \cong kNc \tenk k_\alpha(G/N)$$
for some $\alpha\in H^2(G/N;k^\times)$ such that $kGb$ is Morita
equivalent to a block $\hat b$ of a finite central $p'$-extension
of $G/N$ via a bimodule with diagonal vertex $\Delta P$ and endo-permutation
source.
\end{Proposition}

See for instance \cite[4.4]{HaLi} for 
an explicit description of this isomorphism.
The previous result has been generalised by K\"ulshammer as follows:

\begin{Proposition} [{K\"ulshammer \cite[Proposition 5, 
Theorem 7]{Kuelsred}}] 
\label{Kuelsreduction}
Let $G$ be a finite group, $N$ a normal subgroup of $G$, $c$ a $G$-stable
block of $kN$ and $b$ a block of $kG$ such that $bc = b$. Suppose that 
for any $x\in G$ conjugation by $x$ induces an inner automorphism of 
$kNc$. Then there is a canonical isomorphism 
$$kGc \cong kNc\ \ten_{Z(kNc)}\ Z(kNc)_\alpha(G/N)$$
for some $\alpha\in H^2(G/N; (Z(kNc))^\times)$. Moreover, if
$N$ contains a defect group $P$ of $b$ then the blocks $kGb$ and 
$kNc$ are source algebra equivalent.
\end{Proposition}

As before, this isomorphism can be described explicitly; see e.g.
\cite[2.1]{CEKL}.

\begin{Proposition} [{Dade \cite[3.5, 7.7]{Dadeext}}] 
\label{Dadereduction}
Let $G$ be a finite group, $N$ a normal subgroup of $G$, $c$ a $G$-stable
block of $kN$ and $b$ a block of $kG$ such that $bc = b$. Suppose
that no element $x\in G-N$ acts as inner automorphism on $kNc$.
Then $b=c$ and $G/N$ has order prime to $p$.
\end{Proposition}

\begin{Proposition}[{Puig \cite[4.3]{Puignilext}}] \label{nilextremark}
Let $G$ be a finite group, $N$ a normal subgroup of $G$, $c$ a $G$-stable
block of $kN$ and $b$ a block of $kG$ such that $bc = b$. Suppose
that $b$ is nilpotent. Then the block algebra of $c$ is Morita equivalent 
to its Brauer correspondent 
via a Morita equivalence induced by a bimodule with endo-permutation 
source; in particular, $c$ satisfies Alperin's weight conjecture.
\end{Proposition}

\section{Background material on blocks with defect group 
$C_2\times C_2\times C_2$}

We assume from now on that $p=2$. Let $P\cong C_2\times C_2\times C_2$
be an elementary abelian group of order $8$. The order of $\GL_3(2)$
is $8\cdot 21$, from which one easily deduces that a non-trivial 
subgroup $E$ of $\Aut(P)$ of odd order has either order $3$ or $7$, or
is a Frobenius group of order $21$. In all cases, $E$ has a trivial
Schur multiplier, and hence any block with a normal defect group $P$
has as source algebra $k(P\rtimes E)$. What is unusual in this case
is that the number of characters at the local level does not depend
on fusion (this is well-known; we include a sketch of a proof for
the convenience of the reader):

\begin{Proposition}  \label{c2c2c2local}
Let $P$ be an elementary abelian group of order $8$ and
$E$ a subgroup of $\Aut(P)$ of odd order. The group 
$P\rtimes E$ has $8$ ordinary irreducible characters.
\end{Proposition}

\begin{proof}
This is trivial if $E = 1$. Suppose that $|E| = 3$. Then
$E$ fixes an involution in $P$, hence $P\rtimes E\cong$
$C_2\times (V_4\rtimes C_3)$, from which the statement follows.
If $|E| = 7$ then $P\rtimes E$ is a Frobenius group with $E$ acting 
transitively on the involutions in $P$, which implies the result.
The only remaining case is where $E$ is a Frobenius group $C_7\rtimes C_3$.
In that case, $E$ has $5$ characters, three of degree $1$
with $C_7$ in the kernel, and two more of degree $3$ corresponding 
to the non-trivial $C_3$-orbits in the character group of $C_7$. 
Using the fact that the degrees of irreducible characters of
$P\rtimes E$ divide $|E|$ and that the square of their
degrees sums up to $|P\rtimes E|=$  $168$ one finds that there
are three further irreducible characters of degree $7$.
\end{proof}
 
Landrock showed the inequality $\ell(b)\leq 8$ of Alperin's weight 
conjecture for $2$-blocks $b$ with an elementary abelian defect
group $P$ of order $8$ without the 
classification of finite simple groups. Landrock's results are more 
precise in  that they also  include information about defects and 
heights of characters; we briefly recall these notions. 
Given a block $b$ of $\OG$ or of $kG$ with a defect group $P$ of order
$p^d$, the {\it defect} of a character $\chi$ in the set $\Irr_K(G,b)$
of irreducible $K$-valued characters associated with $b$ is the
integer $d(\chi)$ such that $p^{d(\chi)}$ is the largest power of
$p$ dividing the rational integer $\frac{|G|}{\chi(1)}$. It is
well-known that $d(\chi)\leq d$; the integer $h(\chi)=$ $d-d(\chi)$
is called the {\it height} of $\chi$. There is always at least
one character in $\Irr_K(G,b)$ having height zero, and it has
been conjectured by Brauer that $P$ is abelian if and only if all
characters in $\Irr_K(G,b)$ have height zero. The following 
summary of some of Landrock's results in \cite{Landrock81} implies
in particular that Alperin's weight conjecture holds for
blocks with an elementary abelian defect group of order $8$ if
and only if all characters in those blocks have height zero.

\begin{Proposition}[{Landrock, \cite[2.1, 2.2, 2.3]{Landrock81}}]
\label{leqeight}
Let $G$ be a finite group and $b$ a block of $kG$ with an elementary
abelian defect group $P$ of order $8$ and inertial quotient 
$E\leq \Aut(P)$. Then $5\leq$ $|\Irr_K(G,b)| \leq 8$. If 
$|\Irr_K(G,b)|=8$ then all characters in $\Irr_K(G,b)$ have 
height zero. If $|\Irr_K(G,b)|<8$ then exactly four characters
in $\Irr_K(G,b)$ have height zero, the remaining characters
have height one, and $\ell(b) =4$. Moreover, the following hold.

\smallskip\noindent (i)
If $E$ has order $1$ then $|\Irr_K(G,b)| = 8$ and $\ell(b) =1$.

\smallskip\noindent (ii)
If $E$ has order $3$ then $|\Irr_K(G,b)| = 8$ and $\ell(b)=3$.

\smallskip\noindent (iii)
If $E$ has order $7$ then either $|\Irr_K(G,b)|=5$, $\ell(b)=4$ or
$|\Irr_K(G,b)|=8$, $\ell(b)=7$.

\smallskip\noindent (iv)
If $E$ has order $21$ then either $|\Irr_K(G,b)|=7$, $\ell(b)=4$ or 
$|\Irr_K(G,b)|=8$, $\ell(b)=5$.
\end{Proposition}

Besides Landrock's original proof it is also possible to prove this as 
a consequence of stronger results obtained later: the case $|E|=1$ is a 
particular case of nilpotent blocks \cite{BrPunil}, the case $|E|=3$ 
follows from Watanabe \cite[Theorem 1]{Wat91}. In the case $|E|=7$, the
group $E$ acts regularly on $P-\{1\}$, and by a result of Puig in 
\cite{Puabelian} there is a {\it stable equivalence of Morita type}
(cf. \cite[\S 5]{Broue1994}) between 
$\OGb$ and $\CO(P\rtimes E)$. Any such stable equivalence induces an 
isometry $L^0(G,b)\cong L^0(P\rtimes E)$ between the generalised character
groups which vanish on $p$-regular elements; in this case, these groups
have rank one and are generated by an element of norm $8$, whence the
inequality $|\Irr_K(G,b)|\leq$ $8$. If $|E|=21$ then there is again a stable 
equivalence of Morita  type, by a result of Rouquier in \cite{Rouqb}. Again 
by calculating a basis 
of $L^0(P\rtimes E)$ - which in this case has rank $3$ with a basis 
consisting of three elements of norm four, one also gets this inequality. 
See for instance \cite{KeLionesimple} for an exposition of the well-known 
technique exploiting partial isometries induced by stable equivalences of 
Morita type; this will be used in the proof of Theorem 
\ref{AlperinimpliesBroue}. 
The following observation is a slight refinement of 
Proposition \ref{leqeight}(iii)
in the case where the inertial quotient has order $7$ and
$|\Irr_K(G,b)|<8$.

\begin{Proposition}\label{punchheight}
Let $G$ be a finite group, $b$ a block of $kG$ with an elementary 
abelian defect group $P$ of order $8$ and inertial quotient $E$ of 
order $7$. Suppose that $b$ does not satisfy Alperin's weight 
conjecture. Then there is a labelling $\Irr_K(G,b)=$ 
$\{\chi_i\ |\ 1\leq i\leq 5\}$ with the following properties:

\smallskip\noindent (i)
$\chi_1$ has height one and $\chi_i$ has height zero, 
for $2\leq i\leq 5$.

\smallskip\noindent (ii)
The group $L^0(G,b)$ has rank $1$ and a basis element of the
form $2\chi_1-\sum_{i=2}^5 \delta_i\chi_i$ for some signs
$\delta_i\in\{\pm 1\}$, $2\leq i\leq 5$; moreover, at least one of 
the $\delta_i$ is positive.

\smallskip\noindent (iii)
If $i,j\in\{2,3,4,5\}$ such that $\chi_i(1)=$ $\chi_j(1)$ then
$\delta_i=$ $\delta_j$.
\end{Proposition}

\begin{proof}
Statement (i) is just a reformulation of \ref{leqeight} (iii).
Using Puig's stable equivalence of Morita type 
\cite[6.8]{Puabelian}
we get that $L^0(G,b)\cong$ $L^0(P\rtimes E)$, which is a free abelian
group of rank one with a basis element of norm $8$. The only way to
write a norm $8$ element in $L^0(G,b)$ with less than $8$ characters
is with five characters, exactly one of which shows up with 
multiplicity $2$, and then this character must have height one, 
as follows from comparing character degrees in conjunction with 
the fact that every generalised character in $L^0(G,b)$ vanishes 
at $1$. The signs $\delta_i$ cannot all be negative because the group
$L^0(G,b)$ does not contain any actual non-zero character (again because
its elements vanish at $1$). This proves (ii). If $i, j\in$ 
$\{2,3,4,5\}$ then
$\delta_i\chi_i-\delta_j\chi_j$ is orthogonal to $L^0(G,b)$, hence
a generalised projective character. In particular, at
$1$, its value is divisible by the order of a Sylow $2$-subgroup
of $G$. But if $\chi_i(1)=$ $\chi_j(1)$ and $\delta_i\neq$ $\delta_j$,
this value is $\pm 2\chi_i(1)$, which cannot be divisible by the 
order of a Sylow $2$-subgroup of $G$ as $\chi_i$ has height zero.
The result follows.
\end{proof}

It is not known in general whether a Morita 
equivalence between two block algebras preserves their local structures,
but some easy standard block theoretic arguments show that this is  
true, even for stable equivalences of Morita type, if one of the two 
blocks has an elementary abelian defect group of order $8$.

\begin{Proposition} \label{localstructurepreserved}
Let $G$, $H$ be finite groups, $b$ a block of $kG$ with an
elementary abelian defect group $P$ of order $8$ and $c$ a block
of $kH$ with a defect group $Q$. If there is a stable equivalence
of Morita type between $kGb$ and $kHc$ then $Q\cong P$ and the
blocks $b$ and $c$ have isomorphic inertial quotients
(or equivalently, isomorphic fusion systems).
\end{Proposition}

\begin{proof} 
A stable equivalence of Morita type preserves the
largest elementary divisors of the Cartan matrices of the blocks, 
and these are equal to the orders of the defect groups, 
whence $|Q|=$ $|P|$. A stable equivalence 
of Morita type preserves also the complexity (cf. \cite[5.3.4]{BensonII})
of modules; since the largest
complexity of a module in a block is the rank of a defect group,
we get that $Q$ has rank $3$, and thus $Q\cong$ $P$. Alternatively,
a stable equivalence of Morita type preserves the Krull dimension
of the Hochschild cohomology rings, which are also known to
be equal to the ranks of the defect groups. (This part of the argument
is well-known to remain valid for blocks with arbitrary elementary abelian
defect groups, but we do not need this here.) Finally, a stable
equivalence of Morita type preserves the rank of $L^0(G,b)$, which
is equal to $|\Irr_K(G,b)|-\ell(b)$, or also equal to 
$\sum_{(u,e_u)}\ \ell(e_u)$, where $(u,e_u)$ runs over a 
set of representatives of the conjugacy classes of non-trivial 
$(G,b)$-Brauer elements. It happens so that this number
determines the structure of the inertial quotient $E$ of $b$. 
Indeed, by Proposition \ref{leqeight}, this number is equal to
$7$ if and only if $|E|=1$, equal to $5$ if and only if $|E|=3$,
equal to $1$ if and only if $|E|=7$, and equal to $3$ if and only
if $|E|=21$, whence the result.
\end{proof}
 
When dealing with the exceptional groups of type $E_7(q)$ 
we will need a refinement of the preceding result because 
the finite group of Lie type $E_7(q)$ is a central extension 
of the simple group of type $E_7(q)$ by an involution, and 
so Bonnaf\'e-Rouquier's Jordan decomposition 
\cite[\S 11, Th\'eor\`eme B']{BonRou} 
will have to be applied to blocks with a defect group of 
order $2^4$. 

\begin{Proposition} \label{defect4}
Let $G$, $H$ be finite groups, $b$ a block of $G$ with
a defect group $P$ and $c$ a block of $H$ with a defect group $Q$.
Suppose that there are central involutions $s\in Z(G)$ and
$t\in Z(H)$ such that $P/\langle s\rangle$ is
elementary abelian of order $8$. Denote by $\bar b$ and
$\bar c$ the images of $b$ and $c$ in $kG/\langle s\rangle$ and
$kH/\langle t\rangle$, respectively. Suppose that the block
algebras $kGb$ and $kHc$ are Morita equivalent. If $\bar b$
 does not satisfy Alperin's weight conjecture then  neither  
does $\bar c$ and  the defect groups  of  $\bar c$  
are elementary abelian of order $8$. 
\end{Proposition}

\begin{proof}
Suppose that $\bar b$ does not satisfy Alperin's weight 
conjecture. Since the defect group $P/\langle s\rangle$ of
$\bar b$ is elementary abelian of order $8$ we have 
$\ell(\bar b) =4$  by Proposition \ref{leqeight}. 
Using that the number
of isomorphism classes of simple modules is invariant under
central $2$-extensions and Morita equivalences we get that
$\ell(\bar c) =4$. Note that $P$ and $Q$ have the same order
since $b$ and $c$ are Morita equivalent, and hence $\bar c$
has a defect group $R=$ $Q/\langle t\rangle$ of order $8$. 
If $R$ is cyclic or abelian of rank $2$ or isomorphic to one
of the non-abelian groups of order $8$ then
then $\ell(\bar c)\in$ $\{1,3\}$, a contradiction. So,  
$R$ is elementary abelian. Now it is immediate from Proposition 
\ref{leqeight} that   $\bar c$ does not satisfy the weight conjecture.
\end{proof}

\begin{Remark}
Experts seem to agree that the Morita equivalences 
from Bonnaf\'e-Rouquier's Jordan decomposition 
\cite[\S 10, \S 11]{BonRou}
should preserve the local structure of the blocks, but at
present there is no written reference for this fact. 
The two propositions \ref{localstructurepreserved} and 
\ref{defect4} circumvent this issue by adhoc methods.
\end{Remark}

\section{Reduction to quasi-simple groups}

The part of Brauer's height zero conjecture predicting that all characters
in a block with an abelian defect group have height zero
has been reduced to
blocks of quasi-simple finite groups in work of Berger and Kn\"orr
\cite{BeKn}. We need to make sure that in the reduction 
we can indeed restrict the problem to checking only defect groups of order at
most $8$; this is not entirely obvious since in Step 6 of the proof
of \cite[Theorem]{BeKn} the order of the defect group may possibly go up,
an issue which arises also in the alternative proof given by Murai in 
\cite[\S 6]{Murai94}. 

\begin{Theorem} \label{c2c2c2quasisimple}
Let $G$ be a finite group and $b$ a block of $kG$ with an elementary
abelian defect group $P$ of order $8$. Suppose that $|G/Z(G)|$ is minimal
such that $|\Irr_K(G,b)| < 8$. Then $Z(G)$ has odd order and $G/Z(G)$ 
is simple. If moreover we also choose $|Z(G)|$ minimal then $G$ is 
quasi-simple. In addition, the inertial quotient of $b$ is either
cyclic of order $7$ or a Frobenius group of order $21$.
\end{Theorem}

\begin{proof}
If $|Z(G)|$ is even then $P\cap Z(G)$ is non-trivial, hence contains a
subgroup $Z$ of order $2$. The image $\bar b$ of $b$ in $kG/Z$ is then 
a block of $kG/Z$ with a Klein four defect group $P/Z$. Thus, since $\bar b$
satisfies Alperin's weight conjecture, so does $b$, a contradiction to the
assumption. Thus $Z(G)$ has odd order. 
Let $(P,e)$ be a maximal $(G,b)$-Brauer pair and set $E = N_G(P,e)/C_G(P)$. 
By Proposition \ref{leqeight}, the order of $E$ is either $7$ or $21$. 
In both cases, $E$ acts transitively on $P-\{1\}$. Thus, if $N$ is a normal 
subgroup of $G$ then either $N\cap P = \{1\}$ or $P\leq N$. Moreover, the 
minimality of $|G/Z(G)|$ implies that if $N$ is a normal subgroup of $G$ 
containing $Z(G)$ then there is a unique block $c$ of $kN$ covered by $b$; 
that is, $bc = b$. Let now $N$ be a maximal
normal subgroup of $G$ containing $Z(G)$ and let $c$ be the block of $kN$
satisfying $bc = b$. Consider first the case $P\cap N = \{1\}$.
In that case, by Proposition \ref{Fongreduction}, we have an isomorphism
$$kGc \cong kNc \tenk k_\alpha(G/N)$$
for some $\alpha\in H^2(G/N;k^\times)$ such that $kGb$ is Morita
equivalent to a block $\hat b$ of a finite central $2'$-extension
$H$ of the simple group $G/N$. Consider next the case $P\leq N$.
Let $G[c]$ be the subgroup of $G$ consisting of all $x\in G$ such that
conjugation by $x$ induces an inner automorphism of $kNc$. Since $N$ is
maximal normal in $G$ we have either $N = G[c]$ or $G[c] = G$.
Suppose first that $G[c] = G$. Then, by Proposition \ref{Kuelsreduction}, 
we have an isomorphism
$$kGc \cong kNc\ \ten_{Z(kNc)}\ Z(kNc)_\alpha(G/N)$$
for some $\alpha\in H^2(G/N; (Z(kNc))^\times)$, and the blocks 
$kGb$ and $kNc$
are source algebra equivalent - but this contradicts the minimality of
$|G|$ since source algebra equivalent blocks have in particular the
same number of ordinary irreducible characters.
Thus we have $N = G[c]$. Then, by Proposition \ref{Dadereduction}, we have
$b=c$ and $G/N$ has odd order. Since also $G/N$ is simple, this implies that
$G/N$ is cyclic of odd prime order $\ell$, by Feit-Thompson's Odd Order
theorem. By standard results in Clifford
theory, any $\eta\in\Irr_K(N,c)$ is either $G/N$-stable, in which case it 
extends to exactly $\ell$ different characters in $\Irr_K(G,b)$, or 
$\Ind_N^G(\eta)\in\Irr_K(G,b)$. It follows that
$$|\Irr_K(G,b)| = \ell\cdot m + r < 8$$
where $m$ is the number of characters in $\Irr_K(N,c)$ fixed by $G/N$ and
where $r$ is the number of non-trivial $G/N$-orbits in $\Irr_K(N,c)$.
Since $\ell\geq 3$ we have $m\leq 2$.
Using induction, we have 
$$8 = |\Irr_K(N,c)| = \ell\cdot r + m$$
In all possible choices of $\ell$, $m$, $r$ satisfying
this equality we get the contradiction $\ell\cdot m + r\geq 8$.
Thus the assumption $P\leq N$ is not possible, and therefore
the above implies that $G/Z(G)$ is simple. Finally, since $G/Z(G)$ is simple
we have $G = Z(G)[G,G]$, so $G/[G,G]$ acts trivially on all characters
of $[G,G]$, and so $|\Irr_K(G,b)| = |\Irr_K([G,G], d)|$, where $d$ is a
block of $[G,G]$ satisfying $bd = b$. After repeating this, if
necessary, we also may assume that $G$ is perfect, hence quasi-simple.
\end{proof}

\section{Perfect isometries}

Using Rouquier's stable equivalence for blocks with an elementary
abelian defect group of order $8$, described in the Appendix below,
we show that Alperin's weight conjecture implies the character 
theoretic version of Brou\'e's abelian defect conjecture for
these blocks:

\begin{Theorem} \label{AlperinimpliesBroue}
Let $G$ be a finite group and let $b$ be a block of $\OG$ with
an elementary abelian defect group $P$ of order $8$. Set
$H = N_G(P)$ and denote by $c$ the block of $\OH$ with defect
group $P$ corresponding to $b$ via the Brauer correspondence.
Suppose that $K$ is large enough for $b$ and $c$.
If $|\Irr_K(G,b)| =$ $|\Irr_K(H,c)|$ then the blocks $b$ and
$c$ are isotypic; in particular, there is a perfect
isometry $\Z\Irr_K(G,b)\cong$ $\Z\Irr_K(H,c)$.
In particular we have $Z(\OG b)\cong$ $Z(\OH c)$.
\end{Theorem}

See \cite[6.1, 6.2]{Broue}, \cite{Brouecara}, \cite{Broue1994}
for more precise versions of Brou\'e's abelian defect conjecture, 
as well as background material on perfect isometries and isotypies.

\begin{proof}[Proof of Theorem \ref{AlperinimpliesBroue}]
We refer to \cite[\S\S 2, 3]{KeLionesimple} for notation and an 
expository account of the standard techniques on extending partial 
isometries induced by stable equivalences of Morita type. 
By a result of Rouquier in \cite{Rouqb}, there is a stable equivalence
of Morita type between the block algebras of $b$ and of $c$ over
$\CO$ (a proof of this result is given in 
Theorem \ref{c2c2c2stableequivalence} below) 
given by a bounded complex of bimodules whose indecomposable
summands all have diagonal vertices and trivial source. 
Denote by $E$ the inertial quotient of $b$. 
Since the block algebra $\OHc$  is Morita equivalent to 
$\CO(P\rtimes E)$ via a bimodule with diagonal vertex and trivial source
this implies that there is a stable equivalence of Morita type
between $\OGb$ and $\CO(P\rtimes E)$, induced by a bounded complex
of bimodules with diagonal vertices and trivial source. It is 
well-known (see e.g.  \cite[3.1]{KeLionesimple}) that any such 
stable equivalence induces an isometry $L^0(P\rtimes E)\cong L^0(G,b)$.
It suffices to show that this partial isometry extends to an isometry 
$\Z\Irr_K(P\rtimes E) \cong \Z\Irr_K(G,b)$ because any such extension 
is then a $p$-permutation equivalence by \cite[3.3]{KeLionesimple},
hence induces an isotypy by \cite[Theorem 1.4]{Linpperm}.
We do this by running through all possible inertial
quotients $E$. 

If $E = \{1\}$ the block $b$ is nilpotent, hence
Morita equivalent to $\OP$, and so the result holds trivially in this 
case. Assume that $|E|=3$. Then $P\rtimes E\cong C_2\times A_4$, and
hence we can list the eight ordinary irreducible characters of 
$P\rtimes E$ in such a way that the three characters of the projective
indecomposable $\CO(P\rtimes E)$-modules are of the form
$\chi_i+\chi_{i+3}+\chi_7+\chi_8$
where $1\leq i\leq 3$. Thus a basis of $L^0(P\rtimes E)$ is of the form
$$\{\chi_1-\chi_4, \chi_2-\chi_5, \chi_3-\chi_6, \chi_7-\chi_8,
\chi_1+\chi_2+\chi_3-\chi_7\}$$
The four elements of norm $2$ in this basis must be sent to norm
$2$ elements under the isometry $L^0(P\rtimes E)\cong L^0(G,b)$
no two of which involve a common irreducible character in $\Irr_K(G,b)$,
and hence are mapped to elements of the form
$\delta_1(\eta_1-\eta_4)$, $\delta_2(\eta_2-\eta_5)$, 
$\delta_3(\eta_3-\eta_6)$, $\delta_7(\eta_7-\eta_8)$, for some
labelling $\Irr_K(G,b) = \{\eta_i\ | 1\leq i\leq 8\}$ and some signs
$\delta_i$. We may then choose notation (after possibly exchanging
$\eta_1$ and $\eta_4$ etc.) in such a way that the image in 
$L^0(G,b)$ of the norm four element $\chi_1+\chi_2+\chi_3-\chi_7$
is equal to 
$\delta_1\eta_1+\delta_2\eta_2+\delta_3\eta_3-\delta_7\eta_7$.
Setting $\delta_{i+3}= \delta_i$ for $1\leq i\leq 3$, and $\delta_8
= \delta_7$ it follows that the map sending $\chi_i$ to $\delta_i\eta_i$
induces an isometry $\Z\Irr_K(P\rtimes E) \cong \Z\Irr_K(G,b)$ 
extending the isometry $L^0(P\rtimes E)\cong L^0(G,b)$ as
required. 

Assume next that $|E| = 7$. Then $P\rtimes E$ is a Frobenius
group, whose seven characters of the projective indecomposable modules
are of the form $\chi_i + \chi_8$
with $1\leq i\leq 7$, for some labelling $\Irr_K(P\rtimes E) =
\{\chi_i\ |\ 1\leq i\leq 8\}$; the characters $\chi_i$, $1\leq i\leq 7$
have $P$ in their kernel, and $\chi_8$ is induced from a nontrivial
character of $P$ to $P\rtimes E$. The group $L^0(P\rtimes E)$ has
rank $1$, with a basis element $\sum_{i=1}^{7} \chi_i\ -\ \chi_8$.
This element has norm $8$, hence its image in $L^0(G,b)$ has norm
$8$ as well. Moreover, all irreducible characters in $\Irr_K(G,b)$
have to be involved in this element. Since we assume that
Alperin's weight conjecture holds for $b$, we have $|\Irr_K(G,b)| = 8$,
and so the image of this element in $L^0(G,b)$ is of the form
$\sum_{i=1}^{7} \delta_i\eta_i - \delta_8\eta_8$ for some labelling
$\Irr_K(G,b)=\{\eta_i\ |\ 1\leq i\leq 8\}$ and some signs $\delta_i$.
Again, the map sending $\chi_i$ to $\delta_i\eta_i$ induces the 
required isometry $\Z\Irr_K(P\rtimes E) \cong \Z\Irr_K(G,b)$.

Finally, assume that $|E|=21$. Then $E$ is itself a Frobenius
group, isomorphic to $C_7\rtimes C_3$ with the obvious nontrivial
action of $C_3$ on $C_7$. The group $E$ has $5$ ordinary irreducible
characters, hence $\CO(P\rtimes E)$ has five isomorphism classes
of projective indecomposable modules, and thus $L^0(G,b)$ has rank
$3$. We can label $\Irr_K(P\rtimes E) = \{\chi_i\ |\ 1\leq i\leq 8\}$
in such a way that $\chi_1$, $\chi_2$, $\chi_3$ have degree $1$, 
the characters $\chi_4$, $\chi_5$ have degree $3$ and $\chi_6$, $\chi_7$, 
$\chi_8$ have degree $7$. An easy calculation shows that $L^0(P\rtimes E)$ 
has a basis of the form
$$\{\chi_6-\chi_4-\chi_5-\chi_1,\ \chi_7-\chi_4-\chi_5-\chi_2,\ 
\chi_8-\chi_4-\chi_5-\chi_3\}$$ 
consisting of three elements of norm $4$, such that any two 
different of these basis elements involve two common irreducible characters.
Thus the same is true for $L^0(G,b)$. Using again the hypothesis 
$|\Irr_K(G,b)| = 8$ one deduces that  $L^0(G,b)$  has a basis of the form
$$\{\delta_6\eta_6-\delta_4\eta_4-\delta_5\eta_5-\delta_1\eta_1,\ 
\delta_7\eta_7-\delta_4\eta_4-\delta_5\eta_5-\delta_2\eta_2,\ 
\delta_8\eta_8-\delta_4\eta_4-\delta_5\eta_5-\delta_3\eta_3\}$$
for some labelling $\Irr_K(G,b)=\{\eta_i\ |\ 1\leq i\leq 8\}$ and
some signs $\delta_i$. As before, the map sending $\chi_i$ to $\delta_i\eta_i$
induces the required isometry $\Z\Irr_K(P\rtimes E) \cong \Z\Irr_K(G,b)$.
Since a perfect isometry induces an isomorphism between centers, the
result follows.
\end{proof}

\section{Sporadic finite simple groups}

Let $G$ be a finite group and $b$ a block
of $\OG$. The {\it kernel} ${\mathrm{Ker}}_G(b)$ of $b$ is
defined by
$$
    {\mathrm{Ker}}_G(b) \ = \
    {\bigcap_{\chi \in {\mathrm{Irr}}_K(G,b)}} 
    {\mathrm{Ker}}(\chi),
$$
see \cite[\S 3]{Brauer71}. By  \cite[Proposition (3B)]{Brauer71} we have 
${\mathrm{Ker}}_G(b) = O_{p'}(G) \cap {\mathrm{Ker}}(\chi)$
for  any $\chi \in {\mathrm{Irr}}_K(G,b)$.
Hence, ${\mathrm{Ker}}_G(b)$ is a  normal $p'$-subgroup of $G$,
See \cite[Chap.5, Theorem 8.1]{NagaoTsushima} for an exposition
of this material. We say that $b$ is {\it faithful} if
${\mathrm{Ker}}_G(b) = 1$.

\begin{Proposition} \label{KernelOfBlock}
Let $G$ be a quasi-simple finite group such that $p$ divides $|G|$.

\smallskip\noindent (i) 
We have $O_{p'}(Z(G)) = O_{p'}(G)$ and
${\mathrm{Ker}}_G(b) = O_{p'}(Z(G)) \cap {\mathrm{Ker}}(\chi)$
for any $\chi \in {\mathrm{Irr}}_K(G,b)$.

\smallskip\noindent (ii) 
Set $\bar G=$ $G/{\mathrm{Ker}}_G(b)$ and denote by $\bar b$ the image
of $b$ in $\CO\bar G$. Then $\bar b$ is a faithful block of $\CO\bar G$
and the canonical map $G\to$ $\bar G$ induces an $\CO$-algebra
isomorphism $\OG b\cong \CO\bar G\bar b$.
\end{Proposition}

\begin{proof}
Statement (i) follows from Brauer's result mentioned  
above, and (ii) is an easy consequence of
\cite[Chap.5, Theorem 8.8]{NagaoTsushima}.
\end{proof}

The following table is due to Noeske \cite{Noeske}.
By Proposition \ref{KernelOfBlock}(ii)
it is enough to consider faithful blocks.

\begin{Proposition} [{\cite{Noeske}}] \label{Noeske}
The following is a list of all faithful non-principal 
$2$-blocks with non-cyclic abelian defect groups of sporadic 
simple groups and their covers.
Each number in the $2$nd column corresponds to the number
attached to each block in the Modular Atlas \cite{ModAtlas}.

{\rm
\begin{center}
\begin{tabular}{l|l|l|l|l}
\hline
{group} & {blocks $b$} & defect groups & $k(b)$ & $\ell(b)$ \\
\hline
$M_{12}$    & $2$       & $C_2 \times C_2$ & $4$ & $3$ \\
$12.M_{22}$ & $4$, \, $5$ 
   & $C_2 \times C_2$, \, $C_2 \times C_2$ & $4$, \, $4$ & $1$, \, $1$ \\
$J_2$ & $2$ & $C_2 \times C_2$ & $4$ & $3$ \\
$HS$ & $2$ & $C_2 \times C_2$ & $4$ & $3$ \\
$Ru$ & $2$ & $C_2 \times C_2$ & $4$ & $3$ \\
$Co_3$ & $2$ & $C_2 \times C_2 \times C_2$ & $8$ & $5$ \\
$2.Fi_{22}$ & $3$ & $C_2 \times C_2$ & $4$ & $1$ \\
$Fi{_{24}}'$ & $2$ & $C_2 \times C_2$ & $4$ & $3$ \\
\end{tabular} 
\end{center}
}
\end{Proposition}

\begin{Proposition} \label{Sporadic}
Let $G$ be a quasi-simple finite group 
such that $G/Z(G)$ is a sporadic simple group,
and let $b$ be a block of $kG$ 
with an elementary abelian defect group of order $2^r$ for
some integer $r\geq 3$. Then $r=3$ and either $b$ is
the principal block of $kJ_1$ or a non-principal block
of $kCo_3$. In both cases we have 
$|{\mathrm{Irr}}_K(G,b)| = $ $8$; in particular,
Alperin's weight conjecture holds for $b$.
\end{Proposition} 

\begin{proof}
If $b$ is a principal block then $r=3$ and $G=J_1$, hence
the result follows from \cite[Theorem 3.8]{Landrock81}.
Suppose that $b$ is a non-principal block; by
Proposition \ref{KernelOfBlock} we may assume that $b$ is faithful.
Proposition \ref{Noeske} implies that $r=3$, 
$G =$ $Co_3$ and $|{\mathrm{Irr}}_K(G,b)| =$ $8$.
\end{proof}

\section{Finite simple groups of Lie type with exceptional 
Schur multipliers}
\label{excmultsection}

The Schur multipliers of finite groups
of Lie type tend to be `generic' (that is, dependent only on the
series to which the group belongs) except in a few cases of low
rank where they are larger; see \cite[Definition 6.1.3]{GorLyoSol3}. 
We consider in this section 
the groups of Lie type from \cite[Table 6.1.3]{GorLyoSol3} defined 
over a field of odd characteristic.

\begin{Proposition} 
[{\cite[Table 6.1.3, p.313]{GorLyoSol3}}, {\cite{ModAtlas}}]
\label{TableExeptionalSchurMult}
The finite simple group $G$ of Lie type defined over a field
of odd characteristic with exceptional Schur multipliers are
as follows:

{\rm
\begin{center}
\begin{tabular}{c|c|c|c|c}
$G$ & $A_1(9) \cong A_6$ & $^2A_3(3) \cong\PSU_4(3)$ 
    & $B_3(3) \cong P\Omega_7(3)$ & $G_2(3)$ \\
\hline
${\mathcal M}e$  & $3$  & $3$, \ $3$ & $3$ & $3$ \\
$M(G)$           & $6$  & $12$       & $6$ & $3$ \\
\end{tabular} 
\end{center}
}
\noindent
where ${\mathcal{M}e}$ denotes the elementary divisors of the
exceptional parts of the Schur multipliers $M(G)$ of $G$.
\end{Proposition}

\begin{Proposition} \label{SimpleGroupExSchurMult}
If $G$ is a quasi-simple finite group such that
$Z(G)$ has odd order and
$G/Z(G)$ is of Lie type either $A_1(9)$, $^2A_3(3)$, 
$B_3(3)$ or $G_2(3)$,
then $G$ has no $2$-blocks $b$ 
with elementary abelian defect group of order $2^r$,
where $r\geq 3$.
\end{Proposition} 

\begin{proof}
For the isomorphisms
$A_1(9) \cong $ $A_6$ and $^2A_3(3) \cong$ $\PSU_4(3)$ and
$B_3(3) \cong$ $P\Omega_7(3)$ (also denoted $O_7(3)$ in
the Atlas \cite[p. 106]{Atlas}), used already in the
previous Proposition, see \cite[p.8, Table I]{GorLyoSol1}.
If $G/Z(G) \cong$ $A_6$ then $G$ is isomorphic to
$A_6$ or $3.A_6$; in both cases
$G$ has no $2$-blocks with an elementary abelian
defect group of order $2^r$, $r\geq 3$ by \cite{ModAtlas},
or by \ref{altagain} below.
If $G/Z(G)$ is isomorphic to one of
$\PSU_4(3)$, $P\Omega_7(3)$, $G_2(3)$ then again
$G$ has no $2$-blocks with elementary abelian
defect group of order $2^r$, $r\geq 3$, by \cite{ModAtlas}.
\end{proof}

\section{Alternating groups} \label{altsection}

We denote in this section  by $A_n$ the 
alternating group of degree $n$, where $n$ is a positive integer. 

\begin{Proposition} \label{alt}
If $G\cong A_n$ for some $n\geq 5$ then $G$ has no $2$-blocks
with an elementary abelian defect group of order $2^r$,
where $r\geq 3$.
\end{Proposition}

\begin{proof}
By \cite[1.2, 1.3, 1.4, 1.7]{Kesalt}, a $2$-block of an 
alternating group $A_n$ with defect group $P$ is source algebra 
equivalent to a block of an alternating group $A_m$ for some 
$m\leq n$ having $P$ as Sylow $2$-subgroup. But there is no 
alternating group with an elementary abelian Sylow $2$-subgroup 
of order $2^r$, $r\geq 3$.
\end{proof}

\begin{Proposition}
\label{altagain}
If $G$ is $3.A_6$ or $3.A_7$, then $G$ has no $2$-block
with an elementary abelian defect group of order $2^r$,
where $r\geq 3$.
\end{Proposition}

\begin{proof}
This is clear since Sylow $2$-subgroups of $G$ are
dihedral of order $8$, see \cite[p.4, p.10]{Atlas}.
\end{proof}

\section{Finite groups of Lie type in characteristic $2$}

\begin{Proposition} \label{Lietwo}
Let $G$ be a quasi-simple group such that $G/Z(G)$ is a finite
group of Lie type in characteristic $2$. Suppose that $Z(G)$ has
odd order. Let $b$ be a block of $G$ having an elementary
abelian defect group of order $2^r$ for some integer $r\geq 3$. 
Then $G\cong \PSL_2(2^r)$, 
the block $b$ is the principal block of $G$, and 
Alperin's weight conjecture holds for $b$.
\end{Proposition}

\begin{proof}
Consider first the case where $G/Z(G)$ is not isomorphic to
the Tits simple group ${^2{F_4(2)'}}$, $\PSp_4(2)'\cong A_6$,
or $G_2(2)'\cong \PSU_3(3)$. Then, by \cite[Proposition 8.7]{CEKL},
the defect groups of any $2$-block of $G$ are either trivial or
the Sylow $2$-subgroups of $G$. It is well-known that the only
finite groups of Lie type in characteristic $2$ having abelian
Sylow $2$-subgroups are the groups $\PSL_2(2^r)$, and hence
$G\cong \PSL_2(2^r)$, where we use that the Schur multiplier of
$\PSL_2(2^r)$ is trivial. By \cite[A 1.3]{Broue},
the principal block is isotypic to its Brauer correspondent; in
particular, Alperin's weight conjecture holds.
The group ${^2{F_4(2)'}}$ has trivial Schur multiplier
and by the pages on decomposition numbers in the Modular Atlas \cite{ModAtlas},
${^2{F_4(2)'}}$ has three $2$-blocks: the
principal block (of defect $11$) and two defect zero blocks;
${^2{F_4(2)'}}$ has no block with an elementary abelian defect group
of order $2^r$, $r\geq 3$.
The case $\PSp_4(2)\cong A_6$ is already checked in 
\S\ref{excmultsection} and in Proposition \ref{altagain}. Finally, 
$\PSU_3(3)$ has trivial Schur multiplier
and again  by the pages on decomposition numbers in the Modular Atlas 
\cite{ModAtlas}, $\PSU_3(3)$ has 
besides the principal block (of defect $5$) two blocks of
defect zero; in particular, $\PSU_3(3)$ has
no blocks with an elementary abelian defect group of order $2^r$,
$r\geq 3$.
\end{proof}

\begin{Remark}
For the purpose of the proof of \ref{c2c2c2}
one could have excluded the Tits simple group ${^2{F_4(2)'}}$  and $A_6$ 
also by observing
that its order is not divisible by $7$, and hence the inertial
quotient of a hypothetical block with elementary abelian defect group 
of order $8$ can only be trivial or cyclic of order $3$, in which case 
Alperin's weight conjecture holds by Proposition \ref{leqeight} above.
\end{Remark}

\section{Further background results on finite reductive groups} 
\label{redback}

The book \cite{Digne/Michel:1991}, especially Chapters 13 and 14 
is a  useful reference for    the first part of this section.
The following notation will be in effect for this section.
Let $r$ and $\ell$ be disctinct primes and let $q$ be a power of $r$.
Let $\Gb$ a connected reductive 
group over $\bar{\mathbb F}_q$ and $F: \Gb \to$ $\Gb$ a Frobenius 
morphism with respect to an ${\mathbb F}_q$-structure on $\Gb$.  
Let $(\Gb^*, F^*)$ be a pair in duality with $(\Gb, F)$ with respect 
to some choice of an $F$-stable (respectively $F^*$-stable) maximal 
torus  of $\Gb $ (respectively $\Gb^* $) and  with respect to a 
fixed isomorphism  $\bar{\mathbb F}_q ^{\times} \cong$ 
$({\mathbb Q}/{\mathbb Z})_{r'}$ and a fixed embedding   
$\bar{\mathbb F}_q ^{\times} \hookrightarrow$ 
$\bar{\mathbb Q}_\ell^{\times}$. For an $F^*$-stable semi-simple element  
$s$ of $\Gb^*$, we denote by $\CE(\Gb^F, (s) ) \subseteq$ 
$\Irr_{\bar{\mathbb Q}_\ell}(G)$ the subset of characters corresponding  
to the geometric conjugacy class $(s)$ and by $\CE(\Gb^F, [s])$ the  
subset of characters corresponding to the rational conjugacy class  
$[s]$; the   geometric (and  rational) class of the trivial element 
will be just denoted $\CE(\Gb^F, 1)$.   The elements of   
$\CE(\Gb^F, 1)$ are called the unipotent characters of  $\Gb^F$. 
We set $\Cb(s)=$ $C_{\Gb^*}(s)$, $\Cb^{\circ}(s)= $ 
$C_{\Gb^*}(s)^{\circ}$, the connected component of $\Cb(s)$,  
$\bar {\Cb}^{\circ}(s)$ $=\Cb^{\circ}(s)/ Z(\Cb^{\circ}(s))$, 
the quotient of the connected centraliser by its centre, and 
$\Cb^{\circ }(s)' =$ $[\Cb^{\circ}(s), \Cb^{\circ}(s)]$ the derived 
subgroup of the connected centraliser. 
Set $a_s=$ $\frac{ |\Cb(s)^{F^*}|}{|\Cb^{\circ}(s)^{F^*}|}$, 
$\Zb^{\circ}(s)=$ $Z(\Cb^{\circ}(s))$ and 
$z(s)=$ $|\Zb^{\circ}(s)^{F^*}|$. 
For any positive integer $m$ we denote
by $m_+$ the highest power of $2$ dividing $m$.
 
\subsection{Jordan decomposition of characters} 
By the work of Lusztig (\cite[Theorem 4.23]{Luszconn}, 
\cite[Proposition 5.1]{Luszdisconn}, 
see also  \cite[13.24]{Digne/Michel:1991}), there is a  bijection 
between $\CE(\Gb^F, [s])$ and the set 
$\CE(\Cb(s)^{F^*}, 1) $ of unipotent characters of $\Cb(s)^{F^*} $ 
such that if $\chi \in  \CE(\Gb^F, [s])$ corresponds to  $\tau  \in 
\CE(\Cb(s)^{F^*}, 1) $, then

\begin{equation} \label{matchchardeg}
\chi(1)=  \frac {|\Gb^F|_{r'}}{|\Cb(s)^{F^*}|_{r'}}  \tau(1). 
\end{equation}

Here we note that if $\Cb(s) $ is not connected, then 
$\CE(\Cb(s)^{F^*}, 1) $  is  defined to  be   the set  of  irreducible 
characters   of $\Cb (s)^{F^*}$  covering the set 
$\CE(\Cb^{\circ}(s)^{F^*}, 1) $    
of unipotent characters of $\Cb^{\circ}(s)^{F^*} $.  By standard 
Clifford theory, if $\tau$ is an irreducible character of  $\Cb(s)^{F^*}$  
covering  an irreducible character $\lambda$ of $\Cb^{\circ}(s)^{F^*} $, 
then  $\tau(1) =$ $a \lambda (1)$, for an integer $a$ dividing $a(s)$. 
Thus, to each element of $\CE(\Gb^F,(s))$ is associated a
$\Cb(s)^{F^*}$-orbit of  $\CE(\Cb^{\circ}(s)^{F^*}, 1) $ 
 such that 
if $\chi \in $ $\CE(\Gb, (s))$ corresponds to the orbit of $\lambda\in$ 
$\CE(\Cb^{\circ}(s)^{F^*}, 1)$, then

\begin{equation}\label{mathchcharnonconn} 
\chi(1)=\frac {|\Gb^F|_{r'}}{a_{\chi}|\Cb^{\circ}(s)^{F^*}|_{r'}}\lambda(1) 
\end{equation}
for some  integer $a_{\chi}$ dividing  $a_s$.

Restriction induces a 
degree preserving bijection $\lambda \mapsto$ $\lambda'$ between 
the sets $\CE(\Cb^{\circ}(s)^{F^*}, 1) $ 
and $\CE ({{\Cb}^{\circ}(s)'}^{F^*}, 1)$ 
and there is also a degree preserving bijection 
$ \lambda \to \bar \lambda$ between  $ \CE(\Cb^{\circ}(s)^{F^*}, 1) $ 
and $\CE(\bar {\Cb}^{\circ}(s)^{F^*}, 1)$ 
(cf. \cite[Proposition 3.1]{Cabanes/Enguehard:1994}).     
Further, the  group  $\bar{\Cb}^{\circ}(s)^{F^*}  = $
$\prod_{\omega}\bar {\Cb^{\circ}}_{\omega}^{F^*},$
where $\omega $ runs through the  $F$-orbits  of the Dynkin  diagram  
$\Delta_s$ of $\bar{\Cb}^{\circ}(s)$, and for each  $ \omega $,  
$\bar{\Cb}^{\circ}_{\omega }(s) $  is the direct product 
of   subgroups of   $\bar {\Cb}^{\circ}(s) $ corresponding 
to the elements of $\omega $. The elements of 
$\CE(\bar {\Cb}^{\circ}(s)^{F^*}, 1)$ are  products 
$\prod_{\omega}\phi_{\omega}$ 
where  $\phi_{\omega} $ is a unipotent character of  
${\bar\Cb^{\circ}}(s)_{\omega}^{F^*}$  for each $\omega $.
Tracing through the  above  bijections, and noting that  

\begin{equation} \nonumber
|\Cb^{\circ}(s)^{F^*}| =   z(s)
| {\Cb}^{\circ}(s)^{'F^*}|  = z(s)   
|\bar {\Cb}^{\circ}(s)^{F^*}|  \end{equation}  

\noindent
it follows that  if $\chi \in $ $\CE(\Gb, (s))$ corresponds to the orbit of 
$\lambda\in$ 
$\CE(\Cb^{\circ}(s)^{F^*}, 1)$,  and  $\lambda $ corresponds to  
 to the  unipotent character $\lambda' $ of $
{{\Cb}^{\circ}(s)'}^{F^*}$ and to 
the unipotent 
character 
$\bar\lambda:= \prod_{\omega} \phi_{\omega} $ of  
$ \bar {\Cb}^{\circ}(s)^{F^*}$, 
   then

\begin{eqnarray} \label{matchchardegfinal} 
\chi(1) & =  & \frac {|\Gb^F|_{r'}}{z(s)a_{\chi}|{\Cb}^{\circ }(s)'^{F^*}|_{r'}}
\lambda'(1) \nonumber \\
&=& \frac {|\Gb^F|_{r'}}{z(s)a_{\chi}}
\prod_{\omega} \frac{\phi_{\omega}(1)}{|\bar\Cb^{\circ F^*}_{\omega}(s)|_{r'}}
.
\end{eqnarray}

The above correspondences have the following consequence, which we 
record for use in later sections. If $r$ is odd, then 

\begin{eqnarray} \label{matchtwodefect}  2\text{-defect of } \chi &=
&\alpha_{\chi}+\zeta_s + 2\text{-defect of }\lambda' \\
&=&\alpha_{\chi} + \zeta_s +  \sum_{\omega}  (2\text{-defect of }
\phi_{\omega})\end{eqnarray} 

\noindent
where  $ 2^{\zeta_s}= z(s)_+$ and $ 2^{\alpha_{\chi}} =$ $|{a_{\chi}}|_+$.
We note that we get an analogous formula for $\ell$-defects, for any prime 
$\ell$ different from  $r$. 

\subsection{Jordan decomposition of blocks}
As is the case of many sources cited below, we divert in this section
from our previous notation and use 
the prime $\ell$, instead of $p$, for the characteristic of $k$,
which is assumed to be different from the defining characteristic $r$.
We identify without further comment the sets  
$\Irr_{\bar {\Q}_\ell}(\Gb^F) $ and $\Irr_{K} (\Gb^F)$.
Let $t$ be a semi-simple element of $\Gb^{*F^*}$ of order prime to $\ell$
and let  $\CE_\ell(\Gb^F,[t])=\cup_{u} \CE(\Gb^F,[tu]) $, where $u$ runs 
over the $\ell$-elements of $C_{\Gb^*}^{\circ} (t)^{F^*}$.
By \cite[Th\'eor\`eme 2.2]{BroMic89}, $\CE_\ell(\Gb^F,[t])=\cup_{u} \CE(\Gb^F,[tu]) $, 
is a union of  $\ell$-blocks of $\Gb $; a block $b$ of $\Gb^F$ 
is in this union if and only if  
$\CE(\Gb, [t]) \cap  \Irr_K(\Gb^F, b) \ne \emptyset$ 
(cf. \cite[Theorem 3.1]{Hissbasic}). 
In this case, we say that $b$ is in the series $[t]$ or that     
$[t]$ is the semi-simple label of $b$. Blocks with semi-simple label 
$[1]$ are called {\it unipotent}.
Let $\Lb ^*(t)$ be the  (necessarily  $F^*$-stable)
minimal standard  Levi subgroup of $\Gb^* $-containing $\Cb(t)$ (if $\Cb(t)$ 
is not contained in any proper Levi subgroup of  $\Gb^*$, we take for
$\Lb ^*(t)$  the group  $\Gb$ itself) and let $\Lb$   
be an $F$-stable Levi subgroup of $\Gb $ dual to $ \Lb ^*(t)$. 
Let $e_t^{\Gb^F} $  and $e_t^{\Lb^F} $  be the sum of block idempotents of  
$\CO\Gb^F$ and  $\CO\Lb^F$  in the series $[t]$.
Then by Theorem B' of \cite{BonRou},  the algebras
$\CO G^Fe_t^{\Gb^F}$ and $\CO \Lb^F e_t^{\Lb^F}$  
are Morita equivalent.
Further, if $\Cb(t)$ is itself a Levi subgroup of $\Gb^*$, ie. if 
$\Cb(t)=  \Lb ^*(t)$, then  by Theorem 11.8 of \cite{BonRou},   
$\CO \Gb^F e_t^{\Gb^F}$ is Morita equivalent to the sum of unipotent 
blocks  $\CO \Lb^Fe_1^{\Lb^F}$ of $\CO \Lb^F$.  
We record a consequence of these results for classical groups for $\ell=2$.

\begin{Theorem} \label{jordandecom0}   
Suppose that $\ell=2 $ and either $\Gb =\GL_n(\bar{\mathbb F}_q)$ or that 
$\Gb$ is simple of classical type $B$, $C$ or $D$. 
Let $t$ be an odd order 
semi-simple element of $\Gb ^{*F^*}$. Then, $\Cb(t)$ is a Levi subgroup 
of $\Gb^*$.   
Let $\Lb$ be an $F$-stable Levi subgroup of $\Gb$ in duality with $\Cb(t)$
as above. Then $e_t^{\Gb^F}$ is a block of $\CO\Gb^F$, and $e_1^{\Lb^F}$ 
is the principal block of $\CO\Lb^F$. The block algebras
$\CO\Gb^F e_t^{\Gb^F}$ and $\CO\Lb^Fe_1^{\Lb^F}$ are Morita 
equivalent and a Sylow $2$-subgroup of $\Lb^F$ is a defect group 
of $\CO \Gb^Fe_t^{\Gb^F}$.
\end{Theorem}

\begin{proof} The element $t$ has odd order, whereas 
$Z(\Gb)/Z^{\circ}(\Gb) $ is a $2$-group,  
hence  $\Cb(t)$ is   connected.  The prime $2$ is the only bad prime  for $\Gb$
(if $\Gb$ is   a general linear group, then  all primes are good for  $\Gb$).
Thus,  the order of $t$  is not divisble by any bad prime,  which means that 
$\Cb^{\circ}(t)$ is a Levi subgroup of $\Gb^*$.  
Thus  by the Bonnaf{\'e}-Rouquier 
theorem, $\CO\Gb^F e_t^{\Gb^F}$ and $\CO\Lb^Fe_1^{\Lb^F}$ 
are Morita equivalent.
Now the components of $\Lb$ are all of classical types $A$, $B$, $C$ or $D$.   
Hence by \cite[Theorem 13]{CabEng93}, the principal block of $\CO\Lb^F$ is 
the unique unipotent $2$-block of $\CO\Lb^F$, and the Morita equivalence  
implies that $e_t^{\Gb^F} $ is a block of $\CO\Gb^F$. The assertion on 
defect groups is in \cite[Proposition 1.5(ii), (iii)]{Enguehard-pre}.  
\end{proof}

\begin{Corollary}\label{jordandecomabel}   
Suppose that $\ell=2$ and either $\Gb =$ $\GL_n(\bar{\mathbb F}_q)$  
or that $\Gb$ is simple of classical type $B$, $C$ or $D$.  Let 
$b$ be a block of $\CO \Gb^F$. If $b$ has abelian defect groups, 
then $\CO \Gb^Fb$ is nilpotent.
\end{Corollary}

\begin{proof}  Let $t$ be the semi-simple label of $b$ and suppose that   
$b$ has abelian defect groups. By the theorem, $ b=e_t^{\Gb^F}$ and 
the Sylow $2$-subgroups of   $\Lb^F$ are abelian. 
Suppose that  $[\Lb, \Lb] \ne 1 $. If $\Lb$ has a component of type 
  different from 
$A_1$, then $[\Lb, \Lb]^F $ contains a subquotient  isomorphic to  
$\SL_2(q')$  for some power $q'$ of $q$ 
(see Theorem 3.2.8 of \cite{GorLyoSol3}).
If  all components of  $\Lb $  are of type $A_1$, then $[\Lb, \Lb]^F $ is 
a commuting product of finite  special linear  and projective 
general linear groups of degree $2$. But the Sylow $2$-subgroups
of $\SL_2(q')$ are quaternion  and those of $\PGL_2(q')$  are dihedral  of 
order at least $8$ for any odd prime power $q'$ (see \ref{classical2sylow}
below), a contradiction.
Hence, $\Lb$ and therefore $\Lb^F$ is an abelian group. In particular, 
any block of $\CO\Lb ^F$ is nilpotent. 
Since nilpotent blocks with abelian defect groups 
are precisely the blocks with a symmetric centre 
(cf. \cite[Theorems 3 and 5]{OkuTsu} and \cite{MuellerW}), 
any block Morita equivalent to a block of 
$\CO\Lb^F$ is nilpotent with an abelian defect group, whence the result. 
Alternatively, Morita equivalences of blocks preserve nilpotence by 
\cite[Theorem 8.2]{Puigbook}.
\end{proof}

\section{On the $2$-local structure  of finite  classical groups}\label{2local}

Let $n$ be a natural number, $q$ an odd prime power and let $L$   
denote one of the groups  $\GL_n(q)$, $\GU_n(q)$, $\GO_{2n+1}(q)$,  
$\Sp_{2n}(q)$,  $\GO_{2n}^+(q)$,  or $\GO_{2n}^{-}(q)$.   
Let $Z$ be a central subgroup of $L$ contained in $[L,L]$ and 
set $G=$ $[L,L]/Z $. 
We gather together a few well-known facts on the Sylow $2$-structure of  
the groups $L$ and $G$.

\begin{Lemma}  \label{classical2sylow} With the notation above,  

\smallskip \noindent
(i) If $n \geq 3 $, then   
the Sylow $2$-subgroups of $G$  (and hence of $[L,L]$ and $L$) 
are non-abelian.
 
\smallskip\noindent
(ii) If  $n =2 $, the Sylow $2$-subgroups of $[L,L] $ are non-abelian.
Further, if $n=2 $ and   $ L$ is not one of  $\GL_2(q)$,  
$ q\equiv \pm 3 \ (\mod \ 8 )$, $\GU_2(q) $, $q\equiv   \pm 3 \ (\mod  \ 8)$,  
or $ \GO_{4}^+(q)$, 
$ q \equiv \pm 3 \ (\mod \ 8 )$, then  the Sylow 
$2$-subgroups of $G$ are non-abelian.  

\smallskip\noindent
(iii)  If $L$   is one of $\GL_2(q) $,    $ \GU_2(q)$  or  $ \Sp_2(q) $
then the Sylow $2$-subgroups of $[L, L]$ 
are generalised quaternion groups. They have order at least $16$ if 
$ q \equiv \pm 1 \, (\mod\ 8 )$ and in this case the Sylow 
$2$-subgroups of $G$  are non-abelian. If 
$q \equiv \pm 3\, (\mod \  8) $, then the Sylow $2$-subgroups  of $[L,L]$
have order $8$ and the Sylow $2$-subgroups of $G$  are Klein $4$-groups.
The Sylow $2$-subgroups of $\PGL_2(q)$ are dihedral of order at least $8$.

\smallskip\noindent
(iv)  If $L$ is one of $\GO_{2}^+(q) $, respectively  
$\GO_{2}^-(q) $, then $L$  is a dihedral group of order $2(q-1)$, 
respectively $2(q+1)$.

\smallskip\noindent
(v)   If $L =\GO_3(q)$ and if $q \equiv 1 \ (\mod \ 4) $   then  the 
Sylow $2$-subgroups of $L$  are isomorphic 
to the  direct product of  a cyclic group of order $2$  with a   
Sylow $2$-subgroup of $\GO_{2}^+(q) $ 
and  the Sylow $2$-subgroups of $\SO_3(q)$ are isomorphic to the 
Sylow $2$-subgroups of  $\GO_{2}^+(q) $. If $L =\GO_3(q)$ and 
$q \equiv 3 \ (\mod \ 4) $   then  the 
Sylow $2$-subgroups of $L$  are isomorphic 
to the  direct product of  a cyclic group of order $2$  with a   
Sylow $2$-subgroup of $\GO_{2}^-(q) $ 
and  the Sylow $2$-subgroups of $\SO_3(q)$ are isomorphic to the 
Sylow $2$-subgroups of  $\GO_{2}^-(q) $.
\end{Lemma}

\begin{proof}  Statements (iii), (iv), (v) can be found in 
\cite{CarterFong}.  
 The only simple groups with abelian Sylow $2$-subgroups are 
$\PSL_2(q')$, $q' \cong \pm 3 \   (\mod\ 8)$, $\PSL_2(2^a)$, 
$\,^2G_2(q')$, $q'= 3^{2u+1}$, $u \geq 1 $, $\,^2G_2(3)' \cong$ $\PSL_2(8)$   
or $J_1$ (cf. \cite{Walter}, \cite{Bender}).
Also, if $n\geq 2 $, then unless $L$ is one of 
$\GO_{4}^+(q)$, $\GL_2(2)$, $\GL_2(3)$ or $\GU_2(2) $,
the groups $G$ are all quasi-simple.  Statement (i) 
and the second assertion of (ii) are immediate 
from this. If $L=$ $\GL_2(q)$, or 
$L =$ $\GU_2(q)$, the second assertion of (ii) is immediate from (iii).
Now consider $L=\GO_{4}^+(q)$. Then $L$ contains a subgroup isomorphic to  
$\GL_2(q)$ or to $\GU_2(q)$ (as a centraliser of a semi-simple element), 
hence  $[L,L]=$ $\Omega_{4}^+(q)$ 
contains a subgroup isomorphic to $\SL_2(q)$ or to $\SU_2(q)$ and 
$\SL_2(q)$ and $\SU_2(q)$ have non-abelian Sylow $2$-subgroups.
\end{proof}

In the next sections, we will analyse closely the structure of   
centralisers of semi-simple elements in classical groups. The following 
elementary lemma  will be useful in this context.
In  what follows,  $\GO^{\pm}_0(q) $ are to 
be interpreted as the  trivial group,   and $\GO_1(q) $  as a cyclic group of 
order $2$. Also note that the center of $L$ is a $2$-group.

\begin{Lemma} \label{classicalcent2sylow}  
Let $t\geq 0$, let $d_i $, $m_i$, $1\leq i \leq t$, be positive 
integers and let  $m_0$ be a non-negative integer.  Let
$$H = \prod_{0\leq i\leq  t} H_i $$ be a subgroup of $L$ such that 
$H_0 $ is one of the groups $\Sp_{2m_0}(q)$,  $\GO_{2m_0+1}(q)$ or  
$\GO_{2m_0}^{\pm}(q)$ and $H_i $ is isomorphic to $\GL_{m_i}(q^{d_i})$ or  
$\GU_{m_i}(q^{d_i})$ for $ 1\leq i \leq t $. Let $Z$ be a central 
subgroup of $L$ contained in $H$ such that if  the above decomposition 
of $H$  has more than one  non-trivial factor, then $Z \cap H_i =$ $1$ 
for all $i$, $0\leq i \leq t $. Let $T$ be  a Sylow $2$-subgroup of $H$ 
and set $P=$ $(T \cap [L, L]Z)/Z$. Suppose that  $P$ is abelian. Then,

\smallskip\noindent
(i) $m_i \leq 2 $, $  0 \leq i \leq t $.

\smallskip\noindent
(ii)   If $m_0 = 2 $, then  $H=H_0 = \GO_{4}^{+} (q) $.

\smallskip\noindent 
(iii) If $m_i =2 $ for some $i\geq 1$, 
then  $t=1 $ and $m=0 $, that is  $H=H_1$.
Further,  $d_1$ is odd and   
$ q\equiv  \pm 3   \  (\mod \ 8) $. 

\smallskip\noindent
(iv) If $H_0=\Sp_{2n}(q)$ and if $m_0 \ne 0 $, then  $t=0$,  that is $H=H_0$.
\end{Lemma}

\begin{proof}   For $ 0 \leq i \leq t $,
$[H_i, H_i]  \leq   H \cap [L,L] $, and hence
$$ [H_i, H_i]/[H_i, H_i]\cap Z \cong    [H_i, H_i]Z/Z 
\leq  (H \cap [L,L]Z)/Z . $$
Suppose  first that   two of the  factors of $H$ are non-trivial, say 
$H_i $ and $ H_j $, $ i\ne j$. Then,
 by assumption $[H_i, H_i] \cap Z =1 $. 
It follows from the above 
that  $[H_i, H_i] $  is a subgroup of $(H \cap [L,L])/Z $.   
In particular,  the  Sylow $2$-subgroups of $[H_i, H_i] $   are abelian.
But by Lemma \ref{classical2sylow}, $[H_i, H_i] $ has non-abelian Sylow 
$2$-subgroups if $m_i \geq 2 $. This proves  the first assertions  of (ii) and 
(iii).

Now supppose that some $m_i \geq 2 $. By what has just been proved,   
either $t=0 $  or $t= 1 $ and  $m_0=0 $. If $t=0 $, then $H=H_0$ is 
an orthogonal or symplectic group of dimension $m_0$.  By Lemma
\ref{classical2sylow}, it follows that   $ H_0= \GO_4^+(q)$, proving (ii). 
If $t=1$ and $m_0=0$, then  $[H, H]/Z([H, H])$ 
is a projective special linear or unitary group, 
and by Lemma \ref{classical2sylow} (i), $ m_1 =2 $, 
$H \cong \GL_{2}(q^{d_1})  $,  and $ q^{d_1} \equiv \pm 3 \ (\mod  \ 8 ) $. 
The last can only hold if $d_1 $ is odd and  $ q  \equiv \pm 3 \ (\mod \ 8 ) $.
This proves (iii).
The proof of (iv) is similar, using the fact that   symplectic groups  
$\Sp_{2n}(q)$ are perfect for $ n\geq 2 $.
\end{proof}

\section{Type A in odd characteristic}

By results of Blau and Ellers in \cite{BlEl}, Brauer's height
zero conjecture holds for all blocks in non-defining characteristic
of quasi-simple groups of type $A$ and ${^2{A}}$, and hence
so does Alperin's weight conjecture for all blocks of these
groups in odd characteristic with an elementary abelian defect
group of order $8$, by Landrock's results quoted in Proposition
\ref{leqeight}.
This proves in particular Theorem \ref{c2c2c2} for these groups.
For future reference, and using methods similar to those
in \cite{BlEl}, we prove in this and the following section that
elementary abelian $2$-defect groups of odd rank at least three
do not occur in type $A$ and type ${^2{A}}$, and that blocks
with an elementary abelian $2$-defect group of even rank
at least four of these groups satisfy Alperin's weight conjecture. 
For the remainder of the paper, we assume $p=2$.

\begin{Theorem} \label{typeA}
Let $G$ be a quasi-simple finite group such that
$Z(G)$ has odd order and such that $G/Z(G)\cong \PSL_n(q)$ for
some positive integer $n$ and some odd prime power $q$. 
Let $b$ be a block of $kG$ with an elementary abelian defect
group of order $2^r$ for some integer $r\geq 3$. 
Then $r$ is even and $b$ satisfies Alperin's weight conjecture.
\end{Theorem}

\begin{Lemma} \label{determinantsurjective}
Let $n$, $m$, $d$ be positive integers such that $n=md$.
Consider $\GL_m(q^d)$ as subgroup of $\GL_n(q)$ through
some $\F_q$-decomposition $(\F_q)^n\cong (\F_{q^d})^m$.
Denote by $\det$ the determinant function on $\GL_n(q)$.
Then there is an element $x\in \GL_m(q^d)$ such that $\det(x)$ 
has order $q-1$.
\end{Lemma}

\begin{proof}
Let $\lambda$ be a generator of $\F_{q^d}^\times$ and $f\in\F_q[X]$ the
minimal polynomial of $\lambda$ over $\F_q$. Then $f$ has degree $d$,
and the roots of $f$ are $\lambda$, $\lambda^q$, $\lambda^{q2}$,...,
$\lambda^{q^{d-1}}$.  Let $y\in \GL_d(q)$ with minimal polyomial $f$.
Define $x\in\GL_n(q)$ via $d\times d$-block diagonal matrices
where the first block is $y$ and the remaining $m$ blocks are the
identity matrices $\Id_d$. 
Then, in $\GL_n(\bar\F_q)$, the element $x$ is conjugate to a diagonal
matrix whose diagonal entries are
$\lambda$, $\lambda^q$, $\lambda^{q2}$,...,$\lambda^{q^{d-1}}$, 
$1$, $1$, ...,$1$.
Thus the determinant of $x$ is
$\det(x) = $ $\lambda\cdot\lambda^q\cdots\lambda^{q^{d-1}} =$ 
$\lambda^{\frac{q^d-1}{q-1}}$.
Since $\lambda$ has order $q^d-1$ it follows that $\det(x)$ has order $q-1$.
\end{proof}

\begin{Lemma} \label{orderpreserve}
Let $n$, $m$, $d$ be positive integers such that $n=md$.
Consider $\GL_m(q^d)$ as subgroup of $\GL_n(q)$ through
some $\F_q$-decomposition $(\F_q)^n\cong (\F_{q^d})^m$.
If $m\geq 2$ and $q^d\equiv\ 1\ (\mod\ 4)$ then
there is an element $y\in\GL_m(q^d)$ of order $4$ such that
the image of $y$ in $\PGL_n(q)$ has order $4$.
If moreover $m\geq 3$ we can choose such an element $y$ in $\SL_n(q)$.
\end{Lemma}

\begin{proof}
Since $m\geq 2$, the group $\GL_m(q^d)$ contains a subgroup
isomorphic to $\F_{q^d}^\times\times\F_{q^d}^\times$; choose 
$y=(y_1,1)$ in this subgroup, where $y_1\in\F_{q^d}^\times$ has
order $4$. Then $y$ has an eigenvalue $1$, hence if some
power $y^r$ is a scalar multiple of $\Id_n$ then $y^r=\Id_n$, 
which shows that the order of $y$ remains unchanged upon taking
its image in $\PGL_n(q)$.
If $m\geq 3$ then $\GL_m(q^d)$ contains a subgroup
isomorphic to $(\F_{q^d}^{\times})^3$; choose $y = (y_1,1,y_3)$
with $y_3$ such that $\det(y_3)=\det(y_1)^{-1}$, which is
possible thanks to Lemma \ref{determinantsurjective},
and as before, $y$ has the required properties. 
\end{proof}

Since the case $\PSL_2(9)\cong A_6$ is dealt with in
\S\ref{excmultsection}
in order to prove Theorem \ref{typeA} we
may assume that $G = \SL_n(q)/Z_+$, where $Z_+$ is the
Sylow $2$-subgroup of $Z(\SL_n(q))$. Note that
$|Z_+|$ is equal to the $2$-part $(n,q-1)_+$ of $(n,q-1)$.

Let $b$ be a block of $kG$ and denote by $P$ a defect group
of $b$. Since $\PSL_n(q)\cong G/Z_{-}$, where $Z_{-}$ is the
complement of $Z_+$ in $Z(SL_n(q))$ identified to its image in $G$, 
the image of $P$ in $\PSL_n(q)$ is isomorphic to $P$.
Since $Z_+$ is a central $2$-subgroup, $b$ is the image
of a unique block $\tilde{b}$ of $k\SL_n(q)$, and the inverse
image $\tilde{P}$ of $P$ in $\SL_n(q)$ is a defect group of
$\tilde{b}$. Let $d$ be a block of $k\GL_n(q)$ covering $\tilde{b}$
with a defect group $T$ such that $T\cap \SL_n(q) = \tilde{P}$.
By \cite[Proposition 6.3]{CEKL} the block $\tilde{b}$ is stable
under the $2$-part of $\GL_n(q)/\SL_n(q)$, and hence $T/\tilde{P}$
is cyclic of order the $2$-part $(q-1)_+$ of $q-1$.
By \cite[Th\'eor\`eme 3.3]{brou:localgl},     
there exists a semi-simple element 
$s$ of odd order  in $\GL_n(q)$ such that $T$  is a 
Sylow $2$-subgroup of  $C_{\GL_n(q)}(s) $ (this  can also be seen 
as a consequence  of the Jordan decomposition of Theorem \ref{jordandecom0}).
Further there exist positive 
integers $ m_i, d_i $, $1\leq i\leq t $ such that   
$$n=\sum_{1\leq i\leq t}m_id_i, $$ and  setting $H_i =\GL_{m_i}(q^{d_i})$, 
there is  a decomposition
$$C_{\GL_n(q)}(s)  \equiv \prod_{1\leq i\leq t}H_i$$
corresponding to a  subspace decomposition of the  underlying  $\F_q$-vector 
vector space as  isotypic $\F_q[s]$-modules. 
In particular, $H_i = \GL_{m_i}(q^{d_i}) $ is a subgroup of  
$ \GL_{m_i d_i}(q) $ through some 
$\F_q$-decomposition $(\F_q)^{m_id_i}\cong (\F_{q^{d_i}})^{m_i}$.

\begin{Lemma} \label{tgeqthree}  With the  notation above,
suppose that $P$ is elementary abelian and that $t\geq 3$.
Then the following holds.

\smallskip\noindent
(i) $q \equiv 3\ (\mod\ 4)$.

\smallskip\noindent
(ii) $m_i = 1$ for $1\leq i\leq t$.

\smallskip\noindent
(iii) $d_i$ is odd for $1\leq i\leq t$.

\smallskip\noindent
(iv) If $t$ is even then $|P| = 2^{t-2}$,
and if $t$ is odd then 
$|P| = 2^{t-1}$; in particular, the $2$-rank of $P$ is even.

\smallskip\noindent
(v) The block $d$ of $\GL_n(q)$ is nilpotent with an
elementary abelian defect group $T$ of order $2^t$. 

\smallskip\noindent
(vi) The block $b$ of $kG$ satisfies Alperin's weight conjecture.
\end{Lemma}

\begin{proof}
Clearly, $P = (T\cap \SL_n(q))/Z_+ $ and $Z_+ \cap H_i =1$ for any 
$i$ such that $H_i \ne C_L(s)$. Thus
Lemma \ref{classicalcent2sylow} (ii) and (iii) apply with $L=\GL_n(q)$, 
$H= C_L(s) $ and $ Z=Z_+ $, and assertion (ii) is immediate.  
If $q\equiv\ 1\ (\mod\ 4)$ or if $d_1$ is even, then 
the group $H_1=\GL_{1}(q^{d_1})$ contains an element $y_1$ of order $4$.
By Lemma \ref{determinantsurjective} the group $H_2=\GL_{1}(q^{d_2})$
contains a $2$-element $y_2$ such that $\det(y_2) = \det(y_1)^{-1}$.
Thus $x= y_1y_2 \in$ $\SL_n(q)=$ $[L,L]$.
Since $t \geq 3$,  we have $ H_1 \times H_2 \cap Z_{+} =1 $, 
and it follows that  the image  $xZ_{+}$ of $x$ in $P$ has order $4$, 
a contradiction. Thus (i) and (iii)  hold.
Hence $T$ is a Sylow $2$-subgroup of
$\prod_{i=1}^t\ \F_{q^{d_i}}^\times$. Since the $d_i$ are odd 
and $q\equiv\ 3\ (\mod\ 4)$ this implies that $T$ is elementary
abelian of rank $t$, and hence $\tilde{P}$ is elementary
abelian of rank $t-1$. If $n$ is odd then $P\cong$ $\tilde{P}$ and
since $n=\sum_{i=1}^t\ d_i$ and the $d_i$ are odd it follows that
$t$ is odd, hence $|P| = 2^{t-1}$. If $n$ is even then
$|P| = \frac{|\tilde{P}|}{2}$ and since $n=\sum_{i=1}^t\ d_i$ and 
the $d_i$ are odd it follows that $t$ is even and 
$|P|=2^{t-2}$, which proves (iv). Since $T$ is clearly abelian,
(v) is immediate from  Corollary \ref{jordandecomabel}.
Statement (vi) follows from (v) and Proposition \ref{nilextremark}.
\end{proof}

\begin{Lemma}\label{teqtwo}
Suppose that $P$ is elementary abelian and that
$t = 2$. Then $m_1 =m_2 =1 $, $|P|\leq 4 $   
and $\tilde P$ contains an element of order 
$ max((q^{d_1} -1)_+, (q^{d_2}-1)_+) $.  
\end{Lemma}

\begin{proof}  
The fact that $m_1 = m_2 = 1 $ is a consequence of 
Lemma \ref{classicalcent2sylow}. So $H_i \cong \GL_1 (q^{d_i}) $. 
By Lemma \ref{determinantsurjective}, it follows that 
$ \tilde P =T\cap \SL_n(q)$ contains  elements of order 
$ max((q^{d_1} -1)_+, (q^{d_2} -1)_+) $. 
\end{proof}

\begin{Lemma}\label{teqone}
Suppose that $P$ is elementary abelian and that
$t = 1$. Then $m_1 \leq 2 $. 

If $m_1 = 2 $, then $d_1$ is odd, $n =2d_1 $,  $ |P|=4 $ and   
$ \tilde P $  is a quaternion group of order $8$.

If  $m_1=1 $, then $T$, $\tilde P$ and 
$P$ are  cyclic, $ |P|= 2 $  if and only if 
$n$ is even, and either $ q \equiv 3 \ (\mod \ 4 )$ and $n_+ \leq (q-1)_+ $
or $ q \equiv 1 \ (\mod \ 4 )$  and $n_+ = 2(q -1)_+ $.
\end{Lemma}

\begin{proof} 
By Lemma \ref{classicalcent2sylow}, 
$m_1 \leq 2 $. Suppose that  $m_1=1 $, so $T$ is a Sylow 
$2$-subgroup of $\GL_1(q^n)$ and in particular,  
$T$ is a cyclic group of order $(q^n-1)_+$.   
It follows from Lemma  \ref{determinantsurjective} that 
$ \tilde P =P \cap \SL_n(q) $  is cyclic of order 
$\frac{(q^n-1)_+}{(q-1)_+} $, and hence  $P$ is cyclic of order
$ \frac{(q^n-1)_+}{(q-1)_+ |Z_+|}$. 
Now, $|Z_+|=  min(n_{+}, (q-1)_+) $. If $n$ is odd then $(q^n-1)_+ =$
$(q-1)_{+}$ and if $n$ is even then $(q^n-1)_+ =$ 
$\frac{(q^2-1)_{+}n_+}{2}$.  The statement of the lemma for the case 
$m_1=1 $ follows by an easy calculation.
Now suppose $m_1=2 $. Then by Lemma \ref{classicalcent2sylow}, (ii)
we have $C_{\GL_n(q)}(s) =$ $\GL_2(q^{d_i})$, 
$q \equiv \pm 3 \ (\mod \ 8) $, $d_i $ is odd and $n=2d_i $.  
With this arithmetic, one sees easily that $|P|=4$ and 
$|\tilde P|=$ $8$. Finally, $\tilde P$  is quaternion since  
it contains  a subgroup isomorphic to the Sylow $2$-subgroups of 
$\SL_2(q^{d_i}) $, which by Lemma \ref{classical2sylow}  are quaternion.
\end{proof}

\begin{proof}[Proof of Theorem \ref{typeA}]
If $t\geq 3$, Lemma \ref{tgeqthree} shows that the $2$-rank of
$P$ is at least $4$ and even and that $b$ satisfies Alperin's weight
conjecture. If $t\leq 2$, Lemma \ref{teqtwo} 
and Lemma \ref{teqone} show that the rank of $P$ is at most $2$.
\end{proof}

We also note the following.

\begin{Lemma} \label{psl6cyclic} Suppose that  
$n= 2m$, $m \geq 1$,  and $P$ is elementary abelian of order $2$ or $4$. 
Then   the inverse image of  
$P$ in a non-split  central extension $2.G$ has an element of order $4$  unless
$t=4$,  $ q \equiv   3 (\mod  \ 4)$, and  $d_i$  is odd  for 
$1 \leq i \leq 4 $. In particular, 
if $P$ has order $2$, then   the inverse image has order $4$.
\end{Lemma}

\begin{proof}   
The central extension $2.G $ may be assumed to be a central quotient 
of $\SL_n(q)$ and  the inverse image, $P_0$ of $P$ in $2.G$ is a 
quotient of $ \tilde P $ by  a cyclic (central) group of order 
$\frac{1}{2} |Z_+|$. We will show that  unless we are in the 
exceptional case above, that  either $\tilde P$ is cyclic or that 
$\tilde P_0$ contains an element of order $2|Z|_+$.
If $t=1 $ and $m_1=1$, then $C_L(s)$, and hence $\tilde P$ is cyclic. 
So certainly  $P_0$ is cyclic.  If $t=1 $ and $m_1=2 $, then   
$[H_2,H_2] \leq $ $C_L(s) \cap [L, L]$ is a special linear group  of 
dimension $2$ and  hence $\tilde P  \cap [H_2,H_2]$ is a quaternion 
group of order $q^{2d}-1 $. In particular, $\tilde P $ contains an 
element of order $\frac{1}{2}(q^{2d}-1)$. Since $ |Z|_+ \leq$ 
$(q-1)_+$, if $d$ is even, then $\tilde P$ contains an element of 
order at least $2|Z|_+$. So we may assume that $d$ is odd. Then 
$|Z|_+ =$ $2$, and $\frac{1}{2}(q^{2d}-1) \geq$ $4 = 2 |Z|_+$.
Now suppose $t=2 $. So, $m_1=$ $m_2 =$ $1$ and $n=$ $d_1 +d_2$. Thus,
$T=$ $T_1 \times T_2$, with $H_i$ a cyclic group of order 
$(q^{d_i}-1)_+ $,  $i=1, 2$. We assume without loss 
of generality 
that $|T_1| \geq$ $|T_2|$. By  Lemma \ref{determinantsurjective},  
it is easy to see that  $\tilde P=$ $T\cap  \SL_n(q)$ is 
a direct product $Q_1 \times Q_2$ such that $Q_1$ is cyclic of order 
$|T_1|$ and $Q_2$ is cyclic of order $\frac{|T_2|}{ (q-1)_+}$. 
If $d_1$ is even, then $|T_1| \geq$ $2 (q-1)_+ \geq 2 |Z|_+$.  
If $d_1 $ is odd, then  $d_2 $ is also odd (as $n$ is even),  
hence  $T_2=$ $1$, that is $\tilde T$  is cyclic. 
Finally, suppose that $t\geq 3 $. By Lemma \ref{tgeqthree},   
$P$ is  not of order $2$, and $P$ is of order $4$ if and only  
if $t=4 $, $q \equiv 3 (\mod \ 4)$, and  $d_i$  is odd  for 
$1 \leq i \leq 4 $. Note that  in this case, $\tilde P$ is 
elementary abelian of order $8$.
\end{proof}

\section{Type $\,^2A$ in odd characteristic}

We show that type $\, ^2A$ yields no blocks
with elementary abelian defect groups of order $8$; in fact,
more generally, we have the following result:

\begin{Theorem} \label{type2A}
Let $G$ be a quasi-simple finite group such that
$Z(G)$ has odd order and such that $G/Z(G)\cong \PSU_n(q)$ for
some positive integer $n$ and some odd prime power $q$. 
Let $b$ be a block of $kG$ with an elementary abelian defect
group of order $2^r$ for some integer $r\geq 3$. 
Then $r$ is even and $b$ satisfies Alperin's weight conjecture.
\end{Theorem}

The proof of this  follows the same lines as the untwisted case. We give  
details for the convenience of the reader.
We single out two elementary observations which we will use
in the proof below:

\begin{Lemma} \label{2determinantsurjective}
Let $n$, $m$, $d$ be positive integers   and denote by $\det$ 
the determinant function on $\GL_n(\bar\F_q)$.

\smallskip\noindent (i) 
Suppose that  $n \geq 2d$. Consider  the inclusions 
$\GL_1(q^{2d}) \leq  \GL_d(q^2) \leq  \GU_{2d}(q) \leq \GU_{n}(q)$, 
where $\GL_1(q^{2d})$  is a subgroup of $\GL_d(q^2)$ through  
some $\F_{q^2}$-vector space isomorphism $(\F_{q^2}) ^d\cong$ $\F_{q^{2d}} $, 
$GL_d(q^2)$ is a subgroup of $\GU_{2d}(q)$ through some 
$\F_{q^2} $-vector space   embedding  $ (\F_{q^2})^d \hookrightarrow 
(\F_{q^2})^{d}  \oplus (\F_{q^2})^{d}  $ of the 
form $\lambda \to$ $\lambda + \lambda^{-q} $ and $\GU_{2d}(q) $ 
is a subgroup of $\GU_{n}(q) $  through   some decomposition 
$(\F_{q^2})^n \cong  (\F_{q^2})^{2d} \oplus  (\F_{q^2})^{n-2d} $.
There is an element $x\in \GL_1(q^{2d})$ such that $\det(x)$ 
has order $q +1$.

\smallskip\noindent (ii) 
Suppose that $d$ is odd and that $ n \geq d $. Consider the inclusions
$\GU_1(q^{d}) \leq \GU_d(q) \leq  \GU_n(q)$,
where $\GU_1(q^d)$  is a subgroup of $\GU_d(q)$ 
through an  irreducible unitary representation of   $\GU_1(q^d)$ on a 
$d$-dimensional  $\F_{q^2} $-space, and where $\GU_d(q) $ is a  subgroup of 
$ \GU_n(q) $ through   some decomposition 
$(\F_{q^2})^n \cong  (\F_{q^2})^{2d} \oplus  (\F_{q^2})^{n-2d} $.
There is an element $x\in \GL_1(q^{2d})$ such that $\det(x)$ 
has order $q +1$.
\end{Lemma}

\begin{proof} (i)
Let $\lambda$ be a generator of $\F_{q^{2d}}^\times$ and $f\in\F_{q^2}[X]$ 
the minimal polynomial of $\lambda$ over $\F_{q^2}$. Then $f$ has degree $d$,
and the roots of $f$ are $\lambda$, $\lambda^{q^2}$, $\lambda^{q^4}$,...,
$\lambda^{q^{2(d-1)}}$.  Let $x\in \GL_d(q^2)$ with minimal polyomial $f$.
Then, in $\GL_{n}(\bar\F_q)$, the element $x$ is conjugate to a diagonal
matrix whose diagonal entries are
$\lambda$, $\lambda^{q^2}$, $\lambda^{q^4}$,...,$\lambda^{q^{2(d-1)}}$, 
$  \lambda^{-q}$, $ (\lambda^{q^2})^{-q}$, 
$(\lambda^{q^4})^{-q}$,...,$ (\lambda^{q^{2(d-1})})^{-q}$, $1 \cdots 1 $. 
Thus the determinant of $x$  is $\det(x)=$ $a a^{-q} $, where 
$a =$ $\lambda\cdot\lambda^{q^2}\cdots\lambda^{q^{2(d-1)}} =$ 
$\lambda^{\frac{q^{2d}-1}{q^2-1}}$.
Since $\lambda$ has order $q^{2d}-1$ it follows that $\det(x)$ has order $q+1$.

(ii) Let $\lambda $ be an element of  order $ q^{d+1} $ in 
$\F_{q^{2d}}^\times$  and  $f\in\F_{q^2}[X]$ 
the minimal polynomial of $\lambda$ over $\F_{q^2}$. Since $d$ is odd,  
$f$ has degree $d$.  Let $x\in \GU_d(q)$ with minimal polyomial $f$.
In   $\GL_{n}(\bar\F)$, the element $x$ is conjugate to a diagonal
matrix whose diagonal entries are
$\lambda$, $\lambda^{q^2}$, $\lambda^{q^4}$,...,$\lambda^{q^{2(d-1)}}$.
Thus the determinant of $x$  is
 $$\det(x)=   \lambda\cdot\lambda^{q^2}\cdots\lambda^{q^{2(d-1)}} = 
\lambda^{\frac{q^{2d}-1}{q^2-1}}.$$ 
Since $\lambda$ has order $q^{d} +1$  and  since  $d$ being odd, the integers, 
$ \frac{q^d-1}{q-1} $ and  $ q^d + 1 $ are relatively prime,  it follows that 
$\det(x)$ has order $q+1$.
\end{proof}

We turn now towards the proof of Theorem \ref{type2A}.
Since the case $\PSU_4(3)$ is dealt with in
\S\ref{excmultsection}
in order to prove Theorem \ref{typeA} we
may assume that $G = \SU_n(q)/Z_+$, where $Z_+$ is the
Sylow $2$-subgroup of $Z(\SU_n(q))$. Note that
$|Z_+|$ is equal to the $2$-part $(n,q+1)_+$ of $(n,q+1)$.
Let $b$ be a block of $kG$ and denote by $P$ a defect group
of $b$. Since $\PSU_n(q)\cong G/Z_{-}$, where $Z_{-}$ is the
complement of $C_+$ in $Z(SL_n(q))$ identified to its image in $G$, 
the image of $P$ in $\PSU_n(q)$ is isomorphic to $P$.
Since $Z_+$ is a central $2$-subgroup, $b$ is the image
of a unique block $\tilde{b}$ of $k\SU_n(q)$, and the inverse
image $\tilde{P}$ of $P$ in $\SU_n(q)$ is a defect group of
$\tilde{b}$. Let $d$ be a block of $k\GU_n(q)$ covering $\tilde{b}$
with a defect group $T$ such that $T\cap \SU_n(q) = \tilde{P}$.
By \cite[Proposition 6.3]{CEKL} the block $\tilde{b}$ is stable
under the $2$-part of $\GU_n(q)/\SU_n(q)$, and hence $T/\tilde{P}$
is cyclic of order the $2$-part $(q+1)_+$ of $q+1$.
By Fong-Srinivasan \cite{FoSri}, $T$ is isomorphic to a 
Sylow $2$-subgroup of
$$H=\prod_{i=1}^{s}\ \GL_{m_i}(q^{2d_i})  \times  \prod_{j=1}^{t} 
\GU_{n_j}(q^{e_j})$$
for some non-negative integers $s$, $t$, and some positive integers 
$m_i$, $n_i$ $d_i$, $e_j $ 
such that all $e_j $ are odd and satisfy
$$n = \sum_{i=1}^s\ 2m_i d_i  + \sum_{j=1}^t\ n_j e_j.$$
We keep this notation for the remainder of this section.

\begin{Lemma}\label{2sgeqextra}   

\smallskip\noindent
(i) Suppose that 
$ \sum_{1\leq i \leq s} m_i + \sum_{1\leq j \leq t}  n_j \geq 3 $.  
If   for some $ i, j $,  either $\GL_1 (q^{2d_i}) $ or $\GU_1 (q^{e_j}) $   
contains an element of order $2^a $, then  so does $P$.

\smallskip\noindent
(ii)  Suppose that   $\sum_{1\leq i \leq s} m_i + \sum_{1\leq j \leq t}  
n_j \geq 4 $ and either  $ m_i \geq 2 $ for some  
$i$, $1\leq i \leq s$ or $ n_j \geq 2 $ 
for some $ j$, $ 1\leq j \leq t $.
Then $ P $ contains an element of order $8$.

\smallskip\noindent
(iii) Suppose that $n_j \geq 2 $ for some $ j$, $1 \leq j \leq t $ and 
$\sum_{1\leq i \leq s} m_i + \sum_{1\leq j \leq t}  
n_j \geq 3 $. Then, $ \tilde P$ contains an element of order $8$.
\end{Lemma}

\begin{proof}  (i) This follows easily 
from  Proposition \ref{2determinantsurjective} using arguments  
similar to Lemma \ref{orderpreserve}. 

(ii) Similar argument to (i) using the fact that $ \GL_2(q') $ and 
$\GU_2(q') $ contain elements of order $8$ for any odd prime power $q'$.

(iii ) Again, we use the fact that $\GU_2(q^{e_i}) $  has an element of order 
$8$.
\end{proof}

\begin{Lemma} \label{2sgeqthree}
Suppose that $P$ is elementary abelian and that $s +t\geq 3$.
Then the following hold.

\smallskip\noindent 
(i) $s=0 $.

\smallskip\noindent
(ii) $n_i =1 $  for all $i $, $1\leq i \leq t $.

\smallskip\noindent
(iii) $q \equiv 1\ (\mod\ 4)$.

\smallskip \noindent
(iv) If $t$ is even then $|P| = 2^{t-2}$, and if $t$ is odd then 
$|P| = 2^{t-1}$; in particular, the $2$-rank of $P$ is even.

\smallskip\noindent
(v) The block $d$ of $\GU_n(q)$ is nilpotent with an
elementary abelian defect group $T$ of order $2^s$.

\smallskip\noindent
(vi) The block $b$ of $kG$ satisfies Alperin's weight conjecture.

\end{Lemma}

\begin{proof}  
Since $ \GL_1(q^{2d_1}) $ contains an element of order $8$, 
conclusion (i) is immediate from
Lemma \ref {2sgeqextra}(i).
Assume from now on that $s=0 $. If   $n_1 \geq 2 $, then  the fact that   
$s+t \geq 3 $ implies that  
$\sum_{1\leq j \leq s}n_j  \geq 4 $. 
So, by Lemma \ref {2sgeqextra}(ii), $P$ has an element of order $8$, 
a contradiction. So, (ii) holds.
If $q \equiv 3 \ (\mod\ 4)$, then $\GU_{1}(q^{e_i}) $ contains 
an element of order $4$, and again by Lemma  \ref {2sgeqextra}(i), 
$P$ has an element of order $4$. This proves (iii).   
Hence $T$ is a Sylow $2$-subgroup of
$\prod_{j=1}^t\ \F_{q^{e_j}}^\times$. Since the $e_j$ are odd 
and $q\equiv\ 1\ (\mod\ 4)$ this implies that $T$ is elementary
abelian of rank $t$, and hence $\tilde{P}$ is elementary
abelian of rank $t-1$. If $n$ is odd then $P\cong$ $\tilde{P}$ and
since $n=\sum_{j=1}^t\ e_j$ and the $e_j$ are odd it follows that
$t$ is odd, hence $|P| = 2^{t-1}$. If $n$ is even then
$|P| = \frac{|\tilde{P}|}{2}$ and since $n=\sum_{j=1}^t\ e_j$ and 
the $e_j$ are odd it follows that $t$ is even and 
$|P|=2^{t-2}$, which proves (iv). Finally, the block $d$ of $\GU_n(q)$ is
nilpotent in that case because (ii) implies that the centraliser of 
the semi-simple element labelling $d$ is a torus, whence (v).
Statement (vi) follows from (v) and Proposition \ref{nilextremark}.
\end{proof}

\begin{Lemma}\label{2sleqtwo}
Suppose that $P$ is elementary abelian and that
$s +t \leq 2$ . Then $|P|\leq 4$.
\end{Lemma}

\begin{proof}
Suppose $s+t = 1$; that is, $T$ is a Sylow $2$-subgroup of
$\GL_m(q^{2d})$, where $m = m_1$ and $d=d_1$, and we have $n =2md$  or
$T$ is a Sylow $2$-subgroup of
$\GU_m(q^{e})$, where $m= n_1$ and $e=e_1$, and we have $n = me$.
In the former case, 
if $m = 1$ then $T$ is cyclic, hence $|P|\leq 2$, so we may
assume $m\geq 2$. If $ m\geq 3 $, then since $\GL_1(q^{2d})$ 
contains an element of order $8$, by Lemma \ref {2sgeqextra} (i), 
$P$ has an element of order  $8$, a contradiction.
This contradiction shows $m = 2$. Thus 
$n = 4d$, and so $T$ is
a Sylow $2$-subgroup of $\GL_2(q^{2d})$, and one easily checks
that any elementary abelian subquotient of $T$ has rank at most $2$.
In  the latter case, again if $m=1$, then $T$ is cyclic. If $m \geq 4 $, 
then by Lemma   \ref {2sgeqextra} (ii), $ P$ contains an element of order 
$8$, a contradiction. 
If $m =3 $, then   by Lemma \ref {2sgeqextra} (iii),  
$ \tilde P $ has an element of 
order $8$. But  $n=3e $ is odd,   and  hence   $ P \equiv \tilde P $. This 
contradiction shows that   $m = 2$, and  $T$ is
a Sylow $2$-subgroup of $\GU_2(q^{2})$. But any elementary abelian 
subquotient of $T$ has rank at most $2$.
Now suppose that $s+t = 2$ and $ \sum_i m_i + \sum_j n_j  \geq 3 $. 
By Lemma \ref {2sgeqextra} (i),  $s =0 $, whence $t=2 $.  
By Lemma \ref {2sgeqextra} (ii), at least one of $n_1 $ or 
$n_2 $ equals $1$, say $n_2=1 $. If $n_1  \geq  3 $, then by   
Lemma   \ref {2sgeqextra} (ii), $P$ contains an element of order $8$.  I
If $n_2=2$, then by Lemma \ref{2sgeqextra} (iii), 
$ \tilde P$ contains an element of order $8$. But in this case, 
$n= 2e_1 + e_2 $ is odd, which means that $ \tilde P \equiv P $, and hence 
$P$ has an element of order $8$.   If $ n_1 =1 $, then 
$ T$ is  abelian of rank  $2$, hence $P$ has order at most $4$.
Finally  suppose that $ s+t = 2$ and $ \sum_i m_i + \sum_j n_j  = 2$. 
If $ s=0, t=2 $,  then $ n_1 = n_2 =1 $; if $ s=1, t=1$, then $m_1=n_1=1 $ and 
if  $ s=2,  t=0 $,  then $ m_1 = m_2 =1 $. In all cases   
 $ T$ is  abelian of rank  $2$ whence $P$  has order at most $4$.
\end{proof}

\begin{proof}[Proof of Theorem \ref{type2A}]  This is immediate from the 
preceding lemmas.
\end{proof}

\section{Orthogonal and symplectic groups  in odd characteristic}

We show that blocks of  orthogonal and 
symplectic groups  with elementary abelian defect groups of order at 
least $8$ are all nilpotent.

\begin{Theorem}\label{classical}
Let $G$ be a quasi-simple finite group such that
$Z(G)$ has odd order and such that $G/Z(G) $ is isomorphic to 
one of $\PSp_{2n} (q) $, $n\geq 2$, $\POmega_{2n+1} (q) $ , $n\geq 3$ 
or $ \POmega^{\pm}_{2n} (q) $, $n\geq 4$   for  some odd prime power $q$.
Let  $b$ be a block of $kG$ 
with elementary abelian defect groups of order $2^r$ for some
integer $r\geq 3$.  Then $b$ is  nilpotent.
\end{Theorem}

{\bf Notation.} The group $G$ will denote one of the groups in the above
 theorem. Further, we define $L$, $L_0$  and $\tilde G$ as follows. 

\smallskip 
If $G/Z(G) =\PSp_{2n} (q) $, then  $ L=L_0 =\Sp_{2n} (q) $  and 
$\tilde G= \Sp_{2n}(q)$. 

\smallskip 
If  $G/Z(G) =\POmega_{2n+1} (q) $, then    
$L=\GO_{2n+1} (q)$,    $L_0= \SO_{2n+1} (q) $ and 
$\tilde G=\Omega_{2n+1} (q) $.

\smallskip 
If  $G/Z(G) =\POmega_{2n}^{+} (q) $,  then  $L =\GO^{+}_{2n} (q) $, 
$L_0= \SO^{+}_{2n} (q) $ and $ \tilde G= \Omega_{2n}^{+} (q) $.

\smallskip 
If  $G/Z(G) =\POmega_{2n}^{-} (q) $,  then  $L =\GO^{-}_{2n} (q) $, 
$L_0= \SO^{-}_{2n} (q) $ and $ \tilde G= \Omega_{2n}^{-} (q) $.

So,  $ \tilde G \lhd L_0 \lhd L$   with the indices of the inclusions 
 being $1$ or $2$ and
$ G = \tilde G/ Z $ where $ Z$ is a central subgroup of  $\tilde G$ 
of order $1$ or $2$.  
Let $b$ be a block of $kG$ and denote by $P$ a defect group
of $b$. Since $Z$ is a central $2$-subgroup, $b$ is the image
of a unique block $\tilde{b}$ of $k\tilde G$, and the inverse
image $\tilde{P}$ of $P$ in $k\tilde G$ is a defect group of
$\tilde{b}$. Let $d_0$ be a block of $kL_0$ covering $\tilde{b}$
with a defect group $T_0$ such that $T_0\cap \tilde G  = \tilde{P}$.
We note that the block $\tilde{b}$ is $L_0$-stable  
(see \cite[Corollary 6.4]{CEKL}) whence, 
$d_0= \tilde{b}$ and $T_0/\tilde{P}$ is cyclic of order $[L_0:\tilde G]$.
Let $d$ be a block of $kL$ covering $d_0 =$ $\tilde b$ and $T$ a defect 
group of $b$ containing $T_0$.   
As a byproduct of the Jordan decomposition of blocks 
(see Theorem \ref{jordandecom0}), there exists a  
semi-simple element 
$s \in L_0$ of odd order  such that $T_0$ is   a 
Sylow $2$-subgroup of $ C_{L_0}(s)$ and $T$ is a Sylow $2$-subgroup of 
$ C_{L}(s)$  (see \cite [(5A)]{Anclassodd}).

\begin{Lemma} \label{jordandecom1} If the defect groups   of  
$kL_0d_0 $ are abelian, then $ kL_0d_0 $, $kLd$, $k\tilde G \tilde b $ and 
$k G b $  are all  nilpotent blocks.
\end{Lemma}

\begin{proof} 
$L_0 =\Lb_0 ^F$, where   $\Lb_0$ is   a simple algebraic  
group of type $B$, $C$ or $D$ over $\bar {\mathbb F}_q$   and 
$F:\Lb_0 \to$ $\Lb_0$ is a Frobenius  endomorphism with respect to  an 
$\F_q$-structure on $\Lb_0 $. Thus, by Corollary \ref{jordandecomabel}, 
$kL_od_0 $ is  nilpotent.  Since $\tilde G$ is a normal subgroup of  
$L_0$ of index a power of $2$ and $d_0$ covers $\tilde b$, 
$k\tilde G \tilde b$ is  nilpotent.  Since $G$ is a  quotient of 
$\tilde G$ by a central $2$-subgroup, and $\tilde b$ lifts $b$,
$k G  b$ is  nilpotent.  Finally $L_0$ is of index $1$ or $2$ in 
$L$, and $d$ covers $d_0$, hence  $kLd$ is nilpotent.
\end{proof}

For the rest of this section  $s$ will  denote
a  semi-simple element of odd order in  $L_0 $ such that $T_0$ 
is a Sylow $2$-subgroup of $C_{L_0}(s)$ and $T$ is  a Sylow 
$2$-subgroup of $C_L(s)$.

\begin{Lemma}\label{sympexp1}  
Suppose that $ G = \PSp_{2n}(q) $, $ n\geq 2 $. 
If  $P$ is  abelian, then  $b$ is nilpotent.
\end{Lemma}

\begin{proof}  
Note that $\tilde G = [L, L] = L_0= L $,
$G =L/Z$, $|Z| =2$, and  $ d= d_0 =\tilde b $.  Further,  
by \cite[\S 1]{FoSriclass}, and noting that $s$ has odd order,
the group $H=$ $C_L(s)$ is a direct product  of  groups $H_i $  as in  
Lemma  \ref{classicalcent2sylow} with $H_0 \cong \Sp_{2m_0}(q)$ and 
$$n = m_0 +  \sum_{1\leq i \leq t} m_id_i. $$
The above decomposition of $C_L(s)$ corresponds 
to the  orthogonal decomposition of the underlying symplectic  space as a 
direct sum  of isotypic $\F_q[s]$-modules (for instance, 
$H_0$ corresponds to the $1$-eigen space  of $s$). In particular, if 
the decomposition has more than one non-trivial factor, then 
$ Z \cap H_i=1 $ for all  $i$, $0\leq i \leq t $.
We have  $ T=$ $T_0= $ $\tilde P$ and $P = $ $T\cap [L,L]/Z=$ $T/Z $.
Suppose that  $P$ is abelian.
If $m_0 \ne$ $0 $, then by Lemma \ref{classicalcent2sylow}(ii) and (iv), 
$m_0=1$, $H=$ $H_0=$ $\Sp_{2}(q) $. In particular, $n=$ $m_0=$ $1$,
a contradiction. Thus $m_0= 0$. Suppose that $m_i =2 $ for some 
$i \geq 1$. Then by Lemma \ref{classicalcent2sylow}, $t=1$ and 
$H=$ $\GL_{2}(q^{d_1})$ or $ H_1=$ $\GU_{2}(q^{d_1})$.
But as observed above, $P$  is the quotient of  a Sylow $2$-subgroup
of $H$ by  $Z(H)$, and  $\PGL_2(q')$ and  $\PGU_2(q')$  
have  non-abelian Sylow $2$-subgroups for any odd prime power $q'$, 
a contradiction.  Thus $m_0=$ $0 $ and $m_i\leq$ $1$ for all $ i$, 
$1\leq i\leq t $. In particular, $H$ 
and therefore $T_0$ is abelian. The result follows by  
Lemma \ref{jordandecom1}.
\end{proof}

Before  we proceed, we recall the structure of $C_L(s)$    
when $L$ is an orthogonal group as described in 
\cite[\S 1]{FoSriclass}.
Let $V$ be  an underlying $\F_{q}$-vector space for $L$,
and let $\tau: V \to$ $\F_q $  be a  
non-degenerate  quadratic form underlying $L$. So $L=I(V)$, the  subgroup of 
$\GL(V)$    consisting of isometries  with respect to  $\tau$,
$L_0=I_0(V)$, the subgroup of $I(V)$ consistsing of matrices of determinant $1$
and $ \tilde G  =$ $\Omega (V) =$ $[I(V), I(V)]$.
Let $V_0$ denote the $1$-eigen space of $s$.
The space $V$ decomposes as an orthogonal direct sum
$$ V = V_0 \oplus ( \oplus_{1\leq  i \leq t}   V_i) $$
where   for $i \geq 1 $,  $V_i$, is an isotypic 
$\F_q[s] $-module,  such that $V_i$ and $V_j$ have no common   
irreducible $\F_q[s] $-summands  for $ 0\leq i \ne j \leq  t $  and  such that
$$C_L(s)   =   \prod_{i} C_L(s)  \cap \GL(V_i ). $$
Here  $\prod_{i} \GL(V_i)$   is considered as a subgroup  of   $\GL(V)$   
in  the  standard way.
Further, setting $H_i =  C_L(s)  \cap \GL(V_i )$, we have that   
\[ H_i= \left\{ \begin{array}{ccc}    L \cap  \GL(V_i)  & {\rm if}  & i=0 \\
\GL_{m_i} (\epsilon_i q^{d_i} ) &  {\rm if}   & i  \geq 1.  
\end{array} \right.
\]
Here for each $i \geq 1 $,  
$2d_im_i $ is the  $\F_q$-dimension of $V_i $ and $H_0 =I(V_0) $.
Also,  $ \epsilon_i \in \{\pm 1\} $ and $\GL_{m_i} (\epsilon_i q^{d_i} )  $ 
denotes $ \GU_{n_i}(q^{d_i}) $ if $\epsilon_i =-1$.
If $i \geq 1$, then $H_i \leq$ $L_0$. Thus,
$$C_{L_0}(s) =    (H_0 \cap L_0) \times \prod_{1 \leq i \leq s} H_i. $$
For each $i$,  $ 0\leq i \leq  t $, the restriction 
of $\tau$ to $V_i$ is non-degenerate. 
This form has maximal Witt index if $\epsilon_i^{m_i} =1$ and 
is of non-maximal Witt index if $\epsilon_i^{m_i} = -1$. 
For $1\leq i\leq t$, let $t_i$ be the unique involution in the 
centre of $H_i$. As an element of $\GL(V)$, 
$t_i$ acts as $-1 $ on  $V_i$ and as $1$ on all $V_j$, $j \ne i$.
Now $ \tilde G= [I(V), I(V)] $ is  the kernel of the   
spinorial  norm  from $I_0(V)$  to $\F_q^{\times}/ \F_q^{\times 2} $
(see  \cite[\S 2.7]{GorLyoSol3}).  From this it follows that 
$t_i \in \Omega (V) $ if and only if  
$ q^{d_im_i} \equiv \epsilon_i^{m_i}  \ (\mod\ 4)$.   
We will use this fact in the sequel.

\begin{Lemma} \label{oddclassicalexp} Suppose that $ G = \POmega_{2n+1} (q) $, 
$ n\geq$ $3$ and that $P$ is elementary abelian  with $|P| \geq$ $8$.  
Then   $b$ is nilpotent.
\end{Lemma}
 
\begin{proof}   
Note that  since the dimension of the underlying vector space is 
odd, $G =$ $\tilde G =$ $[L,L]$, $\tilde G$ is of index $2$ in   
$L_0$ and $L_0$ is of index $2$ in $L$. In particular, $\tilde P=$ $P$.
By  Lemma \ref{classicalcent2sylow}, either $m_1=2 $ and  
$m_i =0$ for all $i$ different from $1$,  or all $m_i \leq  1$. 
In the former case, again by Lemma \ref{classicalcent2sylow},  
$T$ is isomorphic to a Sylow $2$-subgroup  of a $2$-dimensional 
general linear  group.  In particular,  $T$ has an element of 
order $8$. Since   $T\leq L$ in this case, $P$ is of index $2$ in  of $T$, 
hence $P$ has an element of order $4$, 
a contradiction.   We assume from now on that all $m_i \leq  1 $.
If $m_0=0 $, then $T$ and  therefore $T_0$ is abelian and  we are done by
Corollary  \ref{jordandecom1}.
So,  $m_0 =1 $, i.e.  $H_0=\GO_3(q)$.  
Since $n \geq 2 $, $ i\geq 1 $, i.e., $m_1 \ne 0 $. 
Let $T^i$ be the $i$-th component  
of $T$. We claim that $ T^i \not\leq  P $ 
for any $i\geq 1$. Indeed,  suppose the contrary. Then 
$T^i = \langle t_i \rangle $  has order $2$.  In particular, this means that 
$q^{d_i} \not \equiv\epsilon_i \ (\mod\ 4)$.  But since $m_i =1$, this 
 means that   $t_i \notin  \Omega(V) $ and hence  $t_i \notin P$.  
This proves the claim. 
By   Lemma \ref{classicalcent2sylow},  $T^0 \cap  H_0 $  is a dihedral 
group of order at least eight. Also,  clearly $T^0\cap  H_0 \leq T_0 $. 
Since $P$ is of index $2$ in $T_0$,  and since as just shown
$T^1 \not\leq P$, it follows that $P$ contains a 
subgroup isomorphic to $T^0\cap H_0$, an impossibility as  
$P$ is abelian and $T^0\cap\SO_3(q)$ is not.
\end{proof}

\begin{Lemma}  
Suppose that $G = \POmega^{+}_{2n} (q)$, or   
$\POmega^{-}_{2n} (q)$,  $n \geq 4$ and that $P$ is elementary 
abelian  with $|P| \geq 8$. Then  $b$ is nilpotent.
\end{Lemma}
 
\begin{proof}     
We first consider $ G = \POmega^{+}_{2n} (q)$.   If $m_1=2 $,  
then by Lemma \ref{classicalcent2sylow}(i) and (iii),  $m_j =0 $  
for $ j \ne i $, $C_L(s)= H_1 \cong \GL(\epsilon_1q^{d_1}) $,  
where $ q\equiv  \pm 3  \  (\mod  \  8 )$ and $d_1$ is odd.           
But $q^{d_1m_1} \equiv$ $1 \equiv$ $\epsilon_1 ^{m_1}  \  ( \mod  \   4 )$, 
hence  the central involution $t_1$ of $H_1$ is  in $[L, L] $ 
and $P$ is a subgroup of $ T/\langle  t_i \rangle$. Since   
$T/\langle t_i\rangle$ is a dihedral group (see \cite{CarterFong}), 
$P$ has rank at most $2$, a contradiction.
Now suppose $m_0 =2$. Then,   $m_i =0 $  for all $i \geq 1 $,   
and $n=2$, a contradiction, as $n$ is assumed to be at least $3$. 
Thus, $m_i \leq  1$ for all $ i $.   If $m_0 =0 $, then  $T_0$ is abelian and 
we are done by  Theorem  \ref{jordandecom1}. So, suppose that  $m_1=1$.
Then $H_0 =\GO_2^{\pm} (q)$ 
is a dihedral group  and  $H_0 \cap L_0 = \SO_2^{\pm} (q)$ is a cyclic group
(see \cite[\S 2.7]{GorLyoSol3}).   Since  
 $H_i $ is  also cyclic  for all  $i \geq 1$,
$T_0 $  is abelian and we are done by  Theorem  \ref{jordandecom1}.
The  proof for  $G=\POmega^{-}_{2n} (q)$ is similar.
\end{proof}

\begin{proof}[Proof  of Theorem \ref{classical}]  This is immediate from  the 
preceding  lemmas.
\end{proof}

\section{Type $G_2$, $\,^2G_2$ and  ${^3D_4}$ in odd characteristic}

Let $q$ an odd prime power. The $2$-rank of $G_2(q)$, 
${^2G_2(q)}$ and ${^3D_4(q)}$ is $3$; see e.g. 
\cite[\S 1]{GorHar}, \cite[Theorem 4.10.5]{GorLyoSol3}.

\begin{Proposition} \label{typeG2}
Let $G$ be a quasi-simple finite group such that $Z(G)$ 
has odd order and such that $G/Z(G)$ is simple of type $G_2(q)$.
Then $kG$ has no block with an elementary abelian defect
group of order $8$.
\end{Proposition}

\begin{proof}
This follows from \cite{HiSh}, where the $2$-blocks of
$G$ are determined. Alternatively, one can use the arguments
in \cite[12.2]{CEKL}: there is a unique conjugacy class
of involutions $u$ in $G$ and by \cite[4.5.1]{GorLyoSol3},
$C_G(u)\cong$ $Z(G)\times$
$(2.(\PSL_2(q)\times\PSL_2(q)).2)$. The acting $2$-automorphism
in the second factor is inner-diagonal, hence stabilises
any block of $2.(\PSL_2(q)\times\PSL_2(q))$. 
Thus a block of $C_G(u)$ with elementary abelian defect
group of order $8$ would cover a block of
$2.(\PSL_2(q)\times\PSL_2(q))$ with a Klein four defect group,
whose image modulo the central involution would yield a
block of $\PSL_2(q)\times\PSL_2(q)$ with a defect group of
order $2$. Any block of this direct product is of the
form $c_0\ten c_1$, where $c_0$, $c_1$ are blocks of
$\PSL_2(q)$, so exactly one of $c_0$, $c_1$ would have
defect zero and the other defect one. But since any inverse
image of an involution in $\PSL_2(q)$ in $\SL_2(q)$ has
order $4$ this is impossible.
\end{proof}

\begin{Proposition} \label{type2G2}
Let $G$ be a quasi-simple finite group such that $Z(G)$ has 
odd order and such that $G/Z(G)$ is simple of type ${^2G_2(q)}$. 
The principal block $b_0$ is the unique block of $kG$ having an 
elementary abelian defect $P$ group of order $8$, and we have
$|\Irr_K(G,b_0)|=8$.
\end{Proposition}

\begin{proof}
The Sylow $2$-subgroups of ${^2{G_2(q)}}$ are elementary
abelian of order $8$.
The simple group of type ${^2G_2(q)}$ has trivial Schur multiplier, 
hence $Z(G)=1$. In that case we have $N_G(P)\cong$
$P\rtimes E$, with $E$ a Frobenius group of order $21$ acting
faithfully on $P$, and hence, by Brauer's First Main Theorem,
the principal block $b_0$ of $kG$ is the unique block having $P$
as defect group. By Ward's explicit calculations in \cite{Ward}
or Landrock's general results in \cite[\S 3]{Landrock81}
we have $|\Irr_K(G,b_0)|=8$. 
\end{proof}

Finally for the triality $D_4$-groups we have 
the following proposition due to 
Deriziotis and Michler \cite[Proposition 5.3]{Deri-Michler}

\begin{Proposition} \label{type3D4}
Let $G$ be a quasi-simple finite group such that $Z(G)$ has 
odd order and such that $G/Z(G)$ is simple of type $\,^3D_4(q)$. Let $P$  
be a defect group of some block of $kG$. Then either $P$ is 
non-abelian or $P$ has rank at most $2$. In particular, no block of $kG$ 
has  defect groups  which are elementary abelian of order $8$.
\end{Proposition}

\section{Unipotent characters with small $2$-defects} 
\label{smalldefectsection}

Recall that the
$2$-defect of an irreducible character $\chi$ of a finite group 
$G$ is the largest integer $d(\chi)$ such that $2^{d(\chi)}$
divides the rational integer $\frac{|G|}{\chi(1)}$. 
The notation of the  finite groups of Lie type and 
the labelling of their unipotent characters due to Lusztig in the following 
two propositions are from the tables page $75/76$ and  pages 
$465$--$488$ in Carter's book \cite{Carter93}.
In particular, if $X$ is some simple type we denote by $X(q)$ the finite
group of Lie type (that is, the group of fixed points in the algebraic
group under some Frobenius endomorphism). We will deviate from the notation 
of \cite{Carter93}  in one respect, i.e., we will denote by 
$\,^2A_l(q)$, $\,^2D_l (q)$, and $\,^3D_4(q)$ the  twisted groups denoted by 
$\,^2A_l(q^2)$, $\,^2D_l(q^2)$, and $\,^3D_4(q^3)$ respectively in 
\cite{Carter93}. For classical groups, we will  also draw upon 
Olsson's treatment  of the combinatorics of symbols \cite{Olsson86}.
Given a positive integer $n$, let  $n_+$ denote as before the 
$2$-part of $n$ and $n_{-} $ the $2'$-part of $n$. So, $n=n_+n_{-}$,  
no odd prime divides $n_+$ and $2$ does not divide  $n_{-}$. 

\begin{Proposition} \label{2defectsclassical} 
Let $q$ be a power of an odd prime $r$. Let $(q-1)_+= 2^d$ and $(q+1)_+= 2^e $. 

\smallskip\noindent (i)  Every  unipotent character of  $A_l(q) $ has $2$-defect 
greater than or equal to $dl $.  If $l+1 $ is not a 
triangular number, then the $2$-defect of  any  unipotent character of 
$ A_l(q) $  is at least $dl +  e \geq 2 + l$. 
If $ l \geq 3 $, then  all unipotent charecters have  
$2$-defect greater than or equal to $5$.
$A_2(q)$ has one  unipotent character, $ \chi^{(1,2)} $ of   $2$-defect  $2d$, 
all other unipotent characters have $2$-defect $2d+e \geq 4 $.   
The $2$-defect of any unipotent character of $A_1(q) $ is $d+e \geq 3 $.

\smallskip\noindent (ii)  
Every  unipotent character of  $\, ^2A_l(q) $ has $2$-defect   
greater than or equal to $el $. If $l+1 $ is not a 
triangular number, then the $2$-defect of any  unipotent character of 
$ \,^2A_l(q) $  is at least $ e l +  d \geq 2 + l$. 
If $ l \geq 3 $, then  all unipotent characters have  
$2$-defect greater than or equal to $6$.
$\,^2A_2(q)$ has one  unipotent character, $ \chi^{(1,2)} $ of   $2$-defect  
$2e$, all other uniptent characters have $2$-defect  $2e+d \geq 4 $.   
The $2$-defect of any  unipotent character of $\,^2A_1(q) $ is $d+e \geq 3 $.
\smallskip\noindent (iii) 
Every unipotent character of $B_l(q)$ or 
$C_l(q)$, $l\geq 2 $ has $2$-defect at least $2l \geq  4$. 

\smallskip\noindent(iv) 
Every unipotent character of  $D_l(q) $ or  $\,^2D_l(q)$, $ l\geq 4 $ 
has $2$-defect at least $2l-1 \geq 7$.
\end{Proposition}

\begin{proof}  We note that $d+e \geq 3$.

(i) The unipotent characters of $A_l(q)$ are parametrized by 
the partitions of $ l+1 $. If   $\alpha $  is a partition of $l$ and  
$\chi^{\alpha }$ is the corresponding  unipotent character, then 
$$|\chi^{\alpha}(1)|_{r'}=\frac { (q-1) |A_l(q)|_{r'}}{ \prod_h (q^h-1)}$$
where $h$ runs over the set of hook lengths of $\alpha$ (see for instance 
\cite[pp.152-153]{Macdonald}).  
Thus the $2$-defect of  $\chi^{\alpha} $ is $f$ where 
$ 2^f=$ $\frac{\prod_h (q^h-1)_+ } {(q-1)_+}$. Since $\alpha$ has 
$l+1 $ hooks, the first assertion follows. If $l+1$ is not a 
triangular number, then at least one hook of $\alpha$ is of even 
length. So, $(q^2-1)$ divides $q^h-1$ for at least one hook length  
$h$ of $\alpha$, whence $f \geq  (d+e) +d(l-1)$. Since $5$ is not a 
triangular number, it is immediate from the first two assertions 
that any unipotent character of $A_l(q)$, $l\geq 4$ has 
$2$-defect at least $6$. Any partition of $4$ has two hooks of 
even length, hence $f \geq 2(d+e) \geq 6$. If $\alpha=(1,1,1)$ or 
$\alpha= (3)$, then the hook lengths of $\alpha $ are $1$, $2$ and 
$3$ and it follows that $f = 2d+e \geq 4$. If $\alpha =(2, 1)$, 
then the hook lengths of $\alpha $  are $1, 1$ and $3$ and the 
$2$-defect is $2d$. Finally, if $\alpha =(1,1)$ or $(2)$, then the  
hook lengths are $1$ and $2$ and the $2$-defect is  $d+e$.

(ii) The  unipotent characters of $\,^2A_l(q)$ are also parametrised  
by  partitions of $l+1 $. If $\chi^{\alpha} $ is the character 
labelled by  $\alpha $, then 
$$|\chi^{\alpha}(1)|_{r'}
=  \frac { (q+1) |\,^2A_l(q)|_{r'}}{\prod_h ((-q)^h-1)}, $$
where $h$ runs over the set of hook lengths of  $\alpha $ 
(see \cite{LuszSri}). The assertions of (ii) follow  as in (i).

(iii) and (iv) Let $G$ be  one of the groups $B_l(q)$, 
$C_l(q)$, $l\geq 2 $,  $D_l(q) $, $\,^2D_l(q)$, $ l\geq 4 $.
In Lusztig's description, the unipotent characters are 
labelled using  symbols. A symbol is an equivalence class   
$[X, Y] $ of unordered pairs $[X, Y] $ with  
$X$, $Y$ (possibly empty) subsets of the non-negative 
integers. The pair   $[X, Y]$ is equivalent to $[X', Y']$  if and 
only if there exists an  integer $t$ such that either $X = X'^{+t}$ 
and $Y=$ $Y'^{+t}$ or $X =$ $Y'^{+t}$ and $Y = X'^{+t}$. Let 
$[X, Y]$ be a symbol, $X =$ $\{ a_1, a_2, \cdots, a_k\}$, 
$Y =\{ b_1, b_2,  \cdots, b_r \}$, with $a_1>a_2\cdots>a_k \geq 0$ 
and $b_1> b_2> \cdots b_r \geq 0$. The rank of $[X, Y]$ is defined by
$${\mathrm rk}[X, Y] =  \sum_i a_i + \sum_j b_j  
- \lfloor \left( \frac{k+r-1}{2}\right) ^2  \rfloor$$ 
where $\lfloor x\rfloor $ denotes the greatest 
integer less than or equal to $x$. Let $c(X,Y)$ be defined by 
\[ c(X,Y) = \left\{ \begin{array} 
{r@{\quad:\quad}l}
 \lfloor  \frac{k+r-1}{2}  \rfloor - |X \cap Y| &  X \ne Y \\
0 & X=Y.
\end{array} \right. \]
If $\chi_{[X, Y]}$ is the unipotent character of $G$ indexed by the 
symbol  $[X,Y] $, then ${\mathrm rk}[X, Y] =$ $l$ and  
$$|\chi_{[X, Y]} (1)|_{r'}=    \frac{|G|_{r'}} {2^{c(X,Y)} \prod_h(q^h-1)
\prod_{h'}(q^{h'} +1)}$$ 
where $h$ runs over the set of hooks of $[X,Y] $ and $h'$ runs over 
the  set of cohooks of $[X,Y]$ (see \cite[Proposition 5]{Olsson86}).  
Let $h^{+}[X, Y]$ be the number of hooks of $[X, Y]$ and let 
$h^{-}[X, Y]$ denote the number of cohooks of $[X,Y]$.
Then by the character formula  above it  suffices to prove that 
$c(X,Y) + h^{+}[X, Y] + h^{-}[X, Y] = $ $2 (\mathrm{rk} [X,Y]) $ 
in case $G=B_l(q)$ or $G=C_l(q)$ and that 
$c(X,Y) + h^{+}[X, Y] + h^{-}[X, Y]  =$ $2 (\mathrm{rk} [X,Y] ) - \delta $,
where $\delta \in \{0, 1\}$ in case $G=D_l(q)$ or $G={^2D_l(q)}$. 
By \cite[Equations 15, 16]{Olsson86},
\begin{equation}\label{numberhooks}h^{+}[X, Y] = 
\sum_{i} a_i + \sum_j b_j  -\binom{k}{2}  -\binom{r}{2}  \end{equation}
and 
\begin{equation}\label{numbercohooks}h^{-}[X, Y] = 
\sum_{i} a_i + \sum_j b_j  - kr + |X\cap Y|.  \end{equation}
From these it is straighforward to check that 
$$ c(X,Y) + h^{+}[X,Y] + h^{-}[X,Y] -2 (\mathrm{rk}[X,Y]) = 0 $$
if  either $ k-r $ is odd  or $X=Y $ and 
$$ c(X,Y) + h^{+}[X,Y] + h^{-}[X,Y] -2 (\mathrm{rk}[X,Y]) = -1 $$
if $ k-r $ is even and $X \ne Y $.
The result follows  since  if   $G=B_l(q)$ or $G=C_l(q)$  then $k-r $ is odd 
and if $ G =D_l(q)$ or $ G=\,{^2D_l(q)}$  then $k-r $ is even 
(see \cite[Section 13.8]{Carter93}).
\end{proof}

\begin{Proposition} \label{2defectsexceptional}
Let $q$ be an odd prime power.

\smallskip\noindent (i)
$G_2(q)$ has two unipotent characters of $2$-defect $0$, (labelled
$G_2[\theta]$, $G_2[\theta^2]$); if $q\equiv\ 1\ (\mod \ 4)$ then
$G_2(q)$ has two unipotent characters of $2$-defect $3$ (labelled
$G_2[1]$, $G_2[-1]$), and if $q\equiv\ 3\ (\mod \ 4)$ then $G_2(q)$ has
one unipotent character of $2$-defect $3$ (labelled $\Phi_{2,2}$).
In both cases, all remaining unipotent characters of
$G_2(q)$ have $2$-defect at least $4$.

\smallskip\noindent (ii)
${^2{G_2(q)}}$ has two cuspidal unipotent characters of
$2$-defect $0$, four cuspidal unipotent characters of
$2$-defect $2$, and the remaining two unipotent characters
have $2$-defect at least $4$. 

\smallskip\noindent (iii)
$F_4(q)$ has two unipotent characters of $2$-defect $0$ (labelled
$F_4[\theta]$, $F_4[\theta^2]$), and
two unipotent characters of $2$-defect $5$ (labelled
$F_4[i]$, $F_4[-i]$).
All other unipotent characters of $F_4(q)$ have $2$-defect at least $7$.

\smallskip\noindent (iv)
$E_6(q)$ and ${^2{E_6(q)}}$ each have two unipotent characters
of $2$-defect $0$ (labelled $E_6[\theta]$, $E_6[\theta^2]$ in
the case of $E_6(q)$), and all other unipotent characters have 
$2$-defect at least $8$. The  unipotent characters of defect $0$ of 
$ E_6(q)$  each  have degree 
$\frac{1}{3} q^7\Phi_1^6(q)\Phi_2^4(q)\Phi_4^2(q)\Phi_5(q)\Phi_8(q)$  and the 
 unipotent characters of defect $0$ of 
$ ^2E_6(q)$   each have degree 
$\frac{1}{3} q^7\Phi_1^4(q)\Phi_2^6(q) \Phi_4^2(q) \Phi_8(q)\Phi_{10}(q)$. 

\smallskip\noindent (v)
$E_7(q)$ has four unipotent characters of $2$-defect $d\geq 3$,
labelled $(E_6[\theta],1)$, $(E_6[\theta^2],1)$,
$(E_6[\theta],\epsilon)$, $(E_6[\theta^2],\epsilon)$, where
$2^d$ is the largest power of $2$ dividing $q^2-1$. These  
characters each have degree $\frac{1}{3} q^7\Phi_1^6(q)\Phi_2^6(q) \Phi_4^2(q) 
\Phi_5(q) \Phi_7(q)\Phi_8(q)\Phi_{10}(q)\Phi_{14}(q)$.     Furthermore,
$E_7(q)$ has exactly two unipotent characters of $2$-defect $8$
(labelled $E_7[\xi]$, $E_7[-\xi]$ if $q\equiv\ 1\ \mod\ 4$ and
$\Phi_{512,11}$, $\Phi_{512,12}$ if $q\equiv\ 3\ \mod\ 4$); all
other unipotent characters have $2$-defect at least $14$.

\smallskip\noindent (vi)
$E_8(q)$ has four unipotent characters of $2$-defect $0$ (labelled
$E_8[\zeta]$, $E_8[\zeta^2]$, $E_8[\zeta^3]$, $E_8[\zeta^4]$),
four unipotent characters of $2$-defect $3$
(labelled 
$E_8[\theta]$, $E_8[\theta^2]$, $E_8[-\theta]$, $E_8[-\theta^2]$
if $q\equiv\ 1\ \mod\ 4$, and
$(E_8[\theta],\Phi_{2,1})$, $(E_8[\theta^2],\Phi_{2,1})$, 
$(E_8[\theta],\Phi'_{2,2})$, $(E_8[\theta^2],\Phi'_{2,2})$
if $q\equiv\ 3\ \mod\ 4$);  
all other unipotent characters have $2$-defect at least $5$.

\smallskip\noindent (vii)
${^3{D_4(q)}}$ has two unipotent characters with $2$-defect $3$,
two unipotent characters with $2$-defect at least $5$, and four
unipotent characters with $2$-defect at least $6$.
\end{Proposition}

\begin{proof}
One proves this by first expressing the group orders given in
\cite[pp. 75-76]{Carter93} as products of cyclotomic polynomials 
evaluated at $q$. Then, for any of the above groups $G$ and any
unipotent character $\Phi$ of $G$ one calculates the $2$-part
of $\frac{|G|}{\Phi(1)}$ by running through Lusztig's lists of 
character degrees of unipotent characters in 
\cite[pp. 477--488]{Carter93}, observing that the only
cyclotomic polynomials $\Phi_d$ which occur in these character
degrees for which $\Phi_d(q)$ is even, are those for $d\in$
$\{1,2,4,8\}$. While the $2$-part of $\Phi_4(q)$ and $\Phi_8(q)$
is exactly $2$, the $2$-part of $\Phi_1(q)\Phi_2(q) =$
$q^2-1$ is at least $8$, possibly bigger, which accounts for
the inequalities in the above statements.
\end{proof}

\section{Type  $F_4$}

\begin{Proposition} \label{typeFpunch}
Let $G$ be a quasi-simple finite group such that $Z(G)$ has 
odd order and such that $G/Z(G)$ is a simple group of type 
$F_4(q)$, where $q$ is an odd prime power. 
If $kG$ has a block $b$ with an elementary abelian defect group 
$P$ of order $8$, then  either  $|\Irr_K(G, b)|=8 $ or  $ \CO  G b $ 
is Morita equivalent to a  block  of $\CO L $ of a finite group $L$ such that
$|L/Z(L)|< |G/Z(G)| $. 
\end{Proposition}

\begin{proof}  Here $Z(G)=1 $ and $ G = \Gb^F $, where  
$\Gb$ is a  simple  algebraic group of type $F_4$ over  $\bar {\mathbb F}_q$
and  $F: \Gb \to \Gb  $ is a Frobenius morphism with respect to an 
${\mathbb F}_q $-structure on $\Gb $.   
Let $ [t]$ be the semi-simple label of $b$ (see \S\ref{redback}) 
and let $\Cb(t):=  C_{\Gb^*}(t)$.   Suppose first that   
$  \Cb(t) $ is   contained in a  proper Levi subgroup  of 
$\Gb ^* $.  Then by the 
Jordan decomposition of blocks  (see  \S\ref{redback}), there exists a 
proper  $F$-stable Levi subgroup, $\Lb$
such that the block $\CO G b $  is Morita equivalent to  some block of  
$\CO \Lb^F$ . Since $\Lb$ is a proper Levi subgroup of $\Gb$,  
$|\Lb(t)^F/Z(\Lb(t)^F)|  $ is strictly smaller than $ |G/Z(G)| $.  
Hence,  we may assume that $t$ is a quasi-isolated  element of  
$\Gb^* $, i.e.  that 
$\Cb(t)$  is  not  contained in any proper Levi subgroup   
of $ \Gb^* $. Since   $Z(\Gb)$ is  trivial, 
$\Cb(t)=\Cb^{\circ}(t)$, hence   $t$ is  even isolated in $\Gb^*$, i.e 
$\Cb^{\circ}(t)$ is not contained in any proper Levi subgroup of $\Gb^*$.
From the   tables describing 
centralisers of semi-simple elements  in  groups of type $F_4$ in
 \cite{Shoji74}, (see also  the tables in \cite{Deri84})
one sees  that $\bar {\Cb}^{\circ} (t)^{F^*}  $ 
is isomorphic to  one of $F_4 (q) $,   $ B_4 (q)$,  
$C_3 (q) \times A_1 (q) $,    $ A_3 (q)\times A_1 (q) $, 
$ \,^2A_3 (q) \times  A_1 (q)$,  $A_2 (q) \times  A_2 (q )$  or  
$\,^2A_2 (q) \times  \,^2A_2 (q )$ (and  $Z ^{\circ}(\Cb ^{\circ} (t))=1 $).
By (\ref{matchchardeg}), (\ref{matchtwodefect}) and 
Propositions \ref{2defectsclassical} and \ref{2defectsexceptional}, 
the $2$-defect of  any element of $\CE(\Gb^F, [t])=\CE(\Gb^F, (t)) $ 
is at least $4 $, whence $\CE(\Gb^F, (t)) \cap \Irr_K(G, b)=\emptyset $, 
a contradiction. 
\end{proof}

\section{On characters of small defect in type $E$ groups}

We use the notation  of \S\ref{redback}.

\begin{Proposition}\label{small2defectE_6}   
Suppose that $ \Gb $  is simple, 
simply connected of  type $E_6$, and  $\Gb^F$ is     a group of type  
$E_6(q)$.  Let $ s $ 
be a semi-simple element of  $\Gb^{*F^*}$   and let $ \chi \in \Irr_K(G ) \cap 
\CE(\Gb^F, [s]) $. Let  $ \lambda $, $\lambda' $ and $\bar \lambda $
correspond to $\chi$ as in \S \ref{redback} (\ref{matchchardegfinal}).
Suppose that the $2$-defect $d(\chi)$ of $\chi$ is at most 
$3$. Then $a_{\chi} $ is $1$ or $3$, $\alpha_{\chi}=0 $
(see  notation  after (\ref{matchtwodefect}) in \S \ref{redback}) 
and $s$, $\chi $ and $\bar \lambda$ satisfy  one of the rows in the 
following table.

{\rm
\small
\begin{center}
\begin{tabular}{l|c|c|l|c|l|c|c}
\hline
& & & & & & & \\
 &  $\Delta_s$ &  $\Cb^{\circ}(s)'^{F^*}  $ &  $z(s)$ & 
$\bar \lambda $ & $ a_{\chi}\chi(1)_{r'}$ 
& $d(\chi)$  &     cond. on $q$ \\
& & & & & && \\
\hline
(i) &   -& -  & $\Phi_1^2\Phi_3 ^2$ &  $1$& 
$\Phi_1^4\Phi_2 ^4\Phi_3\Phi_4^2\Phi_5\Phi_6^2\Phi_8 \Phi_9\Phi_{12}$ & 
$2$  & $q\equiv 3 $ mod  $4 $\\

(ii) &-& - & $\Phi_1^2\Phi_5$ & $1$ &  
$\Phi_1^4\Phi_2 ^4\Phi_3^3\Phi_4^2\Phi_6^2\Phi_8 \Phi_9\Phi_{12} $& 
$2$  & $q\equiv 3 $ mod  $4 $\\

(iii)&   -& - & $\Phi_1\Phi_2\Phi_3^2$ & $1$ &   
$\Phi_1^5\Phi_2 ^3\Phi_3\Phi_4^2\Phi_5\Phi_6^2\Phi_8 \Phi_9\Phi_{12} $& 
$3$  &   $(q^2-1)_+ =  8 $ \\

(iv)& -& - & $\Phi_1\Phi_2\Phi_5$ & $1$ &  
$\Phi_1^5\Phi_2 ^3\Phi_3^3\Phi_4^2\Phi_6^2\Phi_8 \Phi_9\Phi_{12} $& 
$3$  &  $(q^2-1)_+ =  8 $    \\

(v)& -& - & $\Phi_1\Phi_2\Phi_3\Phi_6$ & $1$ &   
$\Phi_1^5\Phi_2 ^3\Phi_3^2\Phi_4^2\Phi_5\Phi_6\Phi_8 \Phi_9\Phi_{12} $& 
$3$  &   $(q^2-1)_+ =  8 $ \\

(vi)& -& - & $\Phi_3^3$ & $1$ &  
$\Phi_1^6\Phi_2 ^4\Phi_4^2\Phi_5\Phi_6^2\Phi_8 \Phi_9\Phi_{12} $& $0$  &   \\

(vii)&-& - & $\Phi_2^2  \Phi_3\Phi_6$ &  $1$ & 
$\Phi_1^6\Phi_2 ^2 \Phi_3^2\Phi_4^2\Phi_5\Phi_6\Phi_8 \Phi_9\Phi_{12} $& 
$2$  & $q\equiv 1$ mod  $4 $  \\

(viii)&-& - & $\Phi_3  \Phi_{12}$ &  
$1$ & $\Phi_1^6\Phi_2 ^4 \Phi_3^2\Phi_4^2\Phi_5\Phi_6^2\Phi_8 \Phi_9$ & 
$0$  &   \\

(ix) & -& - & $\Phi_9  $ &  
$1$ & $\Phi_1^6\Phi_2 ^4 \Phi_3^3\Phi_4^2\Phi_5\Phi_6^2\Phi_8 \Phi_{12}$ 
& $0$  &   \\

(x)&-& - & $\Phi_3\Phi_6^2  $ &  $1$ &  
$\Phi_1^6\Phi_2 ^4 \Phi_3^2\Phi_4^2\Phi_5\Phi_8 \Phi_9\Phi_{12}$ & 
$0$  &   \\

(xi) &$A_2$ & $A_2(q)$ & $\Phi_3^2$ & $\chi^{(2,1)}$&  
$\Phi_1^4\Phi_2 ^4 \Phi_4^2\Phi_5\Phi_6^2\Phi_8 \Phi_9\Phi_{12}$  & 
$2$  & $q\equiv 3 $ mod  $4 $\\

(xii) & $A_2$ & $\, ^2A_2(q)$ & $\Phi_3\Phi_6$ & $\chi^{(2,1)}$&    
$\Phi_1^6\Phi_2 ^2 \Phi_3\Phi_4^2\Phi_5\Phi_8 \Phi_9\Phi_{12}$  & 
$2$  & $q\equiv 1 $ mod  $4 $\\

 (xiii)& $D_4$ & $\, ^3D_4(q)$ & $\Phi_3$ & $\phi_{2,2} $ &  
$\frac{1}{2}\Phi_1^4\Phi_2 ^4 \Phi_4^2\Phi_5\Phi_8 \Phi_9\Phi_{12}$  & 
$3$  & $q\equiv 3 $ mod  $4 $\\

(xiv)& $D_4$ & $\, ^3D_4(q)$ & $\Phi_3$ & $\phi_{2,1} $  & 
$\frac{1}{2}\Phi_1^4\Phi_2 ^4 \Phi_4^2\Phi_5\Phi_6^2\Phi_8 \Phi_9$  & 
$3$  & $q\equiv 3 $ mod  $4 $\\

(xv)&$D_4$ & $\, ^3D_4(q)$ & $\Phi_3$ & $\,^3D_4[-1]$&    
$\frac{1}{2}\Phi_1^6\Phi_2 ^2\Phi_3^2 \Phi_4^2\Phi_5\Phi_8 \Phi_9$  & 
$3$  & $q\equiv 1$ mod  $4 $\\

(xvi)& $D_4$ & $\, ^3D_4(q)$ & $\Phi_3$ &  $\,^3D_4[1]$&  
$\frac{1}{2}\Phi_1^6\Phi_2 ^2\Phi_4^2\Phi_5\Phi_8 \Phi_9\Phi_{12}$  & 
$3$  & $q\equiv 1$ mod  $4 $\\

(xvii)&$3A_2$ & $ A_2(q^3)$ &  $1$ & $\chi^{(2,1)}$&   
$\Phi_1^4\Phi_2 ^4\Phi_3\Phi_4^2\Phi_5\Phi_8 \Phi_{12}$  & $2$  & 
$q\equiv 3$ mod  $4 $\\

(xviii)& $E_6$ & $ E_6(q)$ &  $1$ &   $E_6[\theta]$, $E_6[\theta^2]$ &  
$\frac{1}{3}\Phi_1^6\Phi_2^4\Phi_4^2\Phi_5\Phi_8 $  & $0$  & \\

\hline

\end{tabular} 
\end{center}
}

\noindent
Here  $\Phi_d=\Phi_d(q)$  denotes the $d$-th cyclotomic polynomial over $\Q$
evaluated at $q$. The non-blank entries of the third column are to be 
interpreted  as $\Cb^{\circ}(s)'^{F^*}$ being a group of the given type, 
and do not specify the isomorphism type of $\Cb^{\circ}(s)'^{F^*}$.
\end{Proposition}

\begin{proof}  
Let $ s$, $\chi $, $\lambda $ be as  in the statement. 
The group  $\Cb (s)/\Cb^{\circ}(s) $ 
is isomorphic to  a subgroup of 
$\Irr_{\bar {\Q}_{2'}}( Z(\Gb)/Z^{\circ}(\Gb) )$ (see for instance 
\cite[Lemma 13.14(iii), Remark 2.4]{Digne/Michel:1991}).   
The first  two assertions follow since  $Z(\Gb)/Z^{\circ}(\Gb) $ is a 
cyclic group of order $3$.  

We will use (\ref{matchtwodefect}), and Propositions \ref{2defectsclassical}
and \ref{2defectsexceptional} in conjunction with  the tables giving 
connected centralisers  of semi-simple elements in groups of type  $E_6(q)$
in \cite{Deri84} in order to identify 
the $\Delta_s$, $\Cb^{\circ}(s) ^{F^*} $, $z(s)$ and $\bar \lambda $ 
entries  of a row in the table. Once these have 
been identified, the $\frac{ \chi(1)_{r'}} {\alpha_{\chi}}$ entry of 
the row is  calculated using  (\ref{matchchardegfinal}), the  degree  
formulae for unipotent characters as given in 
\cite[\S 13.8]{Carter93} and the order formulae for the finite groups 
of Lie type (see \cite[pp.75-76]{Carter93}). The last 
two entries  are obtained by comparing 
$\frac{ \chi(1)_{r'}} {\alpha_{\chi}}  $  with $|\Gb^F|$.
We note here that 
$$|\Gb^F| 
=  q^{36} \Phi_1^6 \Phi_2^4 \Phi_3^3\Phi_4^2\Phi_5\Phi_6^2
\Phi_8\Phi_9\Phi_{12}.$$
The connected components of $\Delta_s$  are of type $A_l$, 
$l\geq 1 $, $D_4$, $D_5 $, or $E_6 $. If $\Delta_s$ has a component of type
$E_6 $ (and in particular is a single component), 
then $s$ is  central in  $\Gb^* $,  $z(s)=1 $ and  $\bar\lambda $ 
is one of the  
unipotent characters $E_6[\theta] $ or $E_6[\theta^2] $  of  
$E_6(q)$ as in Proposition \ref {2defectsexceptional}.  Both  
$E_6[\theta] $ and $E_6[\theta^2]$ have degree  
$\frac{1}{3}q^7\Phi_1^6\Phi_2^4\Phi_4^2\Phi_5\Phi_8$, 
hence  $\chi, $ $s$, and $\bar \lambda $ 
are as in the last row of the table.

Assume from now on that the components  of 
$\Delta_s$ are of classical type. So   $\bar{\Cb}^{\circ}(s) ^{F^*} $ 
is a  direct product  of  groups 
$A_l(q^j)$, $l\geq 1$,  $\,^2A_l(q^j)$, $l\geq 2$,  $D_l(q^j), 
\,^2D_l(q^j) $,  $l\geq 4 $ or $\, ^3D_4(q)$.
By Propositions \ref{2defectsclassical}
and \ref{2defectsexceptional}, the $2$-defect of 
any unipotent character of  a  group 
$A_l(q^j)$, $l\geq 1$, $ \,^2A_l(q^j) $, $l\geq 2$,  $D_l(q^j), 
\,^2D_l(q^j) $,  $ l\geq 4 $  or  $ \, ^3D_4(q)$    is at least $2$. 
On the other hand, by  (\ref{matchtwodefect}), $\bar \lambda$ has 
$2$-defect  at most $3$ (note that $\alpha_{\chi}=0 $). Thus,   the connected 
components of $ \Delta(s) $   form one orbit under the action of $ F^*$, 
and in particular are all pairwise isomorphic.  It  also follows from 
Proposition \ref{2defectsclassical}  that the connected components are 
not of type  $D_5$ or  of type $A_l $, $l\geq 3 $.  

Let $j$ be the number of connected components of $\Delta_s$.
We claim that the connected components  of $\Delta_s $ are not of type $A_1 $. 
Indeed, assume otherwise, so  $\Cb^{\circ }(s)'^{F^*}  =  A_1 (q^j)$.  
By Proposition \ref{2defectsclassical}(i)-(ii), 
every unipotent character of $A_1(q^j) $ or  $\,^2A_1(q^j) $ has $2$-defect at 
least $3$. Thus, by (\ref{matchtwodefect}), $\zeta_s=0 $. 
But by  \cite{Deri84}, whenever $\Delta_s$ is  of type 
$jA_1$, $\zeta_s \geq 1$, proving the claim.

Suppose that the connected components of $\Delta_s $ are of type 
$A_2 $,  $1\leq j \leq 3 $.
If $ j= 3 $, then by  
\cite{Deri84},  $\Cb^{\circ }(s)'^{F^*}  =  A_2 (q^3)$, 
$z(s)=1 $, and  by Proposition \ref{2defectsclassical}(i), 
$\bar \lambda $ is the character $\chi^{(2,1)} $. 
By the  hook-length degree formula 
for unipotent characters of   groups of type $A$ (see the proof of 
Proposition \ref{2defectsclassical}(i)),
\[\bar\lambda(1)_{r'} =   \frac { (q^3-1)|A_2(q^3)| }{ (q^9-1)(q^3-1)^2 } =
\frac { |A_2(q^3)| }{ \Phi_1^2\Phi_3^2\Phi_9} , \]
we see that $s$ and $\chi $ are as in row (xvii)  of the table.  
If  $ j=2 $, then by \cite{Deri84},  
$\Cb^{\circ}(s)'^{F^*}(s)  =  A_2 (q^2) $,  whereas by Proposition 
\ref{2defectsclassical}  any unipotent character of $ A_2 (q^2) $ 
has $2$-defect at      least $6$ (note that $(q^2-1)_+ \geq 8 $). 
Thus this case  does not occur. 

If $j=1 $,  and  
$\Cb^{\circ }(s)'^{F^*} =   A_2 (q) $, then  $\zeta_s \leq 1  $, hence
by  \cite{Deri84},  $z(s) $ is  $(q^2+q+1)$ and  $s$ and $\chi $ are  
as in  row  (xi) of the 
table. If $j=1 $,  and  
$\Cb^{\circ }(s)'^{F^*} =   \,^2A_2 (q) $, then  $\zeta_s \leq 1  $, hence
by  \cite{Deri84},  
$z(s) $ is  $(q^4+q^2+1)$ and  $s$ and $\chi $ are  as in row  
(xii) of the table.
 
Suppose that the connected components  of $\Delta_s$  are of  type $ D_4 $, so
$j =1 $. By Proposition  \ref{2defectsclassical}(iv), 
and  by \cite{Deri84}, $\Cb^{\circ }(s)'^{F^*} =   \,^3D_4 (q) $  
and $z(s)= q^2+q+1 $. By
Proposition \ref{2defectsexceptional}(vii)  $\bar \lambda $ is one of 
$\phi_{2,2} $  or $\phi_{2,1} $  if $q\equiv 3 $ (mod $4$) and 
$\bar \lambda $ is one of $\,^3D_4[-1] $  or $ \,^3D_4[1]  $ 
if $q\equiv 1 $ (mod $4$). Hence,  
$\chi $, $s$  and $\lambda $ are as in 
lines   (xiii)-(xvii) of the table.

Finally assume that $ \Cb^{\circ}(s) $ is a torus, so 
$\Cb^{\circ }(s)'$ and $\bar\lambda $ are trivial.
By  (\ref{matchtwodefect}) we have
$z(s)= d(\chi)  \leq 3$.  
An inspection of \cite{Deri84} yields  that  $z(s)$  
is as in the first  ten  rows of the table.
\end{proof}

\begin{Proposition}\label{small2defect2E_6}   
Suppose that $\Gb$ is simple, simply connected 
of  type $E_6$, and  $\Gb^F$ is a group of type  $\,^2E_6(q)$.  
Let $s$ be a semi-simple element of  $\Gb^{*F^*}$ and let 
$\chi \in \Irr_K(G ) \cap  \CE(\Gb^F, [s]) $. Let 
$ \bar \lambda $ correspond to
$\chi $ as in (\ref{matchchardegfinal}).
Suppose  that the $2$-defect $d(\chi)$ of $\chi$ is at 
most $3$. Then $a_{\chi} $ is $1$ or $3$, $\alpha_{\chi}=0 $
 and  $s$, $\chi $ and $ \bar \lambda $   satisfy  
one of the rows in the following table.

{\rm
\footnotesize
\begin{center}
\begin{tabular}{l|c|c|l|c|l|c|c}
\hline
& & & & & & & \\
& $\Delta_s$ &  $\Cb^{\circ}(s)'^{F^*}  $ &  $z(s)$ & $\bar \lambda $ & 
$ a_{\chi}\chi(1)_{r'} $ 
& $d(\chi) $  &     cond. on $q$\\
&& & & & && \\
\hline
(i)&-& -  & $\Phi_2^2\Phi_6 ^2$ &  $1$& 
$\Phi_2^4\Phi_1 ^4\Phi_6\Phi_4^2\Phi_{10}\Phi_{3}^2\Phi_8 \Phi_{18}\Phi_{12}$ & 
$2$  & $q\equiv 1 $ mod  $4 $\\

(ii)& -& - & $\Phi_2^2\Phi_{10}$ & $1$ &  
$\Phi_2^4\Phi_1 ^4\Phi_6^3\Phi_4^2\Phi_3^2\Phi_8 \Phi_{18}\Phi_{12} $& 
$2$  & $q\equiv 1 $ mod  $4 $\\

(iii)& -& - & $\Phi_2\Phi_1\Phi_6^2$ & $1$ &   
$\Phi_2^5\Phi_1 ^3\Phi_6\Phi_4^2\Phi_{10}\Phi_3^2\Phi_8 \Phi_9\Phi_{12} $& 
$3$  &   $(q^2-1)_+ =  8 $ \\

(iv)& -& - & $\Phi_2\Phi_1\Phi_{10}$ & $1$ &  
$\Phi_2^5\Phi_1 ^3\Phi_6^3\Phi_4^2\Phi_3^2\Phi_8 \Phi_{18}\Phi_{12} $& 
$3$  &  $(q^2-1)_+ =  8 $    \\

(v)& -& - & $\Phi_2\Phi_1\Phi_6\Phi_3$ & $1$ &   
$\Phi_2^5\Phi_1^3\Phi_6^2\Phi_4^2\Phi_{10}\Phi_3\Phi_8 \Phi_{18}\Phi_{12} $& 
$3$  &   $(q^2-1)_+ =  8 $ \\

(vi)& -& - & $\Phi_6^3$ & $1$ &  
$\Phi_2^6\Phi_1 ^4\Phi_4^2\Phi_{10}\Phi_9^2\Phi_8 \Phi_{18}\Phi_{12} $& 
$0$  &   \\

(vii)& -& - & $\Phi_1^2  \Phi_6\Phi_3$ &  $1$ & 
$\Phi_2^6\Phi_1 ^2 \Phi_6^2\Phi_4^2\Phi_{10}\Phi_3\Phi_8 \Phi_{18}\Phi_{12}$& 
$2$  & $q\equiv 3$ mod  $4 $  \\

(viii)&-& - & $\Phi_6  \Phi_{12}$ &  $1$ & 
$\Phi_2^6\Phi_1 ^4 \Phi_6^2\Phi_4^2\Phi_{10}\Phi_3^2\Phi_8 \Phi_{18}$ & 
$0$  &   \\

(ix)&-& - & $\Phi_{18}  $ &  $1$ & 
$\Phi_2^6\Phi_1 ^4 \Phi_6^3\Phi_4^2\Phi_{10}\Phi_3^2\Phi_8 \Phi_{12}$ & 
$0$  &   \\

(x)& -& - & $\Phi_6\Phi_3^2  $ &  $1$ &  
$\Phi_2^6\Phi_1 ^4 \Phi_6^2\Phi_4^2\Phi_{10}\Phi_8 \Phi_ {18}\Phi_{12}$ & 
$0$  &   \\

(xi)& $A_2$ & $\,^2A_2(q)$ & $\Phi_6^2$ & $\chi^{(2,1)}$&  
$\Phi_2^4\Phi_1 ^4 \Phi_4^2\Phi_{10}\Phi_3^2\Phi_8 \Phi_{18}\Phi_{12}$  & 
$2$  & $q\equiv 1 $ mod  $4 $\\

(xii)& $A_2$ & $A_2(q)$ & $\Phi_6\Phi_3$ & $\chi^{(2,1)}$&    
$\Phi_2^6\Phi_1 ^2 \Phi_6\Phi_4^2\Phi_{10}\Phi_8 \Phi_{18}\Phi_{12}$ & 
$2$  & $q\equiv 3 $ mod  $4 $\\

(xiii)& $D_4$ & $\, ^3D_4(q)$ & $\Phi_6$ & $ \,^3D_4[1] $ &  
$\frac{1}{2}\Phi_2^4\Phi_1 ^4 \Phi_4^2\Phi_{10}\Phi_8 \Phi_{18}\Phi_{12}$ & 
$3$  & $q\equiv 1 $ mod  $4 $\\

(xiv)& $D_4$ & $\, ^3D_4(q)$ & $\Phi_6$ & $\,^3D_4[-1] $  & 
$\frac{1}{2}\Phi_2^4\Phi_1^4 \Phi_4^2\Phi_{10}\Phi_3^2\Phi_8 \Phi_{18}$ & 
$3$  & $q\equiv 1 $ mod  $4 $\\

(xv)& $D_4$ & $\, ^3D_4(q)$ & $\Phi_6$ & $\phi_{2,1}$&    
$\frac{1}{2}\Phi_2^6\Phi_1 ^2\Phi_6^2 \Phi_4^2\Phi_{10}\Phi_8 \Phi_{18}$ & 
$3$  & $q\equiv 3$ mod  $4 $\\

(xvi)&$D_4$ & $\, ^3D_4(q)$ & $\Phi_6$ &  $\phi_{2,2}  $&  
$\frac{1}{2}\Phi_2^6\Phi_1 ^2\Phi_4^2\Phi_{10}\Phi_8 \Phi_{18}\Phi_{12}$ & 
$3$  & $q\equiv 3$ mod  $4 $\\

(xvii)& $3A_2$ & $ \,^2A_2(q^3)$ &  $1$ & $\chi^{(2,1)}$&   
$\Phi_2^4\Phi_1 ^4\Phi_6\Phi_4^2\Phi_{10}\Phi_8 \Phi_{12}$ & 
$2$  & $q\equiv 1$ mod  $4 $\\

(xviii)&$E_6$ & $ \,^2E_6(q)$ &  $1$ &   $\,^2E_6[\theta]$, 
$\,^2E_6[\theta^2]$ &  $\frac{1}{3}\Phi_2^6\Phi_1^4\Phi_4^2\Phi_{10}\Phi_8$  & 
$0$  & \\

\hline

\end{tabular} 
\end{center}
}
\end{Proposition}

\begin{proof}  This is  exactly as in  the proof of Proposition 
\ref{small2defectE_6}, replacing  $q$  by $-q$ and  noting that 
this switches the roles   of 
the pairs $\Phi_1$ and $\Phi_2$, $\Phi_3$ and $\Phi_6$, $\Phi_5$ and 
$\Phi_{10}$, and $\Phi_{9}$ and $\Phi_{18}$.
\end{proof}

\section{Type $E$ in odd characteristic}

Let $q$ be an odd prime power. 
We do not know whether there are
any blocks with elementary abelian defect group of
order $8$ in type $E$ over a field of odd
characteristic - the next result says that {\it if} 
there is such a block, it is either nilpotent or its 
inertial quotient has order $7$.

\begin{Proposition} \label{typeEfusion}
Let $G$ be a quasi-simple finite group such that $Z(G)$ has 
odd order and such that $G/Z(G)$ is a simple group of type 
$E_6(q)$, ${^2E_6(q)}$, $E_7(q)$ or $E_8(q)$. If $kG$ has a block 
$b$ with an elementary abelian defect group $P$ of order $8$, then
the inertial quotient of $b$ is either trivial (in which case
$b$ is nilpotent) or has order $7$.
\end{Proposition}

\begin{proof}
It suffices to show that if $u\in P$ is an involution and
if $e$ is a block of $kC_G(u)$ with defect group $P$ then
$e$ is nilpotent. Indeed, there are non-nilpotent $(G,b)$-Brauer
elements if and only if the inertial quotient of $b$ has order $3$ or
$21$. The pattern of the proof is this: using the lists of
centralisers in \cite[4.5.1, 4.5.2]{GorLyoSol3}, we show
that in `most' cases, $C_G(u)$ has a direct factor of the form 
$2.H.2$ for some finite group $H$. As in \cite{CEKL}, this
means that $2.H$ is a central extension of $H$ by an involution
such that $2.H$ is contained as a subgroup of index $2$ in
$2.H.2$, and an automorphism $\alpha$ of $2.H$ induced by conjugation
with a $2$-element $a$ in $2.H.2-2.H$ has an image $\bar\alpha$ in 
the outer automorphism group of $2.H$ of order $2$.  Using 
\cite[6.4]{CEKL} and again \cite[4.5.1, 4.5.2]{GorLyoSol3}, we will
then show that $\alpha$ 
stabilises all blocks of $H$, implying that $a$ is  
contained in a defect group of any block, and so any block of 
$2.H.2$ with an elementary abelian defect group $P$ is necessarily
nilpotent, because its image in $H.2$ is a block with
a Klein four defect group which is also a $2$-extension
of a block of $H$ with a defect group of order $2$. 
Here are the details for the different groups; this
follows arguments in \cite[\S 12]{CEKL}. All group
theoretic facts about centralisers of involutions are from
\cite[\S 4]{GorLyoSol3}.

\smallskip
Suppose that $G/Z(G)$ is of type $E_6(q)$. 
There are two classes of involutions in $G$, with representatives
$t_1$, $t_2$. If $u=t_1$ then $C_G(u)$ has a central cyclic
subgroup of order $\gcd(4,q-1)$, hence $q\equiv 3\ (\mod \ 4)$ since
$P$ has exponent $2$. Then $C_G(u)\cong$ 
$2.\POmega_{10}(q).2\times C_{(q-1)/2}$. The automorphism $\alpha$
of $\POmega_{10}$ as constructed above is inner diagonal, hence 
\cite[6.4]{CEKL} implies
that it is contained in a defect group of any block of $C_G(u)$.
Thus any block of $C_G(u)$ with $P$ as defect group is nilpotent
by the argument outlined at the beginning of the proof.
If $u=t_2$ then $C_G(u)\cong$ $2.(\PSL_2(q)\times Z(N).\PSL_6(q)).2$.
Again, $\alpha$ is inner diagonal on both
factors, and so the same argument using \cite[6.4]{CEKL} shows
that any block of $C_G(u)$ with defect group $P$ is nilpotent.

\smallskip
Suppose that $G/Z(G)$ is of type ${^2E_6(q)}$. 
There are two classes of involutions in $G$, with representatives
$t_1$, $t_2$. If $u=t_1$ then $C_G(u)$ has a central cyclic
subgroup of order $\gcd(4,q+1)$, hence $q\equiv 1\ (\mod \ 4)$ since
$P$ has exponent $2$. Then $C_G(u)\cong$ 
$2.\POmega^-_{10}(q).2\times C_{(q+1)/2}$. 
If $u=t_2$ then $C_G(u)\cong$ 
$2.(\PSL_2(q)\times (3,q+1).\PSU_6(q)).2$, and the same argument
as for $E_6(q)$ shows that in both cases, any block of $kC_G(u)$
with defect group $P$ is nilpotent.

\smallskip
Suppose that $G/Z(G)$ is of type $E_7(q)$. 
In that case $Z(G)=\{1\}$, so $G$ is simple, but 
the finite group of Lie type $E_7(q)$ is a central extension
of $G$ by an involution. Let $E=$ $\Inndiag(G)$; then
$|E:G|=2$. Table 4.5.1 in \cite{GorLyoSol3} describes
$C_E(u)$, for any involution $u$ in $G$, and involutions
in $G$ which are $E$-conjugate will have isomorphic
centralisers in $G$. There are five $E$-conjugacy classes
of involutions in $G$, represented by $t_1$, $t_4$, $t'_4$,
$t_7$ and $t'_7$. If $u=t_1$ then $C_E(u)\cong$
$2.(\PSL_2(q)\times \POmega_{12}(q)).(C_2\times C_2)$, so
our standard argument implies that any block of $C_G(u)$ with
$P$ as defect group is nilpotent. If $u=t_4$ then
$q\equiv 5\ (\mod \ 8)$ (else $t_4$ is non-inner), and
$C_E(u)\cong$ $2.\PSL_8(q).C_{\gcd(8,q-1)}.\gamma$, where
$\gamma$ is a graph automorphism. It follows from
considering the table 4.5.2 in \cite{GorLyoSol3} applied to the
inverse image of $u$ in $E_7(q)$ that $\gamma$ is not in $G$,
hence $C_G(u)\cong$ $2.\PSL_8(q).C_{\gcd(8,q-1)}$.
Since $\gcd(8,q-1)$ is at least $2$, we get again that any
block of $C_G(u)$ with $P$ as defect group is nilpotent.
If $u=t'_4$ then $q\equiv 3\ (\mod \ 4)$ (or else $t'_4$ is non-inner)
and $C_E(u)\cong$ $2.\PSU_8(q).C_{\gcd(8,q-1)}.\gamma$. As in the
previous case we get 
$C_G(u)\cong$ $2.\PSU_8(q).C_{\gcd(8,q-1)}$, and thus that
any block of $C_G(u)$ with $P$ as defect group is nilpotent.
If $u=t_7$ then $q\equiv 1\ (\mod \ 4)$ and $C_E(q)\cong$
$(\gcd(3,q-1).E_6(q).3 * C_{q-1}).2$, and hence, by an argument
as before (using \cite[4.5.2]{GorLyoSol3}), we get that
$C_G(u)\cong$ $(\gcd(3,q-1).E_6(q).3 * C_{q-1})$.
This shows that $C_G(u)$ has a normal cyclic subgroup of order
$4$, hence has no block with $P$ as defect group.
Similarly, if $u=t'_7$ then $q\equiv 3\ (\mod \ 4)$ and $C_E(q)\cong$
$(\gcd(3,q+1).{^2{E_6(q)}}.3 * C_{q+1}).2$; as above we get
$C_G(u)\cong$ $(\gcd(3,q+1).{^2{E_6(q)}}.3 * C_{q+1})$.
This shows that $C_G(u)$ has again a normal cyclic subgroup of order
$4$, hence has no block with $P$ as defect group.

\smallskip
Suppose finally that $G/Z(G)$ is of type $E_8(q)$. Then 
$Z(G)=\{1\}$ and $G$ has two conjugacy classes of involutions,
with representatives $t_1$ and $t_8$. Their centralisers are
isomorphic to $2.\POmega_{16}(q).2$ and
$2.(\PSL_2(q)\times \bar E_7(q)).2$, where $\bar E_7(q)$ is the simple
quotient of $E_7(q)$. In both cases our standard argument shows that
every block with defect group $P$ is nilpotent.
\end{proof}

{\bf Notation.} \label{Zsigmondy}    
For each $n > 2 $, let $p_n $ be a Zsigmondy prime for the pair  
$(q,n)$, i.e., $p_n $ is a prime number  such that  $p_n $ divides 
$q^n-1 $ but $p_n $ does not divide $q^m-1 $ for any $m < n $.      
Such $p_n$ exist by Zsigmondy's theorem 
(see \cite[Theorem 3]{Roit}), and note that $q$ is odd and that $n$ 
is assumed greater than $2$. Since $q$ has order $n$ modulo $p_n$,
if $p_n$ divides $\Phi_d(q)$ for some natural number $d$
then $p_n$ divides $q^d-1$, whence $n|d$.


\begin{Proposition} \label{typeE_6punch}
Let $G$ be a quasi-simple finite group such that $Z(G)$ has 
odd order and such that $G/Z(G)$ is a simple group of type 
$E_6(q)$. If $kG$ has a block $b$ with an elementary abelian defect group 
$P$ of order $8$, then  either  $|\Irr_K(G, b)|=8 $ or  $ \CO  G b $ 
is Morita equivalent to a  block  of $\CO L $ of a finite group $L$ such that
$|L/Z(L)|< |G/Z(G) | $. 
\end{Proposition}

\begin{proof} 
The centre $Z(G) $ has order $1$ or $3$ and we may 
assume without loss of generality that  $Z(G) $ has order $3$.  Hence 
$G = \Gb^F $, where $\Gb $ is a simply connected 
simple  algebraic group of type $E_6$ over $\bar {\mathbb F}_q$
and  $F: \Gb \to \Gb  $ is a Frobenius morphism with respect to an 
${\mathbb F}_q $-structure on $\Gb $.   
Keeping the  notation of the 
previous section let $[t]$ be the semi-simple label of $b$.
Arguing as for Proposition \ref{typeFpunch}, we may and we will 
assume that $t$ is quasi-isolated.    
Suppose, if possible that  
$|\Irr_K(G, b)| \ne 8 $.   
By Propositions \ref{punchheight} and \ref{typeEfusion},  
$$ |\Irr_K(G, b)| = 5
\ \  {\text{and}} \  \
\Irr_K(G, b) = \{\chi_j , \ 1 \leq j \leq 5 \}$$ 
where $\chi_1 $ has height one and $\chi_j $ has height zero 
for $ 2\leq j \leq 5 $.  
For each  $j$, $1\leq j \leq 5 $, let $s_j $ be a semi-simple element in 
$\Gb^{*F^*} $ such that $ \chi_j \in \CE(\Gb, [s_j]) $.  
By \cite[Th\'eor\`eme 2.2] {BroMic89}, for any $ 1\leq j \leq 5 $, $s_j$  
can be chosen to have the form $ s_j= tu_j$, where $u_j$ is a $2$-element of  
$\Cb(t)^{F^*}$. Since $\Cb(t) /\Cb^{\circ}(t) $ has exponent 
dividing the order of $t$, 
(see for example \cite[Remark 13.15 (i)]{Digne/Michel:1991}), 
$ u_j \in \Cb^{\circ}(t)^{F^*} $. In particular, the connected components of 
$\Delta_{s_j} $  are subdiagrams of the extended connected components of 
$\Delta_t $. Here, by an extended   Dynkin diagram we mean a completed 
Dynkin diagram in the  sense of \cite[Chapter 6, \S  4.3]{bourbaki:lie:4-6}.

Since $t$ is  quasi-isolated and  has odd order, 
by \cite[Table III]{Bon2004}, 
$\Delta_t$ is one of $E_6 $, 
$A_2\times A_2 \times A_2$, or $D_4 $. On the other hand, since 
$\CE(\Gb, [t])\cap \Irr_K(G, b) \ne \emptyset $,  $\CE(\Gb, [t]) $ and hence 
$\CE(\Gb, (t)) $  contains an element of  
$2$-defect $2$ or $3$. By  Equation \ref{matchtwodefect}  and 
Proposition \ref{small2defectE_6}, it follows that $C_{\Gb}(t)^{F^*} $ 
contains a unipotent character of $2$-defect at most  $3$ and at least $2$.  
Now if $\Delta_t  $ is of type $E_6 $, then  $t$ is central    in $\Gb^* $, 
i.e. $C_{\Gb}(t)^{F^*} = E_6(q)$. But by 
Proposition \ref{2defectsexceptional} 
the $2$-defect of a unipotent character of a group of type $ E_6(q)$ 
is either $0$ or at least $8$, a contradiction.

Suppose   that  $\Delta_t =3A_2$. Then, since   $C_{\Gb}(t)^{\circ F^*} $  
contains a unipotent character  
of $2$-defect $2$ or $3$, we see that 
$t$ is as in row  (xvii) of Table \ref{small2defectE_6}. 
But since the characters corresponding to (xvii) have $2$-defect $2$, 
it follows by Proposition \ref{punchheight} 
that $u_1= 1 $ and 
that $u_j $ is non-trivial for all $j$, $2 \leq j\leq 5 $.

Let $ 2\leq j \leq 5 $.  Since the $2$-defect of $s_j $ is  $3$ and since 
$\Delta_{s_j}$ is not   of  type $D_4 $ or $E_6$,     
$C_{\Gb^*}(s_j)^{\circ} $ is a torus   corresponding to rows (iii), (iv) or (v)
of Table \ref{small2defectE_6}. In particular, 
$\Phi_{9}(q)$ and hence $p_9$  is a  divisor of  
$\chi_j(1)$ for all $j$,  $ 2\leq j \leq 5 $. 
On the other hand, $\chi_1 $ corresponds  to row (xvii) of 
Table \ref{small2defectE_6} and 
we see that $p_9 $ does not divide $\chi_1(1) $.
Hence, $p_9 $ does not divide
$2\chi_1(1) -\sum_{2\leq j\leq 5} \delta_j\chi_j(1)$. In particular, 
$2\chi_1(1) -\sum_{2\leq j\leq 5} \delta_j\chi_j(1) \ne 0$,  
contradicting  Proposition 
\ref{punchheight}(ii).  
  
So $\Delta_t = D_4$.  Let $ j_0  $ be such that  $s_{j_0} =  t $, 
$ 1\leq j_0 \leq 2 $.  Then $t$ corresponds to one of rows 
(xiii)-(xvi) of Table \ref{small2defectE_6}. Since the characters 
corresponding to these rows have $2$-defect $2$,  
$j_0 \geq 2 $, say $ j_0=2 $. 
Let us first consider the case that  
$\chi_{2} $ corresponds to row (xiii)  of the table, so
$q \equiv 3$ (mod $4$) and let $l$, $ 3\leq l \leq 5 $  be such that 
$(s_l, \chi_l)$ also corresponds to row  (xiii) of the table.
Then semi-simple part of $C_{\Gb^*}(s_l)^{\circ}$, and hence  of  
$C_{C_{\Gb^*}(t)^{\circ}} (u_l) $  is of type $D_4$.  Thus  
$u_l $ is central in $C_{\Gb^*}(t)^{\circ}$. 
Since $u_l $ is a $2$-element, 
and  $\zeta_t$ is odd, we get that 
$u_l $ is a central element of $[C_{\Gb^*}(t)^{\circ},C_{\Gb^*}(t)^{\circ}] $.
But the  group  $\,^3D_4(q)$  has a trivial centre, hence $u_1 =1 $.
Thus, if $(s_l, \chi_l) $ corresponds  to  row (xiii) of 
Table \ref{small2defectE_6}, then $[s_l]=$ $[t]$.
By \cite{Bon2004}, $t$ is quasi-isolated but not isolated and 
$C_{\Gb^*}(t)^{\circ} $ is of index $3$ in $C_{\Gb^*}(t)$.    
Since $\phi_{2,1}$ is the unique unipotent character of  
its degree in $\, ^3D_4(q)$, $\phi_{2,1}$ is stable under 
$C_{\Gb^{F^*}}(t)$ and hence there are three possibilities for $\chi_l$, 
each of degree  
$$ \delta =\frac{1}{6}q^{33}
\Phi_1^4(q)\Phi_2 ^4(q) 
\Phi_4^2(q)\Phi_5(q)\Phi_8(q) \Phi_9(q)\Phi_{12}(q). $$ 

Let $\chi_1 -\sum_{2\leq  i\ leq 5} \delta_i \chi_i $ be the element of  
$L^0(G,b)$ as in   Proposition \ref{punchheight}.  Then, 
$\delta_l=$ $\delta_{2}$. Thus, from the degree formula above, 
it follows that $\sum_{i} \delta_i \chi_i(1)$, where $i$ ranges over the 
indices for which  $s_i $ corresponds to row (xiii) of Table 
\ref{small2defectE_6} is not divisible by the Zsigmondy prime $p_6$ 
(note that the number of such indices is at most $3$).

Now let $j$, $1\leq j \leq 5 $ be such that $s_j$ does not correspond to 
row (xiii) of Table \ref{small2defectE_6}.
Since $ q\equiv 3 $ (mod $4$), and  $\Delta_{s_j} $ is a subdiagram of the 
extended Dynkin diagram of $\Delta_t$, 
$(s_j, \chi_j)$ 
corresponds to one of rows (i), (ii), (iii), (iv), (v), (xi), or  (xiv)    
of Table \ref{small2defectE_6}. In particular, 
from the character degree column of Table \ref{small2defectE_6}, 
we see that  $\chi_j(1)$ is 
divisible by $p_6 $  for any such $j$.  Hence,  
$2\chi_1(1)-\sum_{2\leq i \leq 5} \delta_i \chi_i(1) $  is not divisible by 
$p_6$, a contradiction.

We get a similar contradiction if $(t, \chi_2) $  corresponds to  row 
(xvi) of  Table \ref{small2defectE_6} and    
with $p_6 $ replaced by $p_{12}$ if
$(t, \chi_2) $  corresponds to  row (xiv) or row (xv)  of  
Table \ref{small2defectE_6}.
\end{proof}

\begin{Proposition} \label{type2E_6punch}
Let $G$ be a quasi-simple finite group such that $Z(G)$ has 
odd order and such that $G/Z(G)$ is a simple group of type 
$\,^2E_6(q)$. If $kG$ has a block $b$ with an elementary abelian 
defect group $P$ of order $8$, 
then either $|\Irr_K(G, b)|=8 $ or $\CO  G b $ 
is Morita equivalent to a block of $\CO L $ of a finite group $L$ such that
$|L/Z(L)|< |G/Z(G) | $. 
\end{Proposition}

\begin{proof} 
This is as the proof of Proposition \ref{typeE_6punch}, 
with the roles of $q$ and $-q $ suitably reversed.
\end{proof}

\begin{Proposition} \label{typeE_7punch} 
Let $G$ be a quasi-simple finite group such that $Z(G)$ has 
odd order and such that $G/Z(G)$ is a simple group of type 
$E_7(q)$.  Suppose that  
$kG$ has a block $b$ with an elementary abelian defect group 
$P$ of order $8$ such that  $|\Irr_K(G, b)| \ne 8 $. Then,   
there exists a  finite group $L$ with $|L/Z(L)| <  |G/Z(G)| $  and a block 
$c$ of $\CO L$ with elementary abelian defect groups of order $8$ 
and such that $|\Irr_K(L, c)| \ne 8 $.
\end{Proposition}

\begin{proof}  First note that $Z(G)=1 $  and that  
$ G = \tilde  G/\langle z\rangle  $,  where
$ \tilde G =  \tilde \Gb^F$   for a 
simply-connected simple  algebraic group $\tilde \Gb $   of type $E_7 $,  
 $F :\tilde \Gb \to \tilde \Gb $ is a Frobenius morphism   with respect to an 
${\mathbb F}_ q$-structure on $\tilde \Gb $, and  $\langle z\rangle$ 
is  the central subgroup of order $2$ of $\tilde \Gb$.  Let $\tilde b$ be the 
block of $\CO \tilde G $ lifting $b$ and let
$[t]$ be the semi-simple label of $\tilde b$.  

Suppose first that  
$t$  is  not a quasi-isolated element of $\tilde \Gb^* $. Then,  arguing as for
Proposition \ref{typeFpunch}, there exists a  proper $F$-stable 
Levi subgroup $\tilde \Lb $ of $\tilde \Gb $, and a block $\tilde c $ of 
$\CO \tilde \Lb ^F $ such that  $\CO \tilde\Gb^F \tilde b $ and 
$\CO\tilde\Lb^F \tilde c $  are Morita equivalent.  Noting  that  
$z \in \tilde \Lb^F $,  set $L=$ $\tilde \Lb^F/\langle z\rangle $ and let 
$c$ be the image in  $\CO L$ of $\tilde c$. By  Proposition 
\ref{defect4}, the  defect groups of  $c$ are elementary abelian 
and $c$ does not satisfy Alperin's weight conjecture. Since
$$ |L/Z(L)| \leq  L = \frac {|\tilde \Lb^F|}{2} < \frac {|\tilde \Gb^F|}{2} =
|G| = |G/Z(G)|, $$ the result follows.

We assume from now on that  $t$  is quasi-isolated  in $\tilde \Gb^* $ .  
Before proceeding we note that since $t$ has odd order, 
$\Cb_{\tilde \Gb^*}(t)$ is  connected and hence $t$ is in fact isolated.  
By table III of \cite{Bon2004}, and using that $t$ is of odd order, we 
get that either $t=1 $ and  the Dynkin diagram  $\Delta_t$  
corresponding to $C_{\tilde \Gb}(t)$ is  of type  $E_7$ or  
$t$ has order $3$ and $\Delta_t $ is of type  $A_2 \times A_5 $.  
Since there is an ordinary irreducible character of $\tilde b $ in the   
$\CE( \tilde \Gb^F, [t])$, it follows from  Equation 
(\ref{matchchardegfinal}) in \S \ref{redback} that 
$ (\Cb_{ \tilde \Gb^*}(t)/Z(\Cb_{ \tilde \Gb^*}(t)))^{F^*}$ has a 
unipotent character of defect at most $4$ 
(note that  $\tilde b $ has  defect groups of order $16$). By 
Proposition \ref{2defectsclassical} it follows that $\Delta_t$ is 
not  of type $A_2 \times A_5 $.  So,   $\Delta_t  $ is of type $E_7$, 
i.e., $\tilde b$ is a unipotent block. 
But  the defect groups of all non-principal  unipotent blocks of 
$\tilde \Gb^F  $  which are not central in $\tilde \Gb^F  $  are dihedral 
 groups (\cite[p.357]{Enguehard2000}). On the other hand, 
a defect group of $\tilde b$  is  a central extension of   
an elementary abelian group of order $8$ by a group of order $2$,   
a contradiction.
\end{proof}  

\begin{Proposition} \label{typeE_8punch} 
Let $G$ be a quasi-simple finite group such that $Z(G)$ has 
odd order and such that $G/Z(G)$ is a simple group of type 
$E_8(q)$. Suppose that  
$kG$ has a block $b$ with an elementary abelian defect group 
$P$ of order $8$ such that  $|\Irr_K(G, b)|\ne 8 $. Then,   
there exists a finite group $L$ with $|L/Z(L)| < |G/Z(G)| $  and a block 
$c$ of $\CO L $ with elementary abelian defect groups of order $8$ 
and such that $|\Irr_K(L, c)| \ne 8 $.
\end{Proposition}

\begin{proof}  Note  that $ G= \Gb^F$   for a 
simply-connected simple  algebraic group $ \Gb $   of type $E_8 $,  
 and $F : \Gb \to \Gb $  a Frobenius morphism   with respect to an 
${\mathbb F}_ q$-structure on $ \Gb $. Let
 $ [t]$ be the semi-simple label of $ b$.  
Again, the Bonnaf{\'e}-Rouquier result  allows us to reduce to the case   that 
$t$  is   a quasi-isolated  (hence isolated) element of   $\Gb^* $. 
By the tables in  \cite{Deri83}, one sees that $\Delta_t $ is one of 
$A_4 \times A_4 $, $A_5 \times A_2 \times A_1 $,   $A_7 \times  A_1 $, $A_8$,
$D_5  \times A_3 $ , $D_8$,  $E_6 \times A_2 $, $E_7 \times A_1  $ or $E_8 $.
Using  as before that  $ (C_{\Gb^*}(t)/Z(C_{\Gb^*}(t)))^{F^*}  $ 
has a unipotent character of defect at most $3$ 
(see Equation (\ref{matchchardegfinal}) in \S \ref{redback})
and  Propositions \ref{2defectsclassical} and \ref{2defectsexceptional}  
we have that $\Delta_t $ is one of 
$E_6 \times A_2 $, or $E_8 $.   
By \cite[p.364]{Enguehard2000}, the defect groups  
of any unipotent block of positive defect of $G$ have order at least $16$
hence $\Delta_t $ is not of type $E_8$.

Suppose that $\Delta_t$ is of type
$E_6 \times A_2 $ and set  $ \Cb = \Cb_{\Gb^*}(t) $.
Since $\Cb $ is connected and has centre of odd order,   
by Propositions \ref{2defectsclassical} and  
\ref{2defectsexceptional}, and the discussion preceding 
Proposition \ref{small2defectE_6}, $\Cb^{F^*}$ has no unipotent character
of defect $3$, and at most one unipotent character of $2$-defect $2$ 
(corresponding to the product of the unipotent character of $2$-defect 
$2$ of $A_2(q)$ or $\,^2A_2(q)$ and a unipotent character of 
$2$-defect zero of $E_6(q)$ or ${^2E_6(q)}$). We have
$ |\Irr_K(G, b)| = 5$ and can write 
$\Irr_K(G, b) = \{\chi_j , \   1 \leq j \leq 5 \}$
as in Proposition \ref{punchheight}.
For each $j$, 
$1 \leq j\leq 5 $, let $s_j $ and $u_j$ be as in the proof of 
Proposition \ref{typeE_6punch}. Let $ j_0$,  
$1\leq j_0 \leq 5 $ be   such that  $\chi_{j_0} $ is in the rational series
indexed by $[t]$. By the above discussion  $j_0  =1 $ and $\chi_1 $  
is the unique character of $b$ in the rational Lusztig series  
corresponding to $[t]$.  
In other words, $u_j$ is a non-trivial $2$-element for $ 2 \leq j \leq 5 $.  

Let $2 \leq j \leq 5 $. Set $\Cb_j=$ $C_{\Gb^*}(s_j) =$
$\Cb_{\Gb^*}^{\circ}(s_j)$,  $\Zb_j = $ $Z(C_{\Gb^*}(s_j))$ and 
$\bar\Cb_j= $ $\Cb_j/\Zb_j$. Set 
$z_j=$ $|\Zb_j^{F^*}|$ and  let $\zeta_j$ be defined  by 
$2^{ \zeta_j }= $ ${z_j}_+ $.  Further, let $\bar \lambda_j $  
be a unipotent character  of $ \bar \Cb_j^{F^*} $  corresponding to 
$\chi_j $  as in the   discussion preceding  Equation \ref{matchtwodefect}, 
i.e., such that  
\begin{equation}\label{E_8char} \chi_{j_0}(1) = 
\frac{ |G|_{r'} } {z_j |\bar C_j ^{F^*}|  } 
\bar \lambda_j(1) \end{equation}
and 
\begin{equation}  \label{E_8def}  
\zeta_j + d(\tau_j) =3 
\end{equation}
where $d(\tau_j)$ is the $2$-defect of $\tau_j$.
We have that $ \Cb := \Cb_{\Gb^*}(t) = \Xb.\Yb$ where 
$\Xb$ is a simply connected
group of type  $A_2 $ and $\Yb$ is a simple group of 
type $E_6 $, $\Xb$ and $\Yb$ commute and intersect  in a 
central subgroup of order $3$.   
As $u_j$ is a non-trivial $F^*$-stable $2$-element, 
there are unique $F^*$-stable $2$-elements $v_1 \in \Xb$, 
$v_2\in \Yb$, at least one of which is  non-trivial  
and such that $u_j =v_1v_2$,
$$ C_{\Gb^*} (s_j)  = C_{\Xb}(v_1)C_{\Yb}(v_2) $$ and 
$C_{\Gb}(s_j)^{F^*}$ contains  
$C_{\Xb}(v_1)^{F^*}C_{\Yb}(v_2)^{F^*}$ 
as a normal subgroup of index $3$.

In particular, $\langle v_1,v_2 \rangle $  is a central   subgroup of 
$C_{\Gb^*} (s_j)^{F^*} $.  Since  the centre of $C_{\Gb^*} (s_j)^{F^*} $ 
is contained in the  kernel of   any unipotent character of 
$C_{\Gb^*} (s_j)^{F^*} $ it follows that the product of the orders of $v_1$ and 
$v_2$ is at most $8$.

First suppose that both $C_{\Xb}(v_1)$  and $C_{\Yb}(v_2)$ are tori.
Then, both $v_1$ and $v_2 $ are non-trivial. On the other hand,   
the product of the orders of $v_1$ and $v_2$ is at most $8$, hence at least one 
of $v_1$  and $v_2$  has order $2$. But the centraliser  
of  an element of order $2$ in a simple algebraic group of type $A_2 $ 
or $E_6 $ is not a torus (see \cite[Table 4.3.1]{GorLyoSol3}), 
a contradiction.

Suppose that neither $C_{\Xb}(v_1)$ nor $C_{\Yb}(v_2 )$ is a torus.  
We have 
$$\bar \Cb_j =  C_{\Xb}(v_1)/ Z(C_{\Xb}(v_1)) \times 
C_{\Yb}(v_2)/ Z(C_{\Yb}(v_2)). $$
Hence,  $\bar \lambda_j=  \phi_1 \times \phi_2 $,  where 
$\phi_1$ is a unipotent character of  $C_{\Xb}(v_1)/ Z(C_{\Xb}(v_1))^{F^*}$ 
and
$\phi_2$ is a unipotent character of  $C_{\Yb}(v_2)/ Z(C_{\Yb}(v_2))^{F^*}$, 
such that the sum of $\zeta_j $ and  the $2$-defects of $\phi_1 $ and 
$\phi_2$ equals $3$.
Since the only subdiagrams of  
the extended Dynkin diagram of type $A$ are of type  $A$, it follows  from 
\cite{Deri83} and   Propositions \ref{2defectsclassical} and 
\ref{2defectsexceptional}
that either $\Delta_j $  is  
$A_1 \times E_6 $,   and  $\zeta_j$ has odd order  or 
$\Delta_j $ is 
$A_2 \times E_6 $   and  $\zeta_j$   
has  even order. But by   \cite{Deri83}, there are no such $s_j$. 

Thus, exactly one of 
$C_{\Xb}(v_1)$  or  $C_{\Yb}(v_2)$ is a torus. 
Then $\zeta_j >0$ ($v_i$ is a $2$-element) 
so by Equation (\ref{E_8def}), the $2$-defect of $\bar\lambda_j $ 
is  at most $2$.
From Propositions \ref {2defectsclassical} 
and \ref{2defectsexceptional}, it follows that   
either $\Delta_{j}$ is a product of copies of $A_2$
transitively permuted by $F^*$  or 
$\Delta_{j} $ is of type $E_6$. 

If $\Delta_{j}$ is a product of copies of $A_2$ 
transitively permuted by $F^*$, then by the tables in \cite{Deri83}, 
one sees that either $ \zeta_j=0$ or $\zeta_j \geq 2 $, whence by  
Equation (\ref{E_8def}), $\bar\lambda_j $ 
has $2$-defect $3$ or at most $1$.  This is a contradiction as
any unipotent character of $A_2(q^m)$ or $ \,^2A_2(q^m)$ 
has $2$-defect $2$ or greater than $3$.

If $\Delta_{j}$ is of type $E_6 $, then   
$\bar\lambda_j$ has $2$-defect $0$, 
whence the Sylow $2$-subgroups of
$\Zb_j^{F^*}$ have order $8$.  
Since $\Delta_{j} $   is a subdiagram  of the extended diagram associated to 
$A_2 \times E_ 6 $, it follows that $C_{\Xb}(v_1)$ is  a torus and  
that the Sylow $2$-subgroups  of  $C_{\Xb}(v_1)^{F^*}$ have order $8$.
From  \cite{Deri83} (or from the description of $F^*$-stable  maximal tori  in 
type $A_2$), 
it follows that $|C_{\Xb}(v_1)^{F^*}| $ is one  of $(q^2-1)$, 
$ (q \pm 1)^2$, or 
$(q^2 \pm q +1)$.    Thus the only possibility is  that   
$|C_{\Xb}(v_1)^{F^*}| = (q^2-1) $ and  $8$ is the highest power of $
2$ dividing $q^2-1 $. 

First suppose that $\Cb ^{F^*} $ 
is of type $E_6(q)$ or $A_2(q)$. Then  
${\bar \Cb_j}^{F^*} $ is of type  $E_6(q) $  for all $j$, 
$2\leq j \leq  5 $. By the 
formula for the character degrees of unipotent characters of  
$2$-defect $0$ of $E_6(q)$, we get that  for  all $j$, $2 \leq j \leq 5 $,
\begin{equation} \chi_{j}(1) = 
\frac{|G|_{r'} }{3 |E_6(q)|(q^2-1) } q^{7} \Phi_1^6 \Phi_2^4 \Phi_4^2 
\Phi_5 \Phi_8  = \frac{|G|_{r'} }{3 |E_6(q)| } q^{7} \Phi_1^5 \Phi_2^3 
\Phi_4^2 \Phi_5 \Phi_8. \end{equation}
Here, as before, we use $\Phi_d $ to denote $\Phi_d(q)$.

The unipotent character of defect $2$ of  $A_2(q)$ is of degree $q(q+1)$  
whence 
\begin{equation} \chi_{1}(1) = 
\frac{ |G|_{r'} } { 3 |E_6(q)||A_2(q)|}(q+1)q^8 \Phi_1^6 \Phi_2^4 \Phi_4^2 
\Phi_5 \Phi_8 = \frac{ |G|_{r'} } { 3 |E_6(q)|} \frac{q^8 \Phi_1^4 \Phi_2^4 
\Phi_4^2 \Phi_5 \Phi_8}{\Phi_3^2}. 
\end{equation}

Consider the element $ 2\chi_1  - \sum_{2\leq j\leq 5}\delta_j\chi_j $ of 
$L^0(G,b)$ as in Proposition \ref{punchheight}(ii).
By the equation above,   
all $\chi_j $, $2 \leq j\leq 5 $ have  the same degree. 
Hence, by Proposition \ref{punchheight}(iii), 
$\delta_j=\delta_i $ for all $i,j $ such that 
$2\leq i,j\leq 5 $. Thus,
$$  2\chi_1(1)  - \sum_{2\leq j\leq 5}\delta_j\chi_j (1)=  0$$   
implies that 
$$ q\Phi_2 - 2\delta_2\Phi_1 \Phi_3 ^2 =0,$$ 
but this is impossible since $p_3 $ does not divide $\Phi_2(q) $.

The case that  $\Cb ^{F^*}$ 
is of type $E_6(q)A_2(q)$ is similar with $p_3$ replaced by $p_6 $.
\end{proof}

\section{{Proof of Theorem \ref{c2c2c2}}} \label{proofsection}

\begin{proof}[{Proof of Theorem \ref{c2c2c2}}]
By Theorem \ref{AlperinimpliesBroue}, and using its notation, it 
suffices to show that $|\Irr_K(G,b)|=$ $8$.
Arguing inductively, in order to prove Theorem \ref{c2c2c2}
we may assume, by Theorem \ref{c2c2c2quasisimple},
that $G$ is a quasi-simple finite group with a centre of odd
order. We also may assume, by \cite[Theorem 3.7]{Landrock81},
that $b$ is a non-principal block.
If $G/Z(G)$ is a sporadic simple group then by
Proposition \ref{Noeske} we have $G\cong$ $Co_3$ and by 
Proposition \ref{Sporadic}, we have $|\Irr_K(G,b)|=$ $8$.
By the results in \S\S \ref{excmultsection} and \ref{altsection}, 
$G/Z(G)$ is neither a finite simple group
with an exceptional Schur multiplier nor an alternating group.
If $G/Z(G)$ is a finite group of Lie type in characteristic
$2$ then, by Proposition \ref{Lietwo}, 
we have $G\cong$ $\PSL_2(8)$ and
$b$ is the principal block; in particular, we have again
$|\Irr_K(G,b)|=$ $8$. Let $q$ be an odd prime power and
$n$ a positive integer. By Theorems \ref{typeA} and \ref{type2A},
the group $G/Z(G)$ cannot be isomorphic to
$\PSL_n(q)$ or $\PSU_n(q)$; alternatively, 
$|\Irr_K(G,b)|=$ $8$ holds in these cases as a consequence of 
\cite{BlEl}.  
If $G/Z(G)$ is isomorphic to one of $\PSp_{2n}(q)$, $(n\geq 2)$,
$\POmega_{2n+1}(q)$, $(n\geq 3)$, or $\POmega^{\pm}_{2n}(q)$,
$(n\geq 4)$ then, by Theorem \ref{classical},
the block $b$ is nilpotent; in particular, $|\Irr_K(G,b)|=$ $8$.
By Proposition \ref{typeG2}, 
$G/Z(G)$ cannot be isomorphic to a simple
group of type $G_2(q)$.
Since $b$ is assumed to be non-principal, 
$G/Z(G)$ cannot be isomorphic to a simple
group of type ${^2G_2(q)}$, and
by Proposition \ref{type3D4}, $G/Z(G)$ cannot be isomorphic to a simple
group of type ${^3D_4(q)}$.
If $G/Z(G)$ is isomorphic to one of the remaining exceptional
simple groups $F_4(q)$, $E_6(q)$, ${^2E_6(q)}$, $E_7(q)$ or
$E_8(q)$ then $b$ is Morita equivalent to a block of
a finite group $L$ such that $|L/Z(L)|$ is smaller than
$|G/Z(G)|$, by Propositions \ref{typeFpunch},
\ref{typeE_6punch}, \ref{type2E_6punch}, \ref{typeE_7punch} 
and \ref{typeE_8punch},
respectively. Theorem \ref{c2c2c2} follows inductively.
\end{proof}

\begin{Remark}
We do not know whether there are actually any blocks with an
elementary abelian defect group of order $8$ if $G/Z(G)$
is of one of the exceptional types $F$ or $E$. If not, one could
avoid the rather tedious calculations from  
\S \ref{smalldefectsection} onwards.
\end{Remark}

\section{Appendix}

We provide a proof for a result announced by Rouquier in \cite{Rouqb}.
The notation is as in \cite[Appendix]{Linpperm}.  
Let $\CO$ be a complete local commutative Noetherian ring having an
algebraically closed residue field $k$ of characteristic $2$; we
allow the case $\CO=$ $k$.
For $A$, $B$ two symmetric $\CO$-algebras, a bounded complex $X$
of $A$-$B$-bimodules which are projective as left $A$-modules
and as right $B$-modules is said to {\it induce a stable equivalence}
if there are isomorphisms of complexes of bimodules 
$X\tenB X^* \cong A \oplus Y$ and $X^*\tenA X\cong B\oplus Z$
with $Y$ and $Z$ homotopy equivalent to bounded complexes of projective
$A$-$A$-bimodules and $B$-$B$-bimodules, respectively.
If $Y$ and $Z$ are homotopic to zero then $X$ is called a
{\it Rickard complex}.

\begin{Theorem}[{cf. \cite[Theorem 6.10]{Rouqb}}] 
\label{c2c2c2stableequivalence}
Let $G$ be a finite group, let $b$ be a block of $\OG$ with an
elementary abelian defect group of order $8$, set $H = N_G(P)$ 
and denote by $c$ the block of $\OH$ satisfying $\Br_{\Delta P}(b) =$
$\Br_{\Delta P}(c)$. 
Let $i\in (\OGb)^{\Delta P}$ and $j\in (\OHc)^{\Delta P}$ be source
idempotents such that $\Br_{\Delta P}(i) = \Br_{\Delta P}(j)$.
There is a bounded complex of $\OGb$-$\OHc$-bimodules whose components
are finite direct sums of summands of the bimodules $\OG i\tenOQ j\OH$,
with $Q$ running over the subgroups of $P$, such that $X$ induces
a stable equivalence.
\end{Theorem}

\begin{proof}
The proof follows the lines of \cite[6.3]{Rouqb}. For any subgroup $Q$
of $P$ denote by $e_Q$ and $f_Q$ the unique blocks of $kC_G(Q)$ and
$kC_H(Q)$, respectively, satisfying $\Br_{\Delta Q}(i)e_Q\neq 0$
and $\Br_{\Delta Q}(j)f_Q\neq 0$. 
Since $P$ is abelian, the fusion systems on $P$ determined by $i$ and by
$j$ are equal to that of $kN(P,e_P)e_P$. Hence, for any subgroup $Q$ of 
$P$, we have
$$N_G(Q,e_Q)/C_G(Q) \cong N_H(Q,f_Q)/C_H(Q)$$
and both sides have odd order (either $1$ or $3$ in case $Q$ 
is a proper subgroup of $P$). 
The blocks $e_Q$, $f_Q$ lift to unique blocks $\hat e_Q$, $\hat f_Q$ of
$\CO C_G(Q)$, $\CO C_H(Q)$, respectively. The images of $\hat e_Q$, $\hat f_Q$
in $\CO C_G(Q)/Q$, $\CO C_H(Q)/Q$ are blocks, denoted by $\bar e_Q$, $\bar f_Q$,
respectively. Suppose now that $Q$ has order $2$. Then $\bar e_Q$, $\bar f_Q$ 
have the Klein four group $P/Q$ as defect group, and $C_H(Q)/Q$ is the normaliser
in $C_G(Q)/Q$ of $P/Q$. Thus $\bar f_Q$ is in fact the Brauer correspondent
of $\bar e_Q$. By \cite[Theorem 1.1]{CEKL}, the source algebras of blocks 
with a Klein four defect group $V_4$ are either $\CO V_4$, or $\CO A_4$, 
or $\CO A_5 b_0$, where $b_0$ is the principal block of $\CO A_5$. By 
\cite[\S 3]{Ricksplendid}, there is an explicitly described two-term
splendid Rickard complex between $\CO A_4$ and $\CO A_5b_0$.
Thus there is a Rickard complex
of $\CO C_G(Q)/Q\bar e_Q$-$\CO C_H(Q)/Q\bar f_Q$-bimodules $\bar C_Q$ of the
form
$$\bar C_Q = \xymatrix{\cdots\ar[r] & 0\ar[r] 
& \bar N_Q \ar[rr]^{\bar\Phi_Q} & & \bar e_Q C_G(Q)/Q \bar f_Q\ar[r] 
& 0 \ar[r] & \cdots }$$
for some (possibly zero) projective  bimodule $\bar N_Q$. By
\cite[10.2.11]{Rouqa}, this complex lifts to a Rickard complex of 
$\CO C_G(Q)\hat e_Q$-$\CO C_H(Q)\hat f_Q$-bimodules of the form
$$C_Q = \xymatrix{\cdots\ar[r] & 0\ar[r] 
& N_Q \ar[rr]^{\Phi_Q} & &  \hat e_Q C_G(Q) \hat f_Q\ar[r] 
& 0 \ar[r] & \cdots }$$
where $N_Q$ is a projective $\CO(C_G(Q)\times \CO C_H(Q))/\Delta Q$-module
lifting $\bar N_Q$, inflated to $\CO(C_G(Q)\times C_H(Q))$, and
where $\Phi_Q$ lifts the map $\bar\Phi_Q$. By adapting arguments of Marcu\c{s} 
\cite[5.5]{Marcus} this complex extends to the group
$$T = N_{G\times H}(\Delta Q)\cap (N_G(Q,e_Q)\times N_H(Q,f_Q))\ $$
and $T$ contains $C_G(Q)\times C_H(Q)$ as normal subgroup of odd 
index at most $3$ by the above remarks. One can see this also directly: 
the modules of the complex $C_Q$ are clearly $T$-stable, hence 
extend to $T$ (cf. \cite[10.2.13]{Rouqa}), and the map $\bar \Phi$ lifts 
to a $T$-stable map $\Psi$ because $N_Q$ remains projective 
when considered as $\CO T/\Delta Q$-module. We set
$$V_Q = \Ind_T^{G\times H}(N_Q)\ .$$
The inclusion $C_G(Q)\subseteq G$ induces an $\OT$-homomorphism
$e_Q\CO C_G(Q) f_Q \longrightarrow b\OG c$, which, by adjunction,
yields a homomorphism of $\CO(G\times H)$-modules
$$\alpha_Q : \Ind_T^{G\times H}(e_Q\CO C_G(Q) f_Q) \longrightarrow b\OG c\ .$$
Set $\psi_Q = \alpha_Q\circ\Ind^{G\times H}_T(\Phi_Q) : V_Q \rightarrow b\OG c$
and define the complex $X$ by
$$X = \xymatrix{ \cdots\ar[r]& 0 \ar[r]  & \oplus_Q V_Q
\ar[rr]^{\oplus_Q \Psi_Q} & & b\OG c \ar[r] & 0\ar[r] &\cdots }$$ 
with $b\OG c$ in degree zero, where $Q$ runs over a set of representatives
of the $N_G(P,e_P)$-conjugacy classes of subgroups of order $2$ of $P$.
One checks that if $Q$, $R$ are two subgroups of order $2$ which are not 
$N_G(P,e_P)$-conjugate then $V_R(\Delta Q) = \{0\}$. This implies that 
$e_Q X(\Delta Q) f_Q\ \simeq$ $C_Q\tenO k$, and this is a Rickard complex of 
$kC_G(Q)e_Q$-$C_H(Q)f_Q$-bimodules. Moreover, $b\OG c$ is a direct summand
of $\OG i\tenOP j\OH$, and $V_Q$ is a direct sum of summands of
$\OG i\tenOQ j\OH$ because it is obtained from lifting,
inflating and inducing a projective bimodule.
Thus another result of Rouquier (with a proof given in
\cite[Theorem A.1]{Linpperm}) applies, showing that $X$ induces a 
stable equivalence.
\end{proof}

\begin{Remark} 
We have used \cite[Theorem 1.1]{CEKL} for the
description of Rickard complexes for blocks with a Klein four
group, which requires the classification of finite simple groups.
One can avoid this by making use of another technique of Rickard,
replacing one of the endo-permutation sources of a simple
module by a $p$-permutation resolution (see e.g. 
\cite[Theorem 1.3]{Linturk}). The only difference this makes is
that the resulting complex may have more than two non-zero
components. 
One can be more precise in Theorem \ref{c2c2c2stableequivalence} 
if $E$ is either trivial or has order $7$.
If $E$ is trivial, $b$ is nilpotent, and its source algebra
is of the form $\End_\CO(V)\tenO \OP$ for some indecomposable
endo-permutation $\OP$-module $V$ with $P$ as vertex. Using
again \cite[Theorem 1.1]{CEKL} we get that  
$V\cong$ $\Omega_P^n(\CO)$ for some integer $n$.
Equivalently, if $E$ is trivial then $\OGb$ is Morita equivalent to
$\OP$ via an $\OGb$-$\OP$-bimodule which, when viewed as
$\CO(G\times P)$-module, has vertex $\Delta P=$ $\{(u,u)\ |\ u\in P\}$
and source $\Omega_{\Delta P}^n(\CO)$ for some integer $n$.
By applying $\Omega^{-n}_{\CO(G\times P)}$ to $M$
one gets a stable equivalence of Morita type given by a bimodule
with vertex $\Delta P$ and trivial source.
Similarly, if $E$ is cyclic of order $7$, then by a result of
Puig \cite{Puabelian}, there is a stable equivalence of Morita
type between $\OGb$ and $\OP$ given by a bimodule with vertex
$\Delta P$ and trivial source. 
\end{Remark}

\noindent
{\bf Acknowledgements.}
The authors would like to thank Felix Noeske for bringing
to  their  attention the work in \cite{Noeske} and Gunter Malle for
pointing the authors to the references \cite{BeKn}, \cite{BlEl} and 
for  his careful reading of parts of  the  manuscript. 
The authors would further like to thank the referee for many
valuable comments.
The second author  made several visits to the University of Aberdeen
in 2008 and 2009 during which parts of this work were carried out.
He is grateful to the Institute of Mathematics, University of
Aberdeen for its kind hospitality.
For this research the second author was partially
supported by the Japan Society for Promotion of Science (JSPS),
Grant-in-Aid for Scientific Research 
(C)20540008, 2008--2010.

\end{document}